\documentclass[reqno,12pt]{amsart}

\usepackage{amsmath,amsthm,amssymb,amsfonts, mathabx, tikz-cd}
\usepackage{commath}
\usepackage{extarrows}

\usepackage{tikz}
\usetikzlibrary{decorations.pathreplacing,backgrounds}
\usepackage{xcolor}

\usepackage{enumitem}
\usepackage{float}

\usepackage{titletoc}
\usepackage[titletoc]{appendix}

\usepackage{geometry}
\usepackage[format=hang]{caption}

\usepackage[hyperfootnotes=false,draft=false,pdftex, colorlinks=true, linkcolor=black, citecolor=black, urlcolor=black, hypertexnames=false]{hyperref}

\let\oldbibliography\thebibliography
\renewcommand{\thebibliography}[1]{%
	\oldbibliography{#1}%
	\setlength{\itemsep}{3pt}
}

\numberwithin{equation}{section}
\counterwithin{figure}{section}

\newcommand{\eq}{\,=\,}

\newcommand{\ms}[1]{\mathsf{#1}}
\newcommand{\vect}[1]{\vec{#1}\,}
\renewcommand{\norm}[1]{\lVert #1 \rVert}

\newcommand{\textbfit}[1]{\textbf{\textit{#1}}}
\newcommand\Simeq{\;\simeq\;}

\newcommand{\overbar}[1]{\mkern 1mu\overline{\mkern-1mu#1\mkern-1mu}\mkern 1mu}
\newcommand{\highoverline}[1]{\overbar{\vphantom{#1^{\rule{0pt}{0.3ex}}}#1}}

\tikzcdset{
	arrow style=tikz,
	diagrams={thick, >=stealth},
	every arrow/.append style={line width=0.8pt}
}

\newcommand{\LD}{\ms{LD}}

\newcommand{\uFrob}{\ms{Frob}}

\newcommand{\Free}{F^{\LD}_S}
\newcommand{\FreeF}{F^{\uFrob}_S}
\newcommand{\FreeC}{F^{\LD}_{{U(\cC)}}}
\newcommand{\FreeG}{F^{\uFrob}_{{U(G)}}}  

\newcommand{\rotamp}[2]{\mathbin{\rotatebox[origin=c]{180}{$#1\&$}}}
\newcommand{\parll}{{\mathpalette\rotamp\relax}}
\DeclareMathOperator{\partimes}{{\parll}}
\newcommand{\ot}{\otimes}

\newcommand{\tim}{\times}

\newcommand{\id}{{\operatorname{id}}}
\newcommand{\idC}{{\operatorname{id}_{\cC}}}

\newcommand{\dom}{{s}}
\newcommand{\codom}{{t}}
\newcommand{\Ob}{{\operatorname{Ob}}}

\newcommand{\ao}{\alpha}
\newcommand{\aoi}{\alpha^{-1}}

\newcommand{\ap}{\highoverline{\alpha}}
\newcommand{\api}{\highoverline{\alpha}^{-1}}

\newcommand{\distl}{\delta^l}
\newcommand{\distr}{\delta^r}

\newcommand{\ta}{\tau}
\newcommand{\ka}{\kappa}

\newcommand{\slongrightarrow}[1]{\,\longrightarrow{#1}\,}
\newcommand{\sxlongrightarrow}[1]{\,\xlongrightarrow{#1}\,}
\newcommand{\myrightleftarrows}[1]{\,\mathrel{\substack{\xrightarrow{#1} \\[-.5ex] \xleftarrow{#1}}}\,}

\newcommand{\fM}{\mathfrak{M}} 

\newcommand{\fS}{S}
\newcommand{\fY}{\mathfrak{Y}}
\newcommand{\fZ}{\mathfrak{Z}}
\newcommand{\fd}{\mathfrak{d}}
\newcommand{\fm}{\mathfrak{m}} 

\newcommand{\Neq}{\mathrel{\scalebox{1.1}{$\neq$}}}
\newcommand{\NNez}{\NN_{0}}

\newcommand{\e}{\emptyset}

\def\cC{\mathcal C}
\def\cD{\mathcal D}
\def\cE{\mathcal E}
\def\cF{\mathcal F}

\def\CC{\mathbb C}

\def\NN{\mathbb N}

\theoremstyle{plain}
\newtheorem{theorem}{Theorem}[section]
\newtheorem{lemma}[theorem]{Lemma}
\newtheorem{proposition}[theorem]{Proposition}              
\newtheorem{corollary}[theorem]{Corollary}        

\newtheorem*{cohthm}{Categorical Coherence Theorem}
\newtheorem*{funthm}{Functorial Coherence Theorem}
\newtheorem*{spider}{Spider Theorem}

\theoremstyle{definition} 
\newtheorem{definition}[theorem]{Definition}

\theoremstyle{remark}
\newtheorem{remark}[theorem]{Remark}
\newtheorem{example}[theorem]{Example}
\newtheorem{construction}[theorem]{Construction}

\newcounter{cmt}

\usepackage{scalerel}
\newcommand{\mrhup}{\scaleobj{0.5}{\rightharpoonup}}
\newcommand{\mlhup}{\scaleobj{0.5}{\leftharpoonup}}
\newcommand{\mrelb}{\hstretch{0.5}{\scaleobj{0.5}{\relbar}}}
\makeatletter
\def\rhfill@#1{$\m@th\thickmuskip0mu\medmuskip\thickmuskip\thinmuskip\thickmuskip
   \relax#1\mkern+2.5mu\cleaders\hbox{$#1\mrelb\mkern-1mu$}\hfill\mkern-6mu\mrhup\mkern+0.25mu$}
\def\lhfill@#1{$\m@th\thickmuskip0mu\medmuskip\thickmuskip\thinmuskip\thickmuskip
   \relax#1\mkern+1.5mu\mlhup \cleaders\hbox{$#1\mkern-2.5mu\mrelb$}\hfil\mkern+0.5mu$}
\DeclareRobustCommand{\overrightharpoon}{\mathpalette{\overarrow@\rhfill@}}
\DeclareRobustCommand{\overleftharpoon}{\mathpalette{\overarrow@\lhfill@}}
\makeatother

\let\vect=\overrightharpoon

\let\sm=\smallsetminus
\newcommand{\ff}{\mathfrak{f}}

\def\rmlabel{\upshape({\itshape \roman*\,})}

\def\alabel{\upshape({\itshape \alph*\,})}

\newcommand{\ter}{{\operatorname{ter}}}
\newcommand{\init}{{\operatorname{init}}}
\newcommand{\Idd}{{\operatorname{Id}}}

\let\prt=\partimes

\newcommand{\Rho}{\mathrm{P}}

\author{Max Demirdilek, Christian Reiher, Christoph Schweigert}

\title{Linearly distributive coherence in the absence of units}

\begin{document}

\begin{abstract}
	Coherence in a monoidal category asserts that all morphisms built from structural isomorphisms with a fixed source and target coincide. These structural isomorphisms include, in particular, the associators. Linearly distributive categories carry two tensor products, with structural morphisms given by associators and distributors relating the two tensor products. In several examples, including Grothendieck--Verdier categories, also known as \(\ast\)-autonomous categories, these distributors need not be invertible.
	
	We give a self-contained proof that linearly distributive categories {\em without units} are coherent, while units may obstruct coherence. With the same techniques, we also establish an analogous coherence result for Frobenius linearly distributive functors without units. These results admit a reformulation in terms of directed paths in associahedra and multiplihedra.
\end{abstract} 

\maketitle

\tableofcontents

\section{Introduction}\label{sec:Intro}

Mac Lane’s coherence theorem is a foundational result in the theory of monoidal categories.
To understand it, recall that associativity in a monoid implies that all ways of multiplying a finite string of elements give the same value. This statement can be represented by the following family of graphs: for each $n$, consider the graph whose vertices are binary bracketings of $n$ letters, with an unoriented edge for each application of associativity.
For $n=4$, it takes the following form:
\begin{figure}[H]
	\centering
	\begin{tikzpicture}[x=1.3mm,y=1mm,scale=1.31]
		\tikzset{every node/.style={scale=.8}}
		\tikzset{>=stealth}
		\tikzset{mm/.style={execute at begin node=$\displaystyle, execute at end node=$}}
		\draw[mm]
		(-11.5,-5) node (1120) {((a \cdot b)\cdot c)\cdot d}
		(11.5,-5) node (1210) {(a\cdot {(b \cdot c)})\cdot d}
		(0,20) node (3010) {a\cdot {(b \cdot (c\cdot d))}}
		(-16.5,9.5) node (2020) {(a\cdot b)\cdot(c \cdot d)}
		(16.5,9.5) node (2110) {a\cdot {((b\cdot c)\cdot d)}}
		;
		\draw[-,thick] (1210) -- (2110) node [midway,xshift=6ex,yshift=-0.5ex] {};
		\draw[-,thick]
		(3010) -- (2110) node [midway,yshift=1.5ex, xshift=6.5ex] {};
		\draw[-,thick]
		(1120) -- (1210) node [midway,yshift=-4ex] {};
		\draw[-,thick]
		(2020) -- (1120) node [midway,xshift=-6ex,yshift=0.5ex] {};
		\draw[-,thick]
		(2020) -- (3010) node [midway,xshift=-5ex,yshift=2ex] {}
		;
	\end{tikzpicture}
	\caption{The graph of multiplication for the string of elements \(a,b,c,d\) in an associative monoid.}
	\label{fig:ass}
\end{figure}
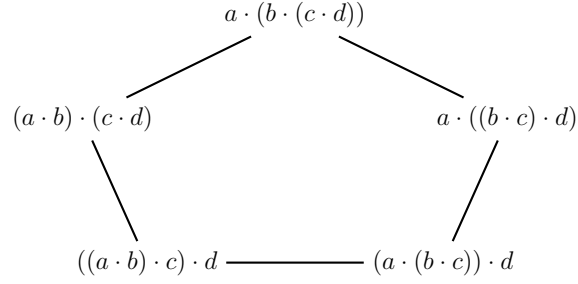

Each such graph is connected; hence finite products are independent of the bracketing.

\addtocontents{toc}{\protect\setcounter{tocdepth}{1}}

\subsection{Monoidal categories} 
A monoidal category $\cC$ comes equipped with a tensor product
\begin{equation}
	\ot:\: \cC\times \cC \slongrightarrow \cC\,.
\end{equation}
Associativity is not imposed as equality but via isomorphisms
\begin{equation}
	\alpha_{A,B,C}: \quad A\otimes (B\otimes C) \myrightleftarrows{\rule{.5cm}{0cm}} 
	(A\otimes B)\otimes C \quad : \alpha^{-1}_{A,B,C} \,,
\end{equation}
more precisely, a natural isomorphism, called the \emph{associator}.

When comparing two different bracketings of an unbracketed tensor product of objects, repeated use of the associator or its inverse may yield a priori different isomorphisms between them. For example, when \(n=4\), there are distinct ways to pass between any two bracketings:

\begin{figure}[H]
	\centering
	\begin{tikzpicture}[x=1.3mm,y=1mm,scale=1.53]
		\tikzset{every node/.style={scale=.8}}
		\tikzset{>=stealth}
		\tikzset{mm/.style={execute at begin node=$\displaystyle, execute at end node=$}}
		\draw[mm]
		(-11.5,-5) node (1120) {((A\ot B)\ot C)\ot D}
		(11.5,-5) node (1210) {(A\ot {(B \ot C)})\ot D}
		(0,20) node (3010) {A\ot {(B \ot (C\ot D))}}
		(-16.5,9.5) node (2020) {(A\ot B)\ot(C \ot D)}
		(16.5,9.5) node (2110) {A\ot {((B\ot C)\ot D)}}
		;
		\draw[transform canvas={yshift=0.3ex,xshift=-0.3ex},->,thick] (1210) -- (2110) node [midway,xshift=-6ex,yshift=0.5ex] {$\alpha^{-1}_{A,B\,\ot\,C,D}$};
		\draw[transform canvas={yshift=-0.3ex,xshift=0.3ex},<-,thick] (1210) -- (2110) node [midway,xshift=6ex,yshift=-0.5ex] {$\alpha_{A,B\,\ot\,C,D}$};
		\draw[transform canvas={yshift=0.4ex,xshift=0.3ex},->,thick]
		(3010) -- (2110) node [midway,yshift=2ex, xshift=5.5ex] {$A\ot \ao_{B,C,D}$};
		\draw[transform canvas={yshift=-0.2ex,xshift=-0.3ex},->,thick]
		(2110) -- (3010) node [midway,yshift=-2ex, xshift=-4.5ex, shorten <=1cm] {$A\ot \aoi_{B,C,D}$};
		\draw[transform canvas={yshift=0.4ex},->,thick]
		(1120) -- (1210) node [midway,yshift=2.5ex] {$\aoi_{A,B,C}\ot D$};
		\draw[transform canvas={yshift=-0.4ex},->,thick]
		(1210) -- (1120) node [midway,yshift=-2.5ex] {$\ao_{A,B,C}\ot D$};
		\draw[transform canvas={yshift=0.2ex,xshift=0.4ex},->,thick]
		(2020) -- (1120) node [midway,xshift=5.5ex,yshift=0.4ex] {$\ao_{A \,\ot\, B,C,D}$};
		\draw[transform canvas={yshift=-0.2ex,xshift=-0.4ex},->,thick]
		(1120) -- (2020) node [midway,xshift=-5.5ex,yshift=-0.4ex] {$\aoi_{A \,\ot\, B,C,D}$};
		\draw[transform canvas={yshift=0.2ex,xshift=-0.4ex},->,thick]
		(2020) -- (3010) node [midway,xshift=-4ex,yshift=2ex] {$\aoi_{A,B,C \,\ot\, D}$};
		\draw[transform canvas={yshift=-0.2ex,xshift=0.4ex},->,thick]
		(3010) -- (2020) node [midway,xshift=4ex,yshift=-2ex] {$\ao_{A,B,C \,\ot\, D}$};
	\end{tikzpicture}
	\caption{The pentagon diagram for a string of objects \(A,B,C,D\)
		in a monoidal category.}
\end{figure}
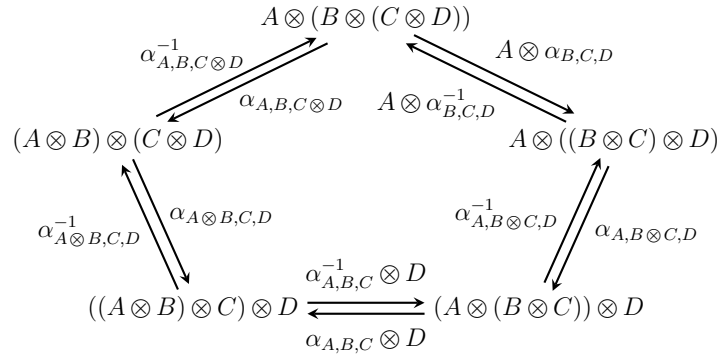

In a monoidal category, one requires that, for any string of four objects, this diagram commutes in the following sense: the composites of morphisms corresponding to any two parallel directed paths coincide. Two paths are parallel if they have the same initial and terminal vertices. It suffices to impose this condition for a single pair of such paths.

\addtocontents{toc}{\protect\setcounter{tocdepth}{1}}
\subsection{Monoidal coherence}

Analogously to Figure~\ref{fig:ass}, where associativity identities were recorded by adding \(1\)-cells, one now keeps track of pentagonal identities by adding \(2\)-cells. Proceeding in this way, one obtains, for each $n\geq 4$, an \((n-2)\)-dimensional polytope known as the \emph{associahedron} \cite{Ta51,St63,Ha84}. Its vertices correspond to bracketings of \(n\) objects, its edges represent applications of associators and their inverses, and its \(2\)-faces record the pentagon identities and the naturality of the associator. The latter \(2\)-faces have a quadrilateral shape. Figure \ref{fig:associahedron} shows the associahedron for $n=5$.

Mac Lane's coherence theorem asserts that, for any \(n\), the composites of structural isomorphisms corresponding to any two parallel directed paths in the \(1\)-skeleton of the \(n^{th}\)~associahedron coincide, provided the pentagon and naturality axioms hold. In more algebraic terms, it states that the associativity constraints, although isomorphisms rather than identities, introduce no ambiguity. Very informally, all diagrams commute.

\sloppy The coherence theorem is closely related to strictification, the statement that every monoidal category is monoidally equivalent to one in which the associators are identities. A degree of strictness underlies many graphical calculi (e.g. \cite{JoSt91}) and, consequently, skein-theoretic constructions in quantum topology.

We have emphasized the relation of the coherence theorem to polytopes; cf., e.g.,  \cite{CuLaA24}. It is further deeply connected to the geometry of moduli spaces of stable punctured curves \cite{Ka93} and to Thompson's group \(F\) \cite{Deh05}.

\addtocontents{toc}{\protect\setcounter{tocdepth}{1}}

\subsection{Linearly distributive categories}
In many applications, one encounters categories equipped with two tensor products \(\ot\) and \(\partimes\). 

In linear logic, the tensor products correspond to the multiplicative connectives `and' and `or' (e.g. \cite{CS97}). Functional-analytic categories, notably the category of Banach spaces with contractions, can be equipped with the projective and injective tensor products (e.g. \cite{BlZ20}). Further examples are \emph{Grothendieck\kern0.1em--Verdier categories}, also known as \emph{\(\ast\)-autonomous categories}, arising as derived categories of sheaf-theoretic origin \cite{BoD13}, representation categories of vertex operator algebras \cite{ALSW25}, and modules over Hopf algebroids \cite{All23, De26}.

In these settings, the two tensor products are related by two mixed associativity constraints
\begin{equation}\begin{array}{lll}
		\mathcolor{purple}{\distl_{A,B,C}}:\: A\ot(B \prt C) & \color{purple}{\longrightarrow} &{(A \ot B)} \prt {C}\,,\\[0.3em]
		\mathcolor{purple}{\distr_{A,B,C}}:\: {(A \prt B)} \ot C & \mathcolor{purple}{\longrightarrow} &A \prt{(B \ot C)}\,,
	\end{array} \label{eq:distributors}
\end{equation}
more precisely, natural transformations, called \emph{distributors}. These avoid the duplication of objects familiar from the classical distributive law $a\cdot(b+c)=a\cdot b + a\cdot c$. The cost is that, in practice (see, e.g., Example \ref{ex: Abimod2} below), they are often non-invertible: they are directed.

Still, distributors must satisfy coherence conditions. A collection of such pentagon axioms was proposed in \cite{CS97}; the resulting structure is a \emph{linearly distributive category}. We display one such pentagon and refer to Definition \ref{def:laxLD} for the full list of eight pentagons:

\begin{figure}[H]
	\centering
	\begin{tikzpicture}[x=1.3mm,y=1mm,scale=1.53]
		\tikzset{every node/.style={scale=.8}}
		\tikzset{>=stealth}
		\tikzset{mm/.style={execute at begin node=$\displaystyle, execute at end node=$}}
		\draw[mm]
		(-11.5,-5) node (1120) {((A\prt B)\ot C)\prt D}
		(11.5,-5) node (1210) {(A\prt {(B \ot C)})\prt D}
		(0,20) node (3010) {A\prt {(B \ot (C\prt D))}}
		(-16.5,9.5) node (2020) {(A\prt B)\ot(C \prt D)}
		(16.5,9.5) node (2110) {A\prt {((B\ot C)\prt D)}}
		;
		\draw[transform canvas={yshift=0.3ex,xshift=-0.3ex},->,thick] (1210) -- (2110) node [midway,xshift=-6ex,yshift=0.5ex] {$\api_{A,B\,\ot\,C,D}$};
		\draw[transform canvas={yshift=-0.3ex,xshift=0.3ex},<-,thick] (1210) -- (2110) node [midway,xshift=6ex,yshift=-0.5ex] {$\ap_{A,B\,\ot\,C,D}$};
		\draw[->,color=purple,thick]
		(3010) -- (2110) node [midway,yshift=1.5ex, xshift=6.5ex] {$A\prt \distl_{B,C,D}$};
		\draw[->,color=purple,thick]
		(1120) -- (1210) node [midway,yshift=-4ex] {$\distr_{A,B,C}\prt D$};
		\draw[->,color=purple,thick]
		(2020) -- (1120) node [midway,xshift=-6ex,yshift=0.5ex] {$\distl_{A\prt B,C,D}$};
		\draw[->,color=purple,thick]
		(2020) -- (3010) node [midway,xshift=-5ex,yshift=2ex] {$\distr_{A,B,C\prt D}$}
		;
	\end{tikzpicture}
	\caption{A pentagon involving distributors; see Diagram~\eqref{eq:P6}.}.
	\label{fig:pentdis}
\end{figure}
Arrows in Figure \ref{fig:pentdis} whose labels involve distributors are directed. Isomorphisms and their inverses are drawn as double arrows. The pentagon axiom requires that morphisms corresponding to different parallel {\em directed} paths are identical: they can be transformed by a polygonal flip.

\addtocontents{toc}{\protect\setcounter{tocdepth}{1}}

\subsection{Coherence for categories}
For linearly distributive categories, associahedra such as the one in
Figure \ref{fig:associahedron} have two types of edges: some edges are
directed. Coherence can then be stated as follows: for each $n$ and each unbracketed strings of $n$ objects separated by either
$\ot$ or $\prt$, any two parallel directed paths in one of the \(2^{n-1}\) directed associahedra must yield the same morphism (which is, in general, not an isomorphism). Put differently, any such paths are related by polygonal flips.

The study of coherence in linearly distributive categories therefore involves both invertible and non-invertible morphisms, making it considerably more difficult. In this article, we develop elementary techniques for handling such directed arrows.

\begin{figure}[H]
	\centering
	\begin{tikzpicture}[x=1.3mm,y=0.97mm,scale=1.6]
		\tikzset{every node/.style={scale=.8}}
		\tikzset{>=stealth}
		\tikzset{mm/.style={execute at begin node=$\displaystyle, execute at end node=$}}
		\draw[mm] 
		(-11.5,-5) node (1120) {((A(BC))D)E}
		(11.5,-5) node (1210) {(A((BC)D))E}
		(0,20) node (3010) {A((BC)(DE))}
		(0,-24) node (1201) {((AB)(CD))E}
		(-22.5,-2) node (2011) {((AB)C)(DE)}
		(22.5,-2) node (2200) {A((B(CD))E)}
		(-16.5,9.5) node (2020) {(A(BC))(DE)}
		(16.5,9.5) node (2110) {A(((BC)D)E)}
		(-17.5,-16.5) node (1111) {(((AB)C)D)E}
		(17.5,-16.5) node (1300) {(A(B(CD)))E}
		;
		\draw[mm,color=black!50!white]
		(0,13) node (4000) {A(B(C(DE)))}
		(0,-11) node (2101) {(AB)((CD)E)}
		(-6,2) node (3001) {(AB)(C(DE))}
		(6,2) node (3100) {A(B((CD)E))}
		;
		
		\draw[transform canvas={yshift=0.4ex,xshift=0.3ex},->,thick, dashed, color=black!50!white] (2011) -- (3001);
		\draw[transform canvas={yshift=-0.4ex,xshift=0.3ex},<-,thick, dashed, color=black!50!white] (2011) -- (3001);
		\draw[transform canvas={yshift=-0.3ex,xshift=0.4ex},->,thick, dashed, color=black!50!white] (3010) -- (4000);
		\draw[transform canvas={xshift=-0.4ex},<-,thick, dashed, color=black!50!white] (3010) -- (4000);
		
		\draw[transform canvas={yshift=0.3ex,xshift=-0.3ex},->,thick] (1210) -- (2110);
		\draw[transform canvas={yshift=-0.3ex,xshift=0.3ex},<-,thick] (1210) -- (2110);
		\draw[transform canvas={yshift=0.3ex,xshift=0.3ex},->,thick] (1210) -- (1300);
		\draw[transform canvas={yshift=-0.3ex,xshift=-0.3ex},<-,thick] (1210) -- (1300);
		\draw[transform canvas={yshift=0.6ex,xshift=0.1ex},->,thick] (1111) -- (1201);
		\draw[transform canvas={yshift=-0.3ex,xshift=0.1ex},<-,thick] (1111) -- (1201);
		\draw[transform canvas={yshift=0.3ex,xshift=-0.3ex},->,thick] (1300) -- (2200);
		\draw[transform canvas={yshift=-0.3ex,xshift=0.3ex},<-,thick] (1300) -- (2200);
		\draw[transform canvas={yshift=0.3ex,xshift=0.3ex},->,thick] (2110) -- (2200);
		\draw[transform canvas={yshift=-0.3ex,xshift=-0.3ex},<-,thick] (2110) -- (2200);
	
		\path[->,color=purple,thick] 
		(3010) edge (2110)
		(1120) edge (1210)
		(2020) edge (1120)
		(2020) edge (3010)
		(1111) edge (1120)
		(1201) edge (1300)
		(2011) edge (1111)
		(2011) edge (2020)
		;
		\begin{scope}[on background layer]
			\path[->, dashed, color=purple!50!white,thick] 
			(3100) edge (2200)
			(2101) edge (1201)
			(3001) edge (2101)
			(3001) edge (4000)
			(2101) edge (3100)
			(4000) edge (3100)
			;
		\end{scope}
	\end{tikzpicture}
	\caption{The directed associahedron for the unbracketed expression \(A\prt B \ot C \ot D \prt E\).}
	\label{fig:associahedron}
\end{figure}
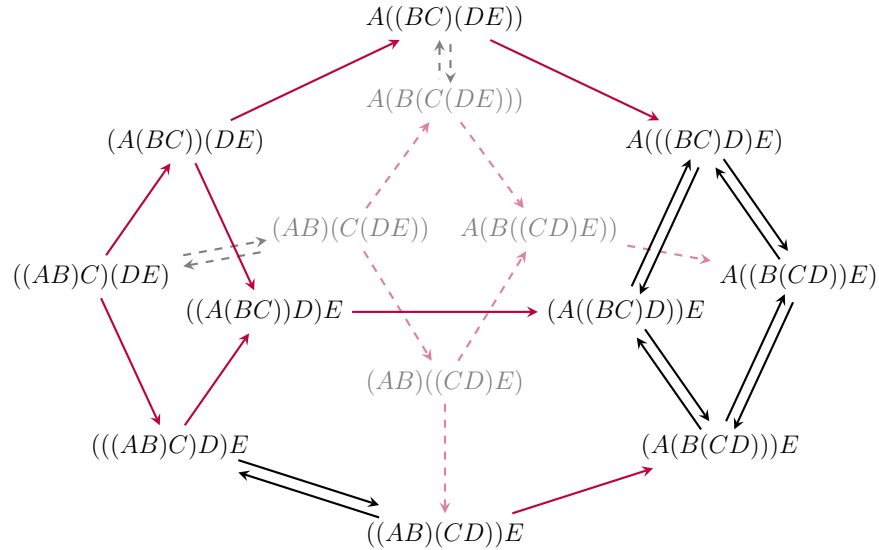

Our original motivation for studying the coherence of linearly distributive categories and their functors comes from the surface-diagrammatic calculus developed in \cite{DeS25}. The coherence theorems simplify computations in this calculus. For example, one of our main results (Theorem \ref{thm: coherence}) implies that any two surface diagrams with the same top and bottom faces, built only from associators and distributors, are equal. Consequently, consecutive applications of pentagon axioms within a cube can be combined into a single move; for instance, Equations (IX) and (X) in the proof of \cite[Lem.~3.25]{DeS25} now reduce to one step.

\subsection{An instructive example}
We note that the rôles of the two tensor products in the
distributors \eqref{eq:distributors} are not symmetric: in both cases, the target has the form $A\prt B$, i.e., it is a $\partimes$-product. This asymmetry makes the treatment of $\partimes$-products in this paper more involved. In the following example, the differing rôles of the two tensor products
arise naturally:

\begin{example}\label{ex: Abimod2}
	Let \(R\) be a finite-dimensional \(k\)-algebra. The category of \(R\)-bimodules carries a linearly distributive structure: the tensor product \(\ot\) is the balanced tensor product \(\otimes_R\). It is defined as a quotient, a colimit, and thus satisfies a universal property for morphisms {\em out of} it. Any $R$-bimodule is also a bicomodule over the coalgebra $R^*=\operatorname{Hom}_{k}(R,k)$. The second tensor 
	product \(\prt\) is the cotensor product  \(\ot^R\) of bicomodules. It is defined as
	a subspace, a limit, and is thus characterized by a universal property for morphisms {\em into} it. Using these universal properties, one naturally obtains distributors from \(\ot\)\kern0.1em-products into $\prt$-products. Since limits and colimits need not commute, these distributors are typically not invertible.
	For the case of unital \(R\)-bimodules over unital \(R\), see \cite[Ex.~2.10]{DeS25}.
\end{example}

\subsection{Units break coherence}\label{subsec:unitsbreak}
We have not mentioned monoidal units for the two tensor products. Unlike in the monoidal setting, full coherence fails for general linearly distributive categories with units. A counterexample, using Example \ref{ex: Abimod2}, appears in Example \ref{ex:counterex coherence 2}. In linear logic, units are well-known to create complications (e.g. \cite{Gi87,LiWi94,Hou07}). Units are also known to obstruct coherence and strictification in higher categories; as illustrated, for example, by \emph{Simpson's conjecture} \cite{Sim98}.

\subsection{Frobenius linearly distributive functors} 
A second coherence result, as foundational for monoidal categories as Mac Lane's, concerns functors. A strong monoidal functor between monoidal categories $\cC$ and $\cD$ is a functor $F:\cC\to \cD$ equipped with isomorphisms
\begin{equation} 
	\mu_{A,B}: \quad F(A)\ot F(B) \myrightleftarrows{\rule{.5cm}{0cm}}  F(A\ot B) \quad :\mu^{-1}_{A,B} \,,
\end{equation}
satisfying the following hexagonal coherence condition:

\begin{figure}[H]
	\centering
	\begin{tikzpicture}[x=1.3mm,y=1mm,scale=1.4]
		\tikzset{every node/.style={scale=.8}}
		\tikzset{>=stealth}
		\tikzset{mm/.style={execute at begin node=$\displaystyle, execute at end node=$}}
		\draw[mm]
		(-15,-5) node (1120) {F\bigl(A\ot (B\ot C)\bigr)}
		(15,-5) node (1210) {F\bigl((A\ot B)\ot C\bigr)}
		(-15,24) node (3010) {F(A)\ot \bigl(F(B)\ot F(C)\bigr)}
		(-24,9.5) node (2020) {F(A)\ot F(B\ot C)}
		(24,9.5) node (2110) {F(A\ot B)\ot F(C)}
		(15,24) node (h) {\bigl(F(A)\ot F(B)\bigr)\ot F(C)}
		;
		\draw[transform canvas={yshift=0.3ex,xshift=-0.3ex},->,thick] (1210) -- (2110) node [midway,xshift=-6ex,yshift=0.5ex] {$\mu^{-1}_{A \,\ot\, B,C}$};
		\draw[transform canvas={yshift=-0.3ex,xshift=0.3ex},<-,thick] (1210) -- (2110) node [midway,xshift=6ex,yshift=-0.5ex] {$\mu_{A \,\ot\, B,C}$};
		\draw[transform canvas={yshift=0.4ex},->,thick]
		(3010) -- (h) node [midway,yshift=3ex,thick] {$\ao_{F(A),F(B),F(C)}$};
		\draw[transform canvas={yshift=-0.4ex},->,thick]
		(h) -- (3010) node [midway,yshift=-3ex] {$\aoi_{F(A),F(B),F(C)}$};
		\draw[transform canvas={yshift=0.4ex},->,thick]
		(1120) -- (1210) node [midway,yshift=3ex] {$F(\ao_{A,B,C})$};
		\draw[transform canvas={yshift=-0.4ex},->,thick]
		(1210) -- (1120) node [midway,yshift=-3ex] {$F(\aoi_{A,B,C})$};
		\draw[transform canvas={yshift=0.2ex,xshift=0.4ex},->,thick]
		(2020) -- (1120) node [midway,xshift=5.5ex,yshift=0.4ex] {$\mu_{A,B \,\ot\, C}$};
		\draw[transform canvas={yshift=-0.2ex,xshift=-0.4ex},->,thick]
		(1120) -- (2020) node [midway,xshift=-5.5ex,yshift=-0.4ex] {$\mu^{-1}_{A,B \,\ot\, C}$};
		\draw[transform canvas={yshift=0.2ex,xshift=-0.4ex},->,thick]
		(2020) -- (3010) node [midway,xshift=-8ex,yshift=1ex] {$F(A) \ot \mu^{-1}_{B,C}$};
		\draw[transform canvas={yshift=-0.2ex,xshift=0.4ex},->,thick]
		(3010) -- (2020) node [midway,xshift=8ex,yshift=-1ex] {$F(A) \ot \mu_{B,C}$};
		\draw[transform canvas={yshift=0.2ex,xshift=0.4ex},->,thick]
		(h) -- (2110) node [midway,xshift=8ex,yshift=1ex] {$\mu_{A,B} \ot F(C)$};
		\draw[transform canvas={yshift=-0.2ex,xshift=-0.4ex},->,thick]
		(2110) -- (h) node [midway,xshift=-8ex,yshift=-1ex] {$\mu^{-1}_{A,B} \ot F(C)$};
	\end{tikzpicture}
	\caption{The hexagon diagram for a string of objects \(A,B,C\) in a monoidal category \(\cC\) and a strong monoidal functor \(F:\cC \to \cD\).}
\end{figure}

Coherence for these functors asserts that all parallel morphisms built from $\mu$ and $\mu^{-1}$, the associators of $\cC$ and \(\cD\), and their inverses coincide \cite{Po89}; see also \cite{Ep66,MaPo22,GuJo25}.

Functorial coherence can also be phrased in terms of paths on polytopes; we now explain this for the class of functors relevant to linearly distributive categories:
A \emph{Frobenius linearly distributive functor} between linearly distributive categories \(\cC\) and \(\cD\) is a functor \(F\colon \!\cC \!\!\to\!\cD\) equipped with natural transformations
\begin{equation}
	\begin{array}{rcl}
	\mathcolor{purple}{\mu_{A, B}}\colon \quad F(A)\ot F(B) & \mathcolor{purple}{\slongrightarrow} &F(A\ot B) \,,\\[0.3em]
	\mathcolor{purple}{\Delta_{A, B}}\colon \quad F(A\partimes B) & \mathcolor{purple}{\slongrightarrow} &F(A)\partimes F(B)\,,
	\end{array}
\end{equation}
not assumed invertible, and satisfying the four hexagon diagrams displayed in Definition \ref{def: Vorsicht Funktor}. As before, these diagrams are directed. Units can again obstruct coherence; this already occurs for lax monoidal functors \cite{Le72}.

Frobenius linearly distributive functors have interesting applications. Applied to the terminal category, they yield linearly distributive Frobenius algebras, for which a rich theory exists \cite{FSSW25b,DeS25}. In monoidal categories, they specialize to Frobenius monoidal functors, recently studied in, e.g., \cite{Ya24, FlLaPo25, CzKQW25, JaYa26, MePou26}. In particular, they lead to relations between graphical calculi \cite{FuSY25}. For linearly distributive categories, a simple example in the setting of Example \ref{ex: Abimod2} is the restriction functor $B\operatorname{-bimod} \to A\operatorname{-bimod}$ along a morphism \(A \to B\) of finite-dimensional algebras; see \cite[Ex.~2.15.(ii)]{DeS25}.

\subsection{Coherence for functors}

Multiplihedra \cite{St70,SaUm04,Fo08,MW10,PP25} play for functors the rôle that associahedra play for categories. Polygonal flips arise from directed pentagons for associators and distributors, directed hexagons for the Frobenius data, and directed naturality squares.
Here, we display one example of such a multiplihedron:

\begin{figure}[H]
	\centering
	\begin{tikzpicture}[x=1.16mm,y=1mm,scale=1.7]
		\tikzset{every node/.style={scale=.8}}
		\tikzset{>=stealth}
		\tikzset{mm/.style={execute at begin node=$\displaystyle, execute at end node=$}}
		\draw[mm] 
		(-11,-8) node (a) {F(A)(F(B)(F(C)F(D)))}
		(12,-12) node (b) {F(A)((F(B)F(C))F(D))}
		(-6,22) node (c) {F(A)F(B(CD))}
		(17,17) node (o) {F(A)F((BC)D)}
		(16,-32) node (d) {((F(A)F(B))F(C))F(D)}
		(-23,-3) node (e) {(F(A)F(B))F(CD)}
		(30,-9) node (f) {(F(A)F(BC))F(D)}
		(-16.5,11) node (g) {F(A)(F(B)F(CD))}
		(21,-1) node (h) {F(A)(F(BC)F(D))}
		(-17,-24) node (i) {(F(A)F(B))(F(C)F(D))}
		(21,-19) node (j) {(F(A)(F(B)F(C)))F(D)}
		(30,23.5) node (p) {F(A((BC)D))}
		(7,28.5) node (q) {F(A(B(CD)))}
		(37.5,16) node (r) {F((A(BC))D)}
		(42.5,-0.5) node (s) {F(A(BC))F(D)}
		(38,-14) node (t) {F((AB)C)F(D)}
		(30,-24) node (u) {(F(AB)F(C))F(D)}
		;
		\draw[mm,color=black!50!white]
		(0,13) node (k) {F((AB)(CD))}
		(-4,-16) node (l) {F(AB)(F(C)F(D))}
		(-9,4) node (m) {F(AB)F(CD)}
		(33.5,4.5) node (n) {F(((AB)C)D)}
		;
		
		 \begin{scope}[on background layer]
		\draw[transform canvas={yshift=-0.3ex,xshift=0.4ex},->,thick, dashed, color=black!50!white] (r) -- (n);
		\draw[transform canvas={xshift=-0.4ex},<-,thick, dashed, color=black!50!white] (r) -- (n);
		\draw[transform canvas={xshift=0.5ex},->,thick, dashed, color=black!50!white] (q) -- (k);
		\draw[transform canvas={xshift=-0.5ex},<-,thick, dashed, color=black!50!white] (q) -- (k);
		\end{scope}
		
		\draw[transform canvas={xshift=0.4ex},->,thick] (s) -- (t);
		\draw[transform canvas={xshift=-0.4ex},<-,thick] (s) -- (t);
		\draw[transform canvas={yshift=-0.3ex,xshift=0.3ex},->,thick] (j) -- (d);
		\draw[transform canvas={yshift=0.3ex,xshift=-0.3ex},<-,thick] (j) -- (d);
		\draw[transform canvas={xshift=-0.4ex},->,thick] (i) -- (a);
		\draw[transform canvas={xshift=0.4ex},<-,thick] (i) -- (a);
		\draw[transform canvas={xshift=0.4ex},->,thick] (g) -- (e);
		\draw[transform canvas={xshift=-0.4ex},<-,thick] (g) -- (e);
		\path[->,color=purple,thick] 
		(a) edge (b)
		(g) edge (a)
		(g) edge (c)
		(j) edge (f)
		(e) edge (i)
		(c) edge (o)
		(o) edge (h)
		(c) edge (q)
		(q) edge (p)
		(o) edge (p)
		(p) edge (r)
		(r) edge (s)
		(f) edge (s)
		(b) edge (j)
		(u) edge (t)
		(d) edge (u)
		(b) edge (h)
		(h) edge (f)
		(i) edge (d)
		;
		 \begin{scope}[on background layer]
		\path[->, dashed, color=purple!50!white,thick] 
		(n) edge (t)
		(l) edge (u)
		(m) edge (l)
		(m) edge (k)
		(e) edge (m)
		(k) edge (n)
		(i) edge (l)
		;
		\end{scope}
	\end{tikzpicture}
	\caption{The directed multiplihedron for the unbracketed expression \(F(A) \ot F(B) \ot F(C) \prt F(D)\).}
\end{figure}
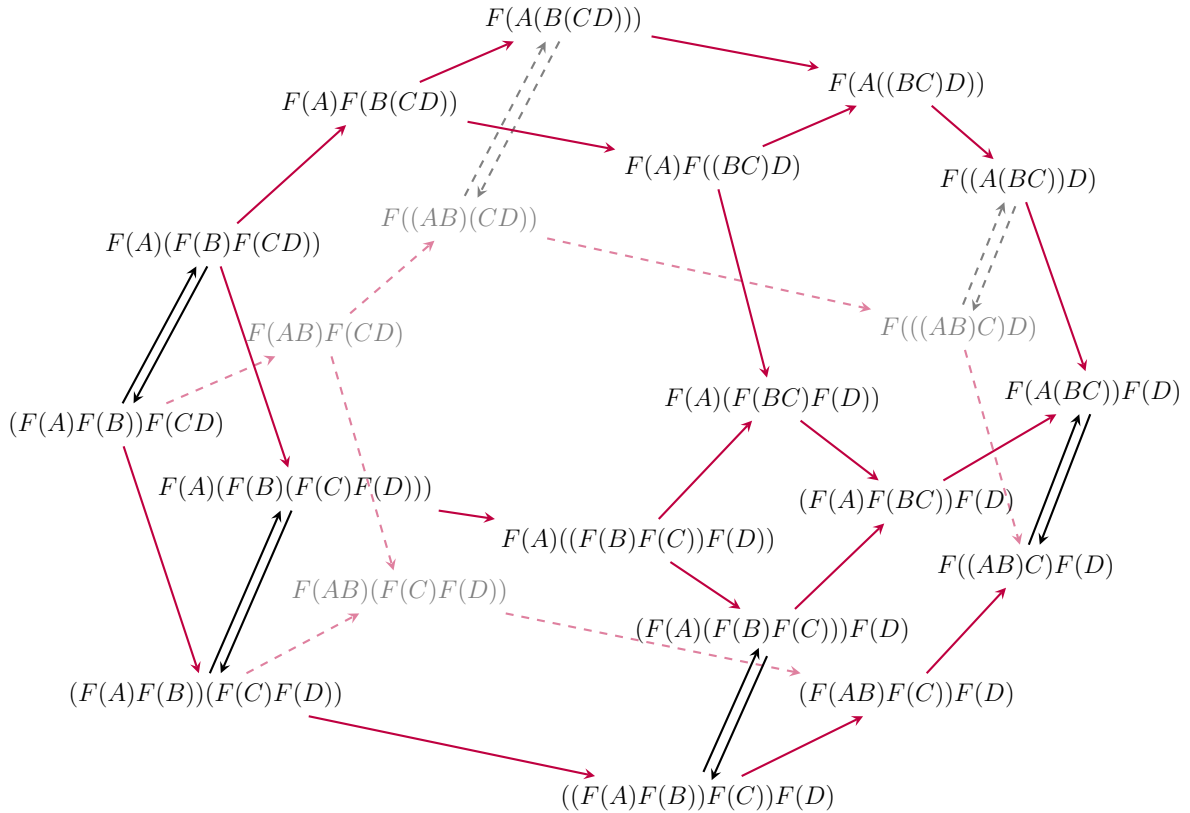

Coherence for Frobenius linearly distributive functors then asserts that any two parallel directed paths in any directed multiplihedron can be transformed by polygonal flips. The techniques developed in this article are strong enough to show such a coherence result.

\subsection{Summary and main results} 
In the rest of this introduction, we summarize the structure of the article and
its main results. 

In Section \ref{sec:prelim} we review the notions
sketched in this introduction. Since units can break coherence, we discard 
the units and their constraints entirely and define a \emph{unitless linearly 
	distributive category} in Definition~\ref{def: unitless LD-category}. 
In the same section, we also describe specific models for free linearly distributive
categories (and some related classes of categories). They have
the property to be tidy, a notion introduced in Definitions \ref{def:montidy} and
\ref{def:LDtidy}. 

In Section \ref{sec:tidyanalysis}, we analyse morphisms in tidy linearly distributive categories. The reader should in particular take notice of the Definition
\ref{def:ot-analysable} of \(\ff\)-analysable morphisms, where  \(\ff\) is a set of morphisms whose targets are \(\ot\)\kern0.1em-products. In Definition \ref{def:par-analysable}, a similar notion of analysability is introduced for a set of morphisms whose targets are \(\prt\)-products.

In Section \ref{sec:normform}, we introduce the notion of an elementary
morphism; in our model of a free linearly distributive category, all morphisms are elementary.
The main result of this section is Proposition \ref{prop:normal cat} which gives a normal form for elementary morphisms in a tidy linearly distributive category. Section \ref{subsec:normFrob}
develops similar results for Frobenius linearly distributive functors and, in particular, applies them to linearly distributive Frobenius algebras.

\medskip

With these results, we derive in Section \ref{sec:cohthms} our two main coherence
results. To state the first result concisely, recall that
a \emph{formal diagram} is, informally, a diagram built from associators and distributors (that is, from rebracketing operations) without 
relying on accidental coincidences between objects; see 
Definition~\ref{def: formal diagram} for a precise formulation. 
   
Call an object \(Y\) of a category
\emph{thin} if for every object \(X\) there is at most one morphism from \(X\) to \(Y\). A category is \emph{thin} if all of its objects are thin.

\medskip

Our first result is a coherence theorem for unitless linearly distributive categories:
\begin{cohthm}[Theorem~\ref{thm: coherence}]
	The free unitless linearly distributive category on a set is thin. Equivalently, every formal diagram in a unitless linearly distributive category commutes. This also holds for unitless linearly distributive categories with not necessarily invertible associators.
\end{cohthm}
In short, only the units can obstruct full coherence. This shows that the usual eight pentagon coherence axioms for associators and distributors in the definition of a linearly distributive category are complete.

\medskip

Our second result is a coherence theorem for Frobenius linearly distributive functors. It relies on the notion of \(F\)-formal diagrams (Definition \ref{def: F-formal diagram}).

\begin{funthm}[Theorem~\ref{thm: funcoherence}]
	The target category of the free unitless Frobenius linearly distributive functor on a set is thin. Equivalently, given a unitless Frobenius linearly distributive functor \(F\), every \(F\)-formal diagram commutes.
\end{funthm}

The following corollary concerns unitless linearly distributive Frobenius algebras.

\begin{spider}[Corollary~\ref{cor:spider}]
	Let \(A\) be a unitless linearly distributive Frobenius algebra, and let \(m,n\) be positive integers. Every morphism \(f\colon A^{\ot m}\to A^{\prt n}\) built from finitely many multiplications \(\mu\), comultiplications \(\Delta\) and identities using the operations \(\circ\), \(\ot\), and \(\prt\), is equal to the following normal form:
	\begin{equation}
		f \eq \Delta^{(n)} \circ \cdots\circ \Delta^{(1)} \circ \mu^{(1)} \circ \cdots\circ \mu^{(m)} \,.
	\end{equation}
	Here, \(A^{\ot 1}=A^{\prt 1}=A\), \(A^{\ot k}=A \ot A^{\ot k-1}\), and \(A^{\prt k}=A \prt A^{\prt k-1}\) for \(k>1\). Moreover, \(\mu^{(1)}=\Delta^{(1)}=\id\), \(\mu^{(2)}=\mu\), \(\Delta^{(2)}=\Delta\), and for \(k>2\), \(\mu^{(k)}=A\ot \mu^{(k-1)}\), \(\Delta^{(k)}=A\prt \Delta^{(k-1)}\). 
\end{spider}

Note that, unlike for Frobenius algebras in monoidal categories, we do not impose that the string diagram obtained from \(f\) is connected; in fact, it is necessarily connected. 

\addtocontents{toc}{\protect\setcounter{tocdepth}{2}}

\newpage
\section{Preliminaries}\label{sec:prelim}

\subsection{Unitless linearly distributive (LD) categories and their functors}

We recall some notions: Unless otherwise specified, all monoidal structures are unitless. Monoidal categories need not have a unit object, and functors need not preserve units. By default, we omit the adjective \emph{unitless} and refer to the unitless variants without further qualification.

\begin{definition}[Cf. \cite{CS97,Bla23}]\label{def:laxLD}
    A \emph{lax linearly distributive (LD) category} consists of a category~\(\cC\), two bifunctors
    \begin{align}
        \ot,\partimes\colon\; \cC \tim \cC &\;\slongrightarrow\; \cC \,,
    \end{align}
    two (not necessarily invertible) natural transformations 
    \begin{align}
        \ao\colon\; {\ot} \circ (\idC \tim \ot) &\;\slongrightarrow\; {\ot} \circ (\ot \tim \idC)\,,\\
        \ap\colon\; {\partimes} \circ (\idC \tim \partimes) &\;\slongrightarrow\; {\partimes} \circ ({\partimes} \tim \idC)\,,
    \end{align}
    called \emph{associators}, and two (not necessarily invertible) natural transformations 
    \begin{align}\label{eq:distl}
        \distl\colon\; {\ot} \circ (\idC \tim {\partimes}) &\;\slongrightarrow\; {\partimes} \circ ({\ot} \tim \idC)\,,\\ \label{eq:distr}
        \distr\colon\; {\ot} \circ ({\partimes} \tim \idC) &\;\slongrightarrow\; {\partimes} \circ ({\idC} \tim {\ot})\,,
    \end{align}
    called \emph{distributors}. The associators and distributors are required to satisfy the following eight pentagon coherence axioms, for all objects \(A,B,C,D\in \cC\):

\begingroup
\renewcommand{\theequation}{\textup{P1}}
\refstepcounter{equation}\label{eq:P1}
\begin{equation*}
	\begin{tikzcd}[column sep=-1cm, row sep=0.5cm]
		A\ot (B\ot (C\ot D)) && (A\ot B)\ot (C\ot D) && ((A\ot B)\ot C)\otimes D \\
		&& \text{(P1)} && \\
		& A\ot ((B\ot C)\ot D) && (A\ot (B\ot C))\ot D &
		\arrow[purple,"\ao", from=1-1, to=1-3]
		\arrow[purple,"\ao", from=1-3, to=1-5]
		\arrow[purple,"\ao\,\ot\,\id", from=3-4, to=1-5]
		\arrow[purple,"\id\,\ot\,\ao", from=1-1, to=3-2]
		\arrow[purple,"\ao", from=3-2, to=3-4]
	\end{tikzcd}
\end{equation*}
\endgroup

\vspace{0.5cm}

\begingroup
\renewcommand{\theequation}{\textup{P2}}
\refstepcounter{equation}\label{eq:P2}
\begin{equation*}
\begin{tikzcd}[column sep=-1cm, row sep=0.5cm]
	A\ot (B\ot (C\partimes D)) && (A\ot B)\ot (C\prt D) && ((A\ot B)\ot C)\prt D \\
	&& \text{(P2)} && \\
	& A\ot ((B\ot C)\prt D) && (A\ot (B\ot C))\prt D &
	\arrow[purple,"\ao", from=1-1, to=1-3]
	\arrow[purple,"\distl", from=1-3, to=1-5]
	\arrow[purple,"\id\,\ot\,\distl", from=1-1, to=3-2]
	\arrow[purple,"\distl", from=3-2, to=3-4]
	\arrow[purple,"\ao\prt\id", from=3-4, to=1-5]
\end{tikzcd}
\end{equation*}
\endgroup

\vspace{0.5cm}

\begingroup
\renewcommand{\theequation}{\textup{P3}}
\refstepcounter{equation}\label{eq:P3}
\begin{equation*}
	\begin{tikzcd}[column sep=-1cm, row sep=0.5cm]
		A\ot ((B\prt C)\ot D) && A\ot (B\prt {(C\ot D)}) && (A\ot B)\prt {(C\ot D)}\\
		&& \text{(P3)} && \\
		& (A\ot (B\prt C))\ot D && ((A\ot B)\prt C)\ot D &
		\arrow[purple,"\id\,\ot\,\distr", from=1-1, to=1-3]
		\arrow[purple,"\distl", from=1-3, to=1-5]
		\arrow[purple,"\ao", from=1-1, to=3-2]
		\arrow[purple,"\distl\,\ot\,\id", from=3-2, to=3-4]
		\arrow[purple,"\distr", from=3-4, to=1-5]
	\end{tikzcd}
\end{equation*}
\endgroup

\vspace{0.5cm}

\begingroup
\renewcommand{\theequation}{\textup{P4}}
\refstepcounter{equation}\label{eq:P4}
\begin{equation*}
	\begin{tikzcd}[column sep=-1cm, row sep=0.5cm]
		A\ot (B\prt {(C\prt D)}) && (A\ot B)\prt {(C\prt D)} && ((A\ot B)\prt C)\prt D \\
		&& \text{(P4)} && \\
		& A\ot ((B\prt C)\prt D) && (A\ot (B\prt C))\prt D &
		\arrow[purple,"\distl", from=1-1, to=1-3]
		\arrow[purple,"\ap", from=1-3, to=1-5]
		\arrow[purple,"\id\,\ot\,\ap", from=1-1, to=3-2]
		\arrow[purple,"\distl", from=3-2, to=3-4]
		\arrow[purple,"\distl\prt\id", from=3-4, to=1-5]
	\end{tikzcd}
\end{equation*}
\endgroup

\vspace{0.5cm}

\begingroup
\renewcommand{\theequation}{\textup{P5}}
\refstepcounter{equation}\label{eq:P5}
\begin{equation*}
	\begin{tikzcd}[column sep=-1cm, row sep=0.5cm]
		(A\prt B)\ot (C\ot D) && A\prt {(B\ot (C\ot D))} && A\prt {((B\ot C)\ot D)} \\
		&& \text{(P5)} && \\
		& ((A\prt B)\ot C)\ot D && (A\prt {(B\ot C)})\ot D &
		\arrow[purple,"\distr", from=1-1, to=1-3]
		\arrow[purple,"\id\prt\ao", from=1-3, to=1-5]
		\arrow[purple,"\ao", from=1-1, to=3-2]
		\arrow[purple,"\distr\,\ot\,\id", from=3-2, to=3-4]
		\arrow[purple,"\distr", from=3-4, to=1-5]
	\end{tikzcd}
\end{equation*}
\endgroup

\vspace{0.5cm}

\begingroup
\renewcommand{\theequation}{\textup{P6}}
\refstepcounter{equation}\label{eq:P6}
\begin{equation*}
	\begin{tikzcd}[column sep=-1cm, row sep=0.5cm]
		(A\prt B)\ot (C\prt D) && ((A\prt B)\ot C)\prt D && (A\prt {(B\ot C)})\prt D \\
		&& \text{(P6)} && \\
		&  A\prt {(B\ot (C\prt D))}  && A\prt {((B\ot C)\prt D)} &
		\arrow[purple,"\distl", from=1-1, to=1-3]
		\arrow[purple,"\distr\prt\id", from=1-3, to=1-5]
		\arrow[purple,"\distr", from=1-1, to=3-2]
		\arrow[purple,"\id\prt\distl", from=3-2, to=3-4]
		\arrow[purple,"\ap", from=3-4, to=1-5]
	\end{tikzcd}
\end{equation*}
\endgroup

\vspace{0.5cm}

\begingroup
\renewcommand{\theequation}{\textup{P7}}
\refstepcounter{equation}\label{eq:P7}
\begin{equation*}
	\begin{tikzcd}[column sep=-1cm, row sep=0.5cm]
		(A\prt {(B\prt C)})\ot D && ((A\prt B)\prt C)\ot D && (A\prt B)\prt {(C\ot D)} \\
		&& \text{(P7)} && \\
		& A\prt {((B\prt C)\ot D)} && A\prt {(B\prt {(C\ot D)})} &
		\arrow[purple,"\ap\,\ot\,\id", from=1-1, to=1-3]
		\arrow[purple,"\distr", from=1-3, to=1-5]
		\arrow[purple,"\distr", from=1-1, to=3-2]
		\arrow[purple,"\id\prt\distr", from=3-2, to=3-4]
		\arrow[purple,"\ap", from=3-4, to=1-5]
	\end{tikzcd}
\end{equation*}
\endgroup

\vspace{0.5cm}

\begingroup
\renewcommand{\theequation}{\textup{P8}}
\refstepcounter{equation}\label{eq:P8}
\begin{equation*}
	\begin{tikzcd}[column sep=-1cm, row sep=0.5cm]
		A\prt {(B\prt {(C\prt D)})} && (A\prt B)\prt {(C\prt D)} && ({(A\prt B)}\prt C)\prt D \\
		&& \text{(P8)} && \\
		& A\prt {({(B\prt C)}\prt D)} && {(A\prt {(B\prt C)})}\prt D &
		\arrow[purple,"\ap", from=1-1, to=1-3]
		\arrow[purple,"\ap", from=1-3, to=1-5]
		\arrow[purple,"\id\prt\ap", from=1-1, to=3-2]
		\arrow[purple,"\ap", from=3-2, to=3-4]
		\arrow[purple,"\ap\prt\id", from=3-4, to=1-5]
	\end{tikzcd}
\end{equation*}
\endgroup
\end{definition}

\begin{definition}[Cf. \cite{CS97}]\label{def: unitless LD-category}
	A lax LD category whose associators are both invertible is called a \emph{linearly distributive (LD) category}.
\end{definition}

\begin{definition}[Cf. \cite{Lap72}]\label{def: unitless lax monoidal category}
	A \emph{lax monoidal category} consists of a category \(\cC\), a bifunctor 
	\begin{align}
		\ot\colon\; \cC \tim \cC &\;\slongrightarrow\; \cC \,,
	\end{align}
	and a (not necessarily invertible) natural transformation 
	\begin{align}
		\ao\colon {\ot} \circ (\idC \tim \ot) &\;\slongrightarrow\; {\ot} \circ {({\ot} \tim {\idC})}\,,
	\end{align}
	satisfying the pentagon coherence condition \eqref{eq:P1}, for all objects \(A,B,C,D\in \cC\).
\end{definition}

\begin{definition}[Cf. \cite{CS99}]\label{def: Vorsicht Funktor}
A \emph{Frobenius linearly distributive (LD) functor} between lax LD categories is a functor \(F\colon\cC \!\!\to \cD\), together with natural transformations
\begin{align*}
	\mu_{A, B}\colon\; F(A)\ot F(B) & \;\slongrightarrow\; F(A\ot B)\,, \\
	\Delta_{A, B}\colon\; F(A\partimes B) & \;\slongrightarrow\; F(A)\partimes F(B)\,,
\end{align*}
such that, for all objects \(A, B, C\in\cC\), the following four hexagon diagrams commute.

\begingroup
\renewcommand{\theequation}{\textup{H1}}
\refstepcounter{equation}\label{eq:H1}
\begin{equation*}
\hspace{-1cm}
\begin{tikzcd}[column sep=1cm, row sep=0.3cm]
F(A)\ot\bigl(F(B)\ot F(C)\bigr) & \bigl(F(A)\ot F(B)\bigr)\ot F(C) && F(A\ot B)\ot F(C) \\
& \qquad\;\text{(H1)} && \\
F(A)\ot F(B\ot C) & F\bigl(A\ot (B\ot C)\bigr) && F\bigl((A\ot B)\ot C\bigr)
\arrow[purple,"\ao", from=1-1, to=1-2]
\arrow[purple,"\mu\,\ot\,\id", from=1-2, to=1-4]
\arrow[purple,"\id\,\ot\,\mu", from=1-1, to=3-1]
\arrow[purple,"\mu", from=1-4, to=3-4]
\arrow[purple,"\mu", from=3-1, to=3-2]
\arrow[purple,"F(\ao)", from=3-2, to=3-4]
\end{tikzcd}
\end{equation*}
\endgroup

\vspace{0.4cm}

\begingroup
\renewcommand{\theequation}{\textup{H2}}
\refstepcounter{equation}\label{eq:H2}
\begin{equation*}
\begin{tikzcd}[column sep=1cm, row sep=0.3cm]
F(A)\ot F(B\prt C) && F(A)\ot\bigl(F(B) \prt F(C)\bigr) & \bigl(F(A)\ot F(B)\bigr)\prt F(C) \\
&& \text{(H2)} & \\
F\bigl(A\ot (B\prt C)\bigr) && F\bigl((A\ot B)\prt C)\bigr) & F(A\ot B)\prt F(C)
\arrow[purple,"\id\,\ot\,\Delta", from=1-1, to=1-3]
\arrow[purple,"\distl", from=1-3, to=1-4]
\arrow[purple,"\mu", from=1-1, to=3-1]
\arrow[purple,"\mu\prt\id", from=1-4, to=3-4]
\arrow[purple,"F(\distl)", from=3-1, to=3-3]
\arrow[purple,"\Delta", from=3-3, to=3-4]
\end{tikzcd}
\end{equation*}
\endgroup

\vspace{0.4cm}

\begingroup
\renewcommand{\theequation}{\textup{H3}}
\refstepcounter{equation}\label{eq:H3}
\begin{equation*}
\begin{tikzcd}[column sep=1cm, row sep=0.3cm]
F(A \prt B)\ot F(C) && \bigl(F(A)\prt F(B)\bigr)\ot F(C) & F(A) \prt {\bigl(F(B)\ot F(C)\bigr)} \\
&& \;\text{(H3)} & \\
F\bigl((A\prt B)\ot C\bigr) && F\bigl(A \prt {(B\ot C)}\bigr) & F(A)\prt F(B\ot C)
\arrow[purple,"\Delta\,\ot\,\id", from=1-1, to=1-3]
\arrow[purple,"\distr", from=1-3, to=1-4]
\arrow[purple,"\mu", from=1-1, to=3-1]
\arrow[purple,"\id\prt\mu", from=1-4, to=3-4]
\arrow[purple,"F(\distr)", from=3-1, to=3-3]
\arrow[purple,"\Delta", from=3-3, to=3-4]
\end{tikzcd}
\end{equation*}
\endgroup

\vspace{0.4cm}

\begingroup
\renewcommand{\theequation}{\textup{H4}}
\refstepcounter{equation}\label{eq:H4}
\begin{equation*} 
\begin{tikzcd}[column sep=1cm, row sep=0.3cm]
F\bigl(A \prt {(B \prt C)}\bigr) && F(A)\prt F(B\prt C) && F(A)\prt {\bigl(F(B)\prt F(C)\bigr)} \\
&& \qquad \quad\, \text{(H4)} && \\
F\bigl((A \prt B) \prt C\bigr) && F(A\prt B)\prt F(C) && \bigl(F(A) \prt F(B)\bigr)\prt F(C)
\arrow[purple,"\Delta", from=1-1, to=1-3]
\arrow[purple,"\id\prt \Delta", from=1-3, to=1-5]
\arrow[purple,"F(\ap)", from=1-1, to=3-1]
\arrow[purple,"\ap", from=1-5, to=3-5]
\arrow[purple,"\Delta", from=3-1, to=3-3]
\arrow[purple,"\Delta\prt\id", from=3-3, to=3-5]
\end{tikzcd}
\end{equation*}
\endgroup
\end{definition}

\begin{definition}
	A Frobenius linearly distributive (LD) functor is called \emph{strict} if its coherence morphisms $\mu$ and $\Delta$ are identities.
\end{definition}

\subsection{Freeness, tidiness, and formality}
In this subsection, we define free LD categories and free Frobenius LD functors on a set. We construct explicit models of these universal structures which have desirable formal properties. Specifically, they are \emph{tidy}, a notion that we introduce in Definitions \ref{def:montidy}, \ref{def:LDtidy} and \ref{def: tidy functor}. We then define formal diagrams, which will be used in the formulation of the coherence theorems.

\subsubsection{Freeness}
Write \(\LD\) for the category of small LD categories with strict Frobenius LD functors. Let \(U\colon \LD\to \ms{Set}\) be the forgetful functor that takes each small LD category to its set of objects and each strict Frobenius LD functor to its induced function on objects.

\begin{definition}\label{def: free LD}
Let \(\fS\) be a set. The \emph{free LD category} on \(\fS\) consists of an LD category \(\Free\) together with a function \(\eta_S\colon \fS \to U(\Free)\) such that the following universal property holds: for every LD category \(\cC\) and every function \(f\colon \fS \to U(\cC)\), there exists a unique strict Frobenius LD functor \(G\colon \Free \to \cC\) satisfying 
\begin{equation}
f \eq U(G)\circ \eta_S.
\end{equation}
\end{definition}

\begin{remark}
Free lax monoidal and free lax LD categories on \(\fS\) are defined analogously.
\end{remark}

The following construction yields a concrete model of the free LD category \(\Free\) on \(\fS\).
\begin{construction}\label{con: free LD cat}
	Its set of objects is defined recursively by declaring \(\fS\subseteq \operatorname{Ob}(\Free)\) and, for any \(A,B\in \operatorname{Ob}(\Free)\), requiring that \(A\ot B \in \operatorname{Ob}(\Free)\) and \(A\partimes B \in \operatorname{Ob}(\Free)\).
	
	The morphisms of \(\Free\) are defined as follows:
	\begin{enumerate}[label=\rmlabel]
		\item For each object \(A\in \operatorname{Ob}(\Free)\), we include 
		a morphism \(\id_A\colon A \to A\). 
		\item For all objects \(A,B,C\in \operatorname{Ob}(\Free)\), we include 
		morphisms \(\ao_{A,B,C}\), \(\ap_{A,B,C}\), \(\aoi_{A,B,C}\), 
		\(\api_{A,B,C}\), \(\distl_{A,B,C}\) and \(\distr_{A,B,C}\), 
		each with its evident source and target. 
		\item Given morphisms \(f\colon A \to B\) and \(g\colon B \to C\), we include a morphism \(g\circ f\colon A \to C\). 
		\item Given morphisms \(f\colon A \to C\) and \(g\colon B \to D\), we include two morphisms\\ \(f\otimes g\colon A \otimes B \to C \otimes D\) 
		and \(f\partimes g\colon A \partimes B \to C \partimes D\). 
	\end{enumerate}
	We impose only the relations required to endow \(\Free\) with the structure 
	of a LD category, and no others. Concretely, we quotient by the relations expressing
	\begin{enumerate}[resume,label=\rmlabel]
		\item that \((\Free,\circ,\id)\) is a category, 
		\item that \(\ot\) and \(\partimes\) define bifunctors, 
		\item and that the families \(\ao\), \(\ap\), \(\distl\), and \(\distr\) 
		are natural. 
	\end{enumerate} 
	Further, we impose 
	\begin{enumerate}[resume,label=\rmlabel]
		\item that \(\ao\) and \(\aoi\) are mutually inverse, as are \(\ap\) 
		and \(\api\), 
		\item and the pentagon coherence 
		axioms~\eqref{eq:P1}\,--\,\eqref{eq:P8} listed in 
		Definition \ref{def:laxLD}.
	\end{enumerate}
	The LD category \(\Free\), together with the evident inclusion \(\eta_S\colon \fS\hookrightarrow \operatorname{Ob}(\Free)\), satisfies the universal property of Definition~\ref{def: free LD}.
\end{construction}

\begin{remark}
Free lax monoidal and free lax LD categories on~\(\fS\) are constructed analogously.
\end{remark}

Next, we consider functors. Write \(\uFrob\) for the category whose objects are Frobenius LD functors between small LD categories, and whose morphisms \(F \to G\) between Frobenius LD functors \(F\colon \cC \to \cD\) and \(G\colon \cE \to \cF\) are pairs \((H,I)\) of strict Frobenius LD functors \(H\colon \cC \to \cE\) and \(I\colon \cD \to \cF\) such that \(I \circ F=G\circ H\); compare \cite[\S 7]{MaPo22}.

Let \(U\colon \uFrob \to \ms{Set}\) be the functor sending a Frobenius LD functor to the set of objects of its source category and a morphism \((H,I)\colon F\to G\) to the object function underlying \(H\).
\begin{definition}\label{def: free Frob}
	Let \(\fS\) be a set. The \emph{free Frobenius LD functor} on \(\fS\) consists of a Frobenius LD functor \(\FreeF\) together with a function \(\eta_S\colon S \to U(\FreeF)\) such that the following universal property holds: for every Frobenius LD functor \(G\colon \cC\to \cD\) and every function \(f\colon \fS \to U(G)\), there exists a unique morphism \((H,I)\colon \FreeF \to G\) in \(\uFrob\) satisfying 
	\begin{equation}
		f \eq U(H,I)\circ \eta_S.
	\end{equation}
\end{definition}

\begin{remark}
	Free lax monoidal functors on \(\fS\) are defined analogously.
\end{remark}

The following construction yields a model of the free Frobenius LD functor \(F_S\) on a set \(S\).

\begin{construction}\label{con: free Frob LD fun}
	Let \(\cC=\Free\) be the free LD category on \(S\), presented in Construction~\ref{con: free LD cat}.
	\begin{enumerate}[label=\rmlabel]
		\item\label{it: con T} Let \(T := \{ F(X) \,\colon\, X \in \cC \}\) be the set of formal symbols, indexed by the objects \(X\in\cC\).
		\item Let \(\cD^-\) be the free LD category on the set \(T\), presented in Construction~\ref{con: free LD cat}. 
		\item Add to \(\cD^-\), for any morphism \(f\colon X \to Y\) of \(\cC\), a new morphism 
		\begin{equation}
		F(f)\colon\; F(X)\slongrightarrow F(Y)\,,
		\end{equation} 
	 	and, for any two objects \(X, Y\in \cC\), two new morphisms 
		\begin{align}
			\mu_{X, Y}\colon\; F(X)\ot F(Y) &\slongrightarrow F(X\ot Y) \,, \\
			\Delta_{X, Y}\colon\; F(X\prt Y) &\slongrightarrow F(X)\prt F(Y) \,.
		\end{align}
		\item Close this new collection of morphisms
		under composition and both tensor products. 
		\item Define \(\cD\) as the quotient of this category by the smallest congruence such that \(\ot\) and \(\prt\) become bifunctors, the families
		\(\ao\), \(\aoi\), \(\distl\), \(\distr\), \(\ap\), \(\api\), \(\mu\) and \(\Delta\) become natural, the assignment
		\(F\) becomes a functor, and the hexagon axioms~\eqref{eq:H1}\,--\,\eqref{eq:H4} are satisfied. 
	\end{enumerate}
	This Frobenius LD functor \(F\colon \!\cC \!\to \cD\), with the inclusion \(\eta_S\colon \fS\hookrightarrow \operatorname{Ob}(\Free)=\operatorname{Ob}(\cC)\), satisfies the universal property of Definition~\ref{def: free Frob}.
\end{construction}

\subsubsection{Tidiness}

We introduce a class of lax LD categories and Frobenius LD functors with good formal properties.
\begin{definition}\label{def:montidy}
	A lax monoidal category \(\cC\) is called \emph{tidy} if \(A_1\otimes A_2=B_1\ot B_2\) implies \(A_1=B_1\) and \(A_2=B_2\) for all 
	objects \(A_1, A_2, B_1, B_2\in \cC\).
\end{definition}

\begin{definition}\label{def:LDtidy}
	A lax LD category \(\cC\) is \emph{tidy} if it is tidy 
	with respect to both tensor products and, moreover, \(A_1\otimes A_2\ne B_1\partimes B_2\) for all objects \(A_1, A_2, B_1, B_2\in \cC\).
\end{definition}

\begin{example}
	The free LD category on a set, as in Construction \ref{con: free LD cat}, and the target category of the free Frobenius LD functor on a set, as in Construction \ref{con: free Frob LD fun}, are both tidy. 
\end{example}

We now turn to Frobenius LD functors. We need the following notions.

\begin{definition}
	An object \(A \in \cC\) in a tidy lax LD category \(\cC\) is called an \emph{\(\ot\)-product} if it can be written 
	as \(A=A_1\ot A_2\) for (necessarily unique) objects \(A_1, A_2\in \cC\), and a \emph{\(\partimes\)-product} if it can be written as \(A = A_1 \partimes A_2\).
\end{definition}

\begin{definition}\label{def: tidy functor}
	A Frobenius LD functor \(F\colon \cC \to \cD\) 
	between tidy lax~LD categories is \emph{tidy} if, 
	for every object \(A\in \cC\), the object \(F(A)\in \cD\) is neither an \(\ot\)\kern0.1em-product nor a \(\prt\)-product.
\end{definition}

\begin{example}
	The free Frobenius LD functor on a set, as in Construction \ref{con: free Frob LD fun}, is tidy.
\end{example}

\subsubsection{Formality}

We use the universal properties of free structures to define formal diagrams. 

First, we consider categories. By the universal property in Definition~\ref{def: free LD}, for any LD category \(\cC\), the identity on \(U(\cC)=\operatorname{Ob}(\cC)\) induces a unique strict Frobenius LD functor
\begin{equation}
\varepsilon_{\cC}\colon\; \FreeC \slongrightarrow \cC
\end{equation} 
such that \(A=\big(U(\varepsilon_{\cC})\circ \eta\big)\big(A\big)\) for every \(A\in \cC\). The morphism \(\varepsilon_{\cC}\) is the component at \(\cC\) of the counit for the free-forgetful adjunction between \(\ms{Set}\) and \(\LD\).

\begin{definition}\label{def: formal diagram}
Let \(\cC\) be a LD category. A diagram \(D\colon J\to \cC\) is called \emph{formal} if it factors through \(\varepsilon_{\cC}\), that is, if it lifts to a diagram \(D^{\sharp}\colon J\to \FreeC\) such that \(\varepsilon_{\cC}\circ D^{\sharp}=D\).
\end{definition}

\begin{remark}
	Formal diagrams in lax LD and lax monoidal categories are defined analogously.
\end{remark}

Next, we turn to functors. By the universal property in Definition~\ref{def: free Frob}, for any Frobenius LD functor \(G\colon \cC \to \cD\), the identity function on \(U(G)=\operatorname{Ob}(\cC)\) induces a unique morphism
\begin{equation}
	(\varepsilon^1_{G},\varepsilon^2_{G})\colon\; \FreeG \slongrightarrow G
\end{equation} 
in \(\uFrob\) such that \(A=\big(U(\varepsilon^1_{G},\varepsilon^2_{G})\circ \eta\big)\big(A\big)\) for every \(A\in\cC\). The morphism \((\varepsilon^1_{G},\varepsilon^2_{G})\) is the component at \(G\) of the counit for the free-forgetful adjunction between \(\ms{Set}\) and \(\uFrob\).

\begin{definition}\label{def: F-formal diagram}
	Let \(F\colon \cC \to \cD\) be a Frobenius LD functor between LD categories. A diagram \(D\colon J\to \cD\) is called \emph{\(F\)-formal} if it factors through \(\varepsilon^2_{F}\).
\end{definition}

\section{Morphisms in tidy LD categories}\label{sec:tidyanalysis}
	Throughout this section, we fix a tidy LD category \(\cC\).

\subsection{Morphisms into \texorpdfstring{$\ot$}{ot}-products}\label{subs:morphot}

We adopt the following terminology and notation.
\begin{definition}
	\begin{enumerate}[label=\rmlabel]
	\item The \emph{free monoid} $\fM$ on a set \(M\) consists of all finite sequences (or \emph{words}) of elements of $M$, with concatenation as multiplication and the empty word $\e$ as the unit. By convention, we concatenate from right to left. Accordingly, for \(v,w\in \fM\), the product \(vw\) is said to \emph{begin} with \(w\) and \emph{end} with \(v\). 
	\item The \emph{length} \(\lvert u \rvert \in \NNez\) of a word \(u \in \fM\) is the number of its terms. 
	\item Elements of $M$ are called \emph{generators}; we identify generators with words of length one.
	\item Each non-empty word \(u\in \fM\) factorises uniquely as \(u=m\init(u)\)
	for a generator \(m\in \fM\) and a word \(\init(u)\in \fM\). We refer to~\(m\) as 
	the \emph{terminal letter} of \(u\) and denote it by \(m=\ter(u)\), 
	while \(\init(u)\) is called the \emph{initial word} of \(u\).
	\end{enumerate}
\end{definition}

Let \(\fY\) be the free monoid with generators \(\ao, \aoi\).

\begin{remark}
	Care must be taken to distinguish the generators \(\ao\) and \(\aoi\) of the monoid \(\fY\) from the coherence morphisms \(\ao\) and \(\aoi\) in the tidy LD category \(\cC\).
\end{remark}

The main construction underlying our analysis of morphisms with target an \(\otimes\)\kern0.1em-product associates to every such morphism \(f\colon H\longrightarrow L\ot R\) a family of natural transformations \((\tau^f_u)_{u\in\fY}\), defined recursively for \(u\in \fY\) by

\begin{equation}\label{eq:tau}
\begin{cases}
    \ta^f_{\e} &\hspace{-12pt}\eq f\,, \\[.1em]
    \ta^f_{\ao u} &\hspace{-12pt}\eq \ao \circ ({\id} \ot {\ta^f_u})\,, \\[.1em]
    \ta^f_{\aoi u} &\hspace{-12pt}\eq \aoi \circ (\ta^f_u \ot \id)\,.
\end{cases}
\end{equation}

For clarity, we specify the source and target functors of the natural transformations \eqref{eq:tau}, using the following terminology and notation:

\begin{definition}
\begin{enumerate}[label=\rmlabel]
	\item For any positive integer \(n\), objects of the Cartesian power \(\cC^n\) are called \emph{vectors} and are written in vector notation \(\vect{A},\vect{B},\vect{C},\ldots\)
\item For any vector \(\vect{A}=(A_1,\ldots, A_n)\in \cC^n\) and \(1\leq k \leq n\), we set 
\begin{equation}
	\vect{A}_{\geq k}\,:=\,(A_k,\ldots, A_n)\in \cC^{n-k+1}
	\qquad \text{and} \qquad
	\vect{A}_{\leq k}\,:=\,(A_1,\ldots, A_k) \in \cC^k.
\end{equation}
\end{enumerate}
\end{definition}

\begin{definition}\label{dfn:2205}
For every morphism \(f\colon H\longrightarrow {L \ot R}\) and every \(u\in \fY\), we define functors 
\begin{equation}
    H^f_u, L^f_u, R^f_u\colon\; \cC^{\lvert u \rvert} \,\slongrightarrow\, \cC
\end{equation}
by recursion on \(|u|\). In the base case, we set

\begin{equation}\label{eq:HLRot0}
	\begin{cases}
		H^f_{\e}(\ast) &\hspace{-9pt}\eq H\,, \\[.3em]
		L^f_{\e}(\ast) &\hspace{-9pt}\eq L\,, \\[.3em]
		R^f_{\e}(\ast) &\hspace{-9pt}\eq R\,,
	\end{cases}
\end{equation}
and in the recursive step, we set
\begin{align}
	\begin{cases}
		H^f_{\ao u}(\vect{A}) &\hspace{-8pt}\eq A_1 \ot H^f_{u}(\vect{A}_{\geq 2})\,, \\[.1em]
		L^f_{\ao u}(\vect{A}) &\hspace{-8pt}\eq A_1 \ot L^f_{u}(\vect{A}_{\geq 2})\,,\\[.1em]
		R^f_{\ao u}(\vect{A}) &\hspace{-8pt}\eq R^f_{u}(\vect{A}_{\geq 2})\,,
	\end{cases}
	&&\begin{cases}
		H^f_{\aoi u}(\vect{A}) &\hspace{-8pt}\eq H^f_{u}(\vect{A}_{\leq \lvert u \rvert}) \ot A_{\lvert u \rvert+1}\,, \\[.1em]
		L^f_{\aoi u}(\vect{A}) &\hspace{-8pt}\eq L^f_{u}(\vect{A}_{\leq \lvert u \rvert})\,,\\[.1em]
		R^f_{\aoi u}(\vect{A}) &\hspace{-8pt}\eq R^f_{u}(\vect{A}_{\leq \lvert u \rvert})\ot A_{\lvert u \rvert+1}\,.
	\end{cases}
\end{align}
\end{definition}

\begin{remark}
The category \(\cC^0\) appearing in the base case~\eqref{eq:HLRot0} is the category with a single object \(\ast\) and a single morphism. A functor \(\cC^0 \to \cC\) therefore corresponds to an object of \(\cC\), and natural transformations between such functors correspond to morphisms in~\(\cC\). 

There is also no ambiguity in the base case~\eqref{eq:HLRot0}: the objects~\(H\),~\(L\), and~\(R\) appearing on the right side are determined by the superscript~\(f\) on the left side.
\end{remark}

The recursion~\eqref{eq:tau} defines a family of natural transformations \(\tau^f_u\colon H^f_u \to L^f_{u} \ot R^f_{u}\), where~\(\ot\) denotes the following \(\ot\)\kern0.1em-product of functors:
\begin{definition}\label{def:pointwise tensor product}
	Given functors \(F,G\colon \cC\to \cD\) between LD categories, the \emph{pointwise \(\ot\)-product} is the composite \(F\ot G\,:=\,\ot \circ (F \tim G)\circ \Delta\), where \(\Delta \colon\, \cC \to \cC^2\) is the diagonal functor.
\end{definition}

The tidiness of \(\cC\) immediately implies the following injectivity property of our 
functors~\(H^f_u\).

\begin{lemma}\label{lemma: H is injective on objects}
    If for a word \(u \in \fY\), two vectors \(\vect{A}, \vect{B}\in\cC^{|u|}\)
    and two morphisms \(f\), \(g\) into \(\ot\)-products we 
    have \(H^f_u(\vect{A})= H^g_u(\vect{B})\), then \(\vect{A}=\vect{B}\) 
    and \(s(f)=s(g)\). \qed
\end{lemma}

The pentagon axiom~\eqref{eq:P1} enters our argument only via the following lemma. 

\begin{lemma}\label{lemma: pushing tau through}
    For every morphism~\(f\) into an \(\ot\)-product, the family of natural 
    transformations~\((\ta^f_u)_{u\in \fY}\) satisfies the relations
        \begin{align}
        \ta^f_{\ao u} \circ \ao 
        &\eq (\ao \ot {\id}) \circ \ta^f_{\ao \ao u}\,,\label{eq: alpha alpha}\\[.2em]
        \ta^f_{\aoi \ao u} \circ \ao 
        &\eq \ta^f_{\ao \aoi u}\,,\label{eq: alphai alpha alpha}\\[.2em]
        \ta^f_{\aoi \aoi u} \circ \ao 
        &\eq ({\id} \ot \ao) \circ \ta^f_{\aoi u}\,\label{eq: alphai alphai alpha}\\[.2em]
        \ta^f_{\aoi u} \circ \aoi 
        &\eq ({\id} \ot \aoi) \circ \ta^f_{\aoi \aoi u}\,,\label{eq: alphai alphai}\\[.2em]
        \ta^f_{\ao \aoi u} \circ \aoi 
        &\eq \ta^f_{\aoi \ao u}\,,\label{eq: alpha alphai alphai}\\[.2em]
        \ta^f_{\ao \ao u} \circ \aoi 
        &\eq (\aoi \ot {\id})\circ \ta^f_{\ao u}\,.\label{eq: alpha alpha alphai}
        \end{align} 
\end{lemma}

For the convenience of the reader, we write out the equality of natural transformations~\eqref{eq: alpha alpha} explicitly in components. It reads
\begin{align}\label{eq:in-components}
	\ta^f_{\ao u}(A_1\otimes A_2,\vect{A}_{\geq 3}) 
	\circ 
	\ao_{A_1,A_2,H^f_{u}(\vect{A}_{\geq 3})} 
	&\eq 
	\big(\ao_{A_1,A_2,L^f_u(\vect{A}_{\geq 3})} \ot {R^f_u(\vect{A}_{\geq 3})}\big) 
	\circ 
	\ta^f_{\ao \ao u}(\vect{A})
\end{align}
for all \(\vect{A}\in \cC^{\lvert u \rvert +2}\); the meaning of the five other equations is similar.

We explain the notation used in Equation~\eqref{eq:in-components}:
\begin{remark}
For typographical simplicity, we write \(\ta^f_{\ao \ao u}(\vect{A})\) instead of \((\ta^f_{\ao \ao u})_{\vect{A}}\) for the component of the natural transformation \(\ta^f_{\ao \ao u}\) at \(\vect{A}\), using parentheses instead of subscripts. This notation is employed throughout for the family of natural transformations \(\tau\) and a related family introduced in Subsection \ref{subs:parmor}. As is standard, we sometimes write \({R^f_u(\vect{A}_{\geq 3})}\) instead of $\id_{{R^f_u(\vect{A}_{\geq 3})}}$ for the identity morphism on \({R^f_u(\vect{A}_{\geq 3})}\).
\end{remark}

\begin{proof}
    We prove Equation~\eqref{eq: alpha alpha} as follows:
    \begin{align}
        \ta^f_{\ao u} \circ \ao 
        &\eq \ao \circ (\id \ot \ta^f_{u}) \circ \ao 
        	&\tag{definition \eqref{eq:tau} of \(\ta^f_{\ao u}\)}\\[.2em]
        &\eq \ao \circ \big((\id \ot \id) \ot \ta^f_{u}\big) \circ \ao 
        	&\tag{bifunctoriality of \(\ot\)}\\[.2em]
        &\eq \ao \circ \ao \circ \big(\id \ot (\id \ot \ta^f_{u})\big) 
        	&\tag{naturality of \(\ao\)}\\[.2em]
        &\eq (\ao \ot \id) \circ \ao \circ (\id \ot \ao) 
        	\circ \big(\id \ot (\id \ot \ta^f_{u})\big) 
			&\tag*{\eqref{eq:P1}}\\[.2em]
        &\eq (\ao \ot \id) \circ \ao \circ (\id \ot \ta^f_{\ao u}) 
        	&\tag{definition of \(\ta^f_{\ao u}\)}\\[.2em]
        &\eq (\ao \ot \id) \circ \ta^f_{\ao \ao u}\,. 
        	&\tag{definition of \(\ta^f_{\ao\ao u}\)}
    \end{align}
    Equation~\eqref{eq: alphai alpha alpha} follows similarly:
    \begin{align}
        \ta^f_{\aoi \ao u} \circ \ao 
        &\eq \aoi \circ (\ta^f_{\ao u}\ot \id) \circ \ao 
        	&\tag{definition of \(\ta^f_{\aoi\ao u}\)}\\[.2em]
        &\eq \aoi \circ (\ao \ot \id) \circ \big((\id \ot \ta^f_u) \ot \id\big) \circ \ao 
        	&\tag{definition of \(\ta^f_{\ao u}\)}\\[.2em]
        &\eq \aoi \circ (\ao \ot \id) \circ \ao \circ \big(\id \ot (\ta^f_u \ot \id)\big) 
        	&\tag{naturality of \(\ao\)}\\[.2em]
        &\eq \ao \circ (\id \ot \aoi) \circ \big(\id \ot (\ta^f_u \ot \id)\big) 
        	&\tag*{\eqref{eq:P1}}\\[.2em]
        &\eq \ao \circ (\id \ot \ta^f_{\aoi u}) 
        	&\tag{definition of \(\ta^f_{\aoi u}\)}\\[.2em]
        &\eq \ta^f_{\ao \aoi u}\,. &\tag{definition of \(\ta^f_{\ao \aoi u}\)}
    \end{align}
    Equation~\eqref{eq: alphai alphai alpha} is proved similarly:
    \begin{align}
        \ta^f_{\aoi \aoi u} \circ \ao 
        &\eq \aoi \circ (\ta^f_{\aoi u}\ot \id) \circ \ao 
        	&\tag{definition of \(\ta^f_{\aoi \aoi u}\)}\\[.2em]
        &\eq \aoi \circ (\aoi \ot \id) \circ \big((\ta^f_u \ot \id) \ot \id\big) \circ \ao 
        	&\tag{definition of \(\ta^f_{\aoi u}\)}\\[.2em]
        &\eq \aoi \circ (\aoi \ot \id) \circ \ao \circ \big(\ta^f_u \ot (\id \ot \id)\big) 
        	&\tag{naturality of \(\ao\)}\\[.2em]
        &\eq (\id \ot \ao) \circ \aoi \circ \big(\ta^f_u \ot (\id \ot \id)\big) 
        	&\tag*{\eqref{eq:P1}}\\[.2em]
        &\eq (\id \ot \ao) \circ \aoi \circ (\ta^f_u \ot \id) 
        	&\tag{bifunctoriality of \(\ot\)}\\[.2em]
        &\eq (\id \ot \ao) \circ \ta^f_{\aoi u}\,. 
        	&\tag{definition of \(\ta^f_{\aoi u}\)}
    \end{align}
    
    Equation~\eqref{eq: alphai alphai} is obtained by precomposing and postcomposing 
    Equation~\eqref{eq: alphai alphai alpha} with \(\aoi\) and~\(\id \otimes \aoi\), 
    respectively. Equations~\eqref{eq: alpha alphai alphai} 
    and~\eqref{eq: alpha alpha alphai} follow analogously from 
    Equations~\eqref{eq: alphai alpha alpha} and~\eqref{eq: alpha alpha}.
\end{proof}

\begin{definition}\label{def:ot-analysable}
	Let \(\ff\) be a set of morphisms whose targets are \(\ot\)\kern0.1em-products. A morphism \(g\) in \(\cC\) is \emph{\(\ff\)-analysable} if it can be written 
	as \(g=(g'\ot g'')\circ\tau^f_u(\vect{A})\),
	for some morphisms \(g'\), \(g''\) in~\(\cC\), some morphism \(f\in\ff\), some word \(u\in\fY\), 
	and some vector \(\vect{A}\in\cC^{|u|}\).    
\end{definition}

The next result will be used twice in the next section, each time taking care of a large 
number of cases in a proof by induction.

\begin{lemma}\label{lemma: ot f good}
	 \sloppy Fix a set \(\ff\) of morphisms in \(\cC\) whose targets are \(\ot\)-products. Fix a non-empty word \(u\in\fY\). Let \(f\in \ff\), and let \(\vect{A}\in\cC^{|u|}\). Consider a morphism \(h\) in \(\cC\) with target \(H^f_u(\vect{A})\). Suppose that \(\tau^f_u(\vect{A})\circ h\) is not \(\ff\)-analysable.
	\begin{enumerate}[label=\rmlabel]
		\item\label{it:35a} Then the morphism \(h\) is neither a component of \(\distl\), \(\distr\), 
		\(\ap\) or  \(\api\), nor of the form \(\id\partimes h'\) or \(h'\partimes\id\) for any morphism \(h'\).
	\end{enumerate}
	For the remaining cases,
	\begin{enumerate}[resume,label=\rmlabel]
		\item\label{it:35b} if \(h\) is a component of \(\ao\) or \(\aoi\), then \(u=\aoi\) 
		or \(u=\ao\), respectively.
		\item\label{it:35c} If \(h=\id\ot h'\) for some morphism \(h'\), 
		then \(\ter(u)=\ao\), and the composite \({\tau^f_{\init(u)}(\vect{A}_{\ge 2})\circ h'}\) is not \(\ff\)-analysable. 
		\item\label{it:35d} If \(h={h' \ot \id}\) for some morphism \(h'\), 
		then \(\ter(u)=\aoi\), and the composite \(\tau^f_{\init(u)}(\vect{A}_{\le |u|-1})\circ h'\) is not \(\ff\)-analysable. 
	\end{enumerate}
\end{lemma}

\begin{proof}
	Part~\ref{it:35a} is immediate, since \(t(h)=H^f_u(\vect{A})\) is an \(\ot\)\kern0.1em-product, whereas the morphisms listed in~\ref{it:35a} have target a \(\partimes\)-product. By tidiness of \(\cC\), the claim follows.
	
	\smallskip
	
	For part~\ref{it:35b}, assume that \(h\) is a component of \(\ao\). Since the word \(u\) is non-empty, it ends with either \(\ao\) or \(\aoi\). If \(u\) ended with \(\ao\), \(\aoi\ao\), or \(\aoi\aoi\), then Equations~\eqref{eq: alpha alpha}, \eqref{eq: alphai alpha alpha}, and \eqref{eq: alphai alphai alpha} of Lemma~\ref{lemma: pushing tau through} would imply that \(\tau^f_u(\vect{A})\circ h\) is~\(\ff\)-analysable, contradicting our assumption. Hence \(u=\aoi\). The second claim in part~\ref{it:35b} follows similarly from Equations~\eqref{eq: alphai alphai}, \eqref{eq: alpha alphai alphai}, and \eqref{eq: alpha alpha alphai} in Lemma~\ref{lemma: pushing tau through}.
	
	\smallskip
	
	Proceeding with part~\ref{it:35c}, assume \(h=\id\ot h'\) for some morphism \(h'\)
	and set \(w:=\init(u)\).
	Suppose, for a contradiction, that \(\ter(u)=\aoi\). Omitting component indices, we have
	\begin{align}
		\tau^f_u \circ h 
		&\eq \aoi \circ (\tau^f_w\ot \id)\circ (\id \ot h')
		&\tag{\(u = \aoi w\)}\\[.2em]
		&\eq \aoi \circ \big((\id \ot \id)\ot h'\big) \circ (\tau^f_w \ot \id)
		&\tag{bifunctoriality of \(\ot\)}\\[.2em]
		&\eq \big(\id \ot (\id\ot h')\big) \circ \aoi \circ (\tau^f_w\ot \id)
		&\tag{naturality of \(\aoi\)}\\[.2em]
		&\eq \big(\id \ot (\id\ot h')\big) \circ \tau^f_u \,,
		&\tag{\(u = \aoi w\)}
	\end{align}
 	contradicting the assumption that \(\tau^f_u(\vect{A})\circ h\) is not \(\ff\)-analysable. Therefore \(\ter(u)=\ao\).
	
	Next, suppose for a contradiction that \(\tau^f_w(\vect{A}_{\ge 2})\circ h'\) is \(\ff\)-analysable. Then it can be written as \((g'\ot g'')\circ \tau_v^{f'}(\vect{B})\) for some \(f'\in\ff\), \(v\in\fY\), \(\vect{B}\in\cC^{|v|}\) and morphisms \(g'\), \(g''\). It follows that
	\begin{align*}
		\tau^f_u(\vect{A})\circ h
		&\eq \ao \circ \big(A_1 \ot \tau^f_w(\vect{A}_{\ge 2})\big)\circ (A_1\ot h') 
		&\tag{\(u = \ao w\) and \(h=\id\ot h'\)}\\[.2em]
		&\eq \ao \circ \big(A_1 \ot (\tau^f_w(\vect{A}_{\ge 2}) \circ h')\big)
		&\tag{bifunctoriality of \(\ot\)}\\[.2em]
		&\eq \ao\circ \big(A_1 \ot (g'\ot g'')\big)\circ\bigl(A_1\ot\ta^{f'}_{v}(\vect{B})\bigr) 
		&\tag{\(\ff\)-analysability and bifunctoriality}\\[.2em]
		&\eq \big((A_1 \ot g')\ot g''\big) \circ\ao\circ\bigl(A_1\ot\ta^{f'}_{v}(\vect{B})\bigr)
		&\tag{naturality of \(\ao\)}\\[.2em]
		&\eq \big((A_1 \ot g')\ot g''\big) \circ \ta^{f'}_{\ao v}(A_1, \vect{B}) \,,
		&\tag{definition of \(\ta^{f'}_{\ao v}\)}
	\end{align*}
	again contradicting the assumption that \(\tau^f_u(\vect{A})\circ h\) is not \(\ff\)-analysable.
	
	\smallskip
	
	Part~\ref{it:35d} is proved similarly.
\end{proof}

\subsection{Morphisms into \texorpdfstring{$\partimes$}{par}-products}\label{subs:parmor}
The analysis of morphisms with target a \(\prt\)-product is more involved since, by \eqref{eq:distl} and \eqref{eq:distr},
the target of both distributors is a \(\prt\)-product.

Let \(\fZ\) be the free monoid with generators \(\ap\), \(\api\), \(\distl\),~\(\distr\). For every morphism \(f\colon H\to L\partimes R\) with \(\partimes\)-product target and every \(u\in\fZ\), we define a family of natural transformations \(\bigl(\ka^f_u\bigr)_{u\in\fZ}\) by recursion on \(|u|\) as follows:

\begin{equation}\label{eq:kadef}
\begin{cases}
    \ka^f_{\e} 		 	&\hspace{-11pt}\eq f\,,\\[.1em]
    \ka^f_{\ap u} 	 &\hspace{-11pt}\eq \ap \circ ({\id} \partimes {\ka^f_u})\,,\\[.1em]
    \ka^f_{\api u} 	 &\hspace{-11pt}\eq \api \circ (\ka^f_u \partimes \id)\,,\\[.1em]
    \ka^f_{\distl u} &\hspace{-11pt}\eq \distl \circ (\id \ot \ka^f_u)\,,\\[.1em]
    \ka^f_{\distr u} &\hspace{-11pt}\eq \distr \circ (\ka^f_u \ot \id)\,.
\end{cases}
\end{equation}

For clarity, we specify the source and target functors of the natural transformations \eqref{eq:kadef}.

\begin{definition}\label{def: Hparf}
For every morphism \(f\colon H\longrightarrow L\partimes R\) and every \(u\in \fZ\), we define functors
\begin{equation}
    H^f_u, L^f_u, R^f_u\colon\; \cC^{\lvert u \rvert} \;\slongrightarrow\; \cC
\end{equation}
by recursion on \(|u|\). In the base case, we set
\begin{align}
    &\begin{cases}
        H^f_{\e}(\ast) &\hspace{-8pt}\eq H\,, \\[.1em]
        L^f_{\e}(\ast) &\hspace{-8pt}\eq L\,, \\[.1em]
        R^f_{\e}(\ast) &\hspace{-8pt}\eq R\,,
    \end{cases}
\end{align}
and in the recursive step, we set
\begin{align}
    &\begin{cases}
        H^f_{\ap u}(\vect{A}) &\hspace{-8pt}\eq A_1 \partimes H^f_{u}(\vect{A}_{\geq 2})\,, \\[.1em]
        L^f_{\ap u}(\vect{A}) &\hspace{-8pt}\eq A_1 \partimes L^f_{u}(\vect{A}_{\geq 2})\,,\\[.1em]
        R^f_{\ap u}(\vect{A}) &\hspace{-8pt}\eq R^f_{u}(\vect{A}_{\geq 2})\,,
    \end{cases}
	&&\begin{cases}
        H^f_{\api u}(\vect{A}) 
        &\hspace{-8pt} \eq H^f_{u}(\vect{A}_{\le\lvert u \rvert}) \partimes A_{\lvert u\rvert+1}\,,\\[.1em]
        L^f_{\api u}(\vect{A})
        &\hspace{-8pt} \eq L^f_{u}(\vect{A}_{\leq \lvert u \rvert})\,,\\[.1em]
        R^f_{\api u}(\vect{A}) 
        &\hspace{-8pt} \eq R^f_{u}(\vect{A}_{\leq \lvert u \rvert})\partimes A_{\lvert u \rvert+1}\,,
    \end{cases}
	\\[0.5cm]
    &\begin{cases}
        H^f_{\distl u}(\vect{A}) &\hspace{-8pt}\eq A_1 \ot H^f_{u}(\vect{A}_{\geq 2})\,,\\[.1em]
        L^f_{\distl u}(\vect{A}) &\hspace{-8pt}\eq A_1 \ot L^f_{u}(\vect{A}_{\geq 2})\,,\\[.1em]
        R^f_{\distl u}(\vect{A}) &\hspace{-8pt}\eq R^f_{u}(\vect{A}_{\geq 2})\,,
    \end{cases}
    &&\begin{cases}
        H^f_{\distr u}(\vect{A}) 
        &\hspace{-8pt} \eq H^f_{u}(\vect{A}_{\leq \lvert u \rvert})\ot A_{\lvert u \rvert+1}\,,\\[.1em]
        L^f_{\distr u}(\vect{A}) 
        &\hspace{-8pt} \eq L^f_{u}(\vect{A}_{\leq \lvert u \rvert})\,,\\[.1em]
        R^f_{\distr u}(\vect{A}) 
        &\hspace{-8pt} \eq R^f_{u}(\vect{A}_{\leq \lvert u \rvert})\ot A_{\lvert u \rvert+1}\,.
    \end{cases}
\end{align}
\end{definition}

With these definitions,~\eqref{eq:kadef} yields a natural transformation \(\ka^f_u\colon H^f_u \to L^f_u \partimes R^f_u\), for each word \(u\) and any morphism \(f\) with target a \(\partimes\)-product. Here,~\(L^f_u \partimes R^f_u\) denotes the pointwise \(\prt\)-tensor product of functors \(L^f_u\) and \(R^f_u\), defined analogously to Definition \ref{def:pointwise tensor product}.

The tidiness of \(\cC\) yields the following obvious analogue of Lemma~\ref{lemma: H is injective on objects}.

\begin{lemma}\label{lemma: Hpar is injective on objects}
    If for a word \(u \in \fZ\), two vectors \(\vect{A}, \vect{B}\in\cC^{|u|}\)
    and two morphisms \(f\), \(g\) into \(\partimes\)-products we 
    have \(H^f_u(\vect{A})= H^g_u(\vect{B})\), then \(\vect{A}=\vect{B}\) 
    and \(\dom(f)=\dom(g)\). \qed
\end{lemma}

The following result is proved analogously to Lemma~\ref{lemma: pushing tau through}, using the 
coherence axiom~\eqref{eq:P8} in place of \eqref{eq:P1}.

\begin{lemma}\label{lemma: pushing kappa through}
	For every morphism \(f\) into a \(\prt\)-product, the family of natural 
	transformation~\((\ka^f_u)_{u\in \fZ}\) satisfies the relations
    \begin{align}
    	\ka^f_{\ap u} \circ \ap 
        	&\eq (\ap \partimes {\id}) \circ \ka^f_{\ap \ap u}\,,\label{eq: ap ap}\\[.2em]
        \ka^f_{\api \ap u} \circ \ap 
        	&\eq \ka^f_{\ap \api u}\,,\label{eq: api ap ap}\\[.2em]
        \ka^f_{\api \api u} \circ \ap 
        	&\eq ({\id} \partimes \ap) \circ \ka^f_{\api u}\,,\label{eq: api api ap}\\[.2em]
        \ka^f_{\api u} \circ \api 
        &\eq ({\id} \partimes \api) \circ \ka^f_{\api \api u}\,,\label{eq: api api}\\[.2em]
        \ka^f_{\ap \api u} \circ \api 
        &\eq \ka^f_{\api \ap u}\,,\label{eq: ap api api}\\[.2em]
        \ka^f_{\ap \ap u} \circ \api 
        &\eq (\api \partimes {\id})\circ \ka^f_{\ap u}\,.\label{eq: ap ap api}
	\end{align}
\end{lemma}

The remaining six pentagon axioms,~\eqref{eq:P2}\,--\,\eqref{eq:P7}, determine what happens when a component of \(\ka^f_u\) is precomposed with a component of \(\distl\), \(\distr\), \(\ao\), or \(\aoi\):

\begin{lemma}\label{lemma: pushing kappa through2}
   	For every morphism \(f\) into a \(\prt\)-product, the family of natural 
	transformations~\((\ka^f_u)_{u\in \fZ}\) satisfies the relations
    \begin{align}
        \ka^f_{\ap u} \circ \distl 
        &\eq ({\distl} \partimes \id) \circ \ka^f_{\distl \ap u}\,,\label{eq: ap distl}\\[.2em]
        \ka^f_{\api \distl u} \circ \distl 
        &\eq \ka^f_{\distl \api u}\,,\label{eq: api distl distl}\\[.2em]
        \ka^f_{\api \distr u} \circ \distl 
        &\eq (\id \partimes \distl) \circ \ka^f_{\distr u}\,,\label{eq: api distr distl}\\[.2em]
		\ka^f_{\ap \distl u} \circ \distr 
        &\eq (\distr \partimes \id)\circ \ka^f_{\distl u}\,,\label{eq: ap distl distr}\\[.2em]
        \ka^f_{\ap \distr u} \circ \distr 
        &\eq \ka^f_{\distr \ap u}\,,\label{eq: ap distr distr}\\[.2em]
        \ka^f_{\api u} \circ \distr 
        &\eq (\id \partimes \distr) \circ \ka^f_{\distr \api u}\,,\label{eq: api distr}\\[.2em]
        \ka^f_{\distl u} \circ \ao 
        &\eq (\ao \partimes {\id}) \circ \ka^f_{\distl \distl u}\,,\label{eq: distl ao}\\[.2em]
        \ka^f_{\distr \distl u} \circ \ao 
        &\eq \ka^f_{\distl \distr u}\,,\label{eq: distr distl ao}\\[.2em]
        \ka^f_{\distr \distr u} \circ \ao 
        &\eq ({\id} \partimes \ao) \circ \ka^f_{\distr u}\,,\label{eq: distr distr ao}\\[.2em]
        \ka^f_{\distr u}\circ\aoi
        &\eq({\id} \partimes \aoi)\circ\ka^f_{\distr \distr u}\,,\\[.2em]
        \ka^f_{\distl \distl u}\circ\aoi
        &\eq (\aoi \partimes {\id})\circ\ka^f_{\distl u}\,,\\[.2em]
        \ka^f_{\distl \distr u}\circ\aoi
        &\eq\ka^f_{\distr \distl u}\,.
     \end{align}
\end{lemma}

\begin{proof}
	Equation~\eqref{eq: ap distl} is proved as follows:
    \begin{align}
        \ka^f_{\ap u} \circ \distl 
        &\eq \ap \circ (\id \partimes \ka^f_{u}) \circ \distl 
        	&\tag{definition of \(\ka^f\)}\\[.2em]
        &\eq \ap \circ \big((\id \ot \id) \partimes \ka^f_{u}\big) \circ \distl 
        	&\tag{bifunctoriality of \(\ot\)}\\[.2em]
        &\eq \ap \circ \distl \circ \big(\id \ot (\id \partimes \ka^f_{u})\big) 
        	&\tag{naturality of \(\distl\)}\\[.2em]
        &\eq (\distl \partimes \id) \circ \distl \circ (\id \ot \ap) 
        	\circ \big(\id \ot (\id \partimes \ka^f_{u})\big) 
			&\tag*{\eqref{eq:P4}}\\[.2em]
        &\eq (\distl \partimes \id) \circ \distl \circ (\id \ot \ka^f_{\ap u}) 
        	&\tag{definition of \(\ka^f\)}\\[.2em]
        &\eq (\distl \partimes \id) \circ \ka^f_{\distl \ap u}. 
        	&\tag{definition of \(\ka^f\)}
    \end{align}

Equation~\eqref{eq: api distl distl} is proved as follows:
    \begin{align}
        \ka^f_{\api \distl u} \circ \distl 
        &\eq \api \circ (\ka^f_{\distl u}\partimes \id) \circ \distl 
        	&\tag{definition of \(\ka^f\)}\\[.2em]
        &\eq \api \circ (\distl \partimes \id) 
        	\circ \big((\id \ot \ka^f_{u}) \partimes \id \big) \circ \distl 
			&\tag{definition of \(\ka^f\)}\\[.2em]
        &\eq \api \circ (\distl \partimes \id) \circ \distl 
        	\circ \big(\id \ot (\ka^f_{u} \partimes \id) \big) 
			&\tag{naturality of \(\distl\)}\\[.2em]
        &\eq \distl \circ (\id \ot \api) \circ \big(\id \ot (\ka^f_{u} \partimes \id) \big) 
        	&\tag*{\eqref{eq:P4}}\\[.2em]
        &\eq \distl \circ (\id \ot \ka^f_{\api u}) 
        	&\tag{definition of \(\ka^f\)}\\[.2em]
        &\eq \ka^f_{\distl \api u}.
        	&\tag{definition of \(\ka^f\)}
    \end{align}
Equation~\eqref{eq: api distr distl} is proved as follows:
    \begin{align}
        \ka^f_{\api \distr u} \circ \distl 
        &\eq \api \circ (\ka^f_{\distr u}\partimes \id) \circ \distl 
        	&\tag{definition of \(\ka^f\)}\\[.2em]
        &\eq \api \circ (\distr \partimes \id) 
        	\circ \big((\ka^f_u \ot \id) \partimes \id \big) \circ \distl 
			&\tag{definition of \(\ka^f\)}\\[.2em]
        &\eq \api \circ (\distr \partimes \id) \circ \distl 
        	\circ \big(\ka^f_u \ot (\id \partimes \id) \big) 
			&\tag{naturality of \(\distl\)}\\[.2em]
        &\eq (\id \partimes \distl) \circ \distr 
        	\circ \big(\ka^f_u \ot (\id \partimes \id) \big) 
			&\tag*{\eqref{eq:P6}}\\[.2em]
        &\eq (\id \partimes \distl) \circ \distr \circ (\ka^f_u \ot \id) 
        	&\tag{bifunctoriality of \(\prt\)}\\[.2em]
        &\eq (\id \partimes \distl) \circ \ka^f_{\distr u}.
        	&\tag{definition of \(\ka^f\)}
    \end{align}

Equation~\eqref{eq: ap distl distr} is proved as follows:
    \begin{align}
        \ka^f_{\ap \distl u} \circ \distr 
        &\eq \ap \circ (\id \partimes \ka^f_{\distl u}) \circ \distr 
        	&\tag{definition of \(\ka^f\)}\\[.2em]
        &\eq \ap \circ (\id \partimes \distl) 
        	\circ \big(\id \partimes (\id \ot \ka^f_u)\big) \circ \distr 
			&\tag{definition of \(\ka^f\)}\\[.2em]
        &\eq \ap \circ (\id \partimes \distl) \circ \distr 
        	\circ \big((\id \partimes \id) \ot \ka^f_u\big) 
			&\tag{naturality of \(\distr\)}\\[.2em]
         &\eq (\distr \partimes \id) \circ \distl 
         	\circ \big((\id \partimes \id) \ot \ka^f_u\big) 
			&\tag*{\eqref{eq:P6}}\\[.2em]
        &\eq (\distr \partimes \id) \circ \distl \circ (\id \ot \ka^f_u) 
        	&\tag{bifunctoriality of \(\partimes\)}\\[.2em]
        &\eq (\distr \partimes \id) \circ \ka^f_{\distl u}. 
        	&\tag{definition of \(\ka^f\)}
    \end{align}
Equation~\eqref{eq: ap distr distr} is proved as follows:
    \begin{align}
        \ka^f_{\ap \distr u} \circ \distr 
        &\eq \ap \circ (\id \partimes \ka^f_{\distr u}) \circ \distr 
        	&\tag{definition of \(\ka^f\)}\\[.2em]
        &\eq \ap \circ (\id \partimes \distr) 
        	\circ \big(\id \partimes (\ka^f_u \ot \id)\big) \circ \distr 
			&\tag{definition of \(\ka^f\)}\\[.2em]
        &\eq \ap \circ (\id \partimes \distr) \circ \distr 
        	\circ \big((\id \partimes \ka^f_u) \ot \id\big) 
			&\tag{naturality of \(\distr\)}\\[.2em]
        &\eq \distr \circ (\ap \ot \id) \circ \big((\id \partimes \ka^f_u) \ot \id\big) 
        	&\tag*{\eqref{eq:P7}}\\[.2em]
        &\eq \distr \circ (\ka^f_{\ap u} \ot \id) 
        	&\tag{definition of \(\ka^f\)}\\[.2em]
        &\eq \ka^f_{\distr \ap u}. 
        	&\tag{definition of \(\ka^f\)}
    \end{align}
Equation~\eqref{eq: api distr} is proved as follows:
    \begin{align}
        \ka^f_{\api u} \circ \distr 
        &\eq \api \circ (\ka^f_{u}\partimes \id) \circ \distr 
        	&\tag{definition of \(\ka^f\)}\\[.2em]
        &\eq \api \circ \big(\ka^f_{u}\partimes (\id \ot \id)\big) \circ \distr 
        	&\tag{bifunctoriality of \(\ot\)}\\[.2em]
        &\eq \api \circ \distr \circ \big((\ka^f_{u}\partimes \id) \ot \id\big) 
        	&\tag{naturality of \(\distr\)}\\[.2em]
        &\eq (\id \partimes \distr) \circ \distr \circ (\api \ot \id) 
        	\circ \big((\ka^f_{u}\partimes \id) \ot \id\big) 
			&\tag*{\eqref{eq:P7}}\\[.2em]
        &\eq (\id \partimes \distr) \circ \distr 
        	\circ (\ka^f_{\api u}\ot \id)
			&\tag{definition of \(\ka^f\)}\\[.2em]
        &\eq (\id \partimes \distr) \circ \ka^f_{\distr \api u}. 
        	&\tag{definition of \(\ka^f\)}
    \end{align}
     Equation~\eqref{eq: distl ao} is proved as follows:
    \begin{align}
        \ka^f_{\distl u} \circ \ao 
        &\eq \distl \circ (\id \ot \ka^f_{u}) \circ \ao 
        	&\tag{definition of \(\ka^f\)}\\[.2em]
        &\eq \distl \circ \big((\id \ot \id) \ot \ka^f_{u}\big) \circ \ao 
        	&\tag{bifunctoriality of \(\ot\)}\\[.2em]
        &\eq \distl \circ \ao \circ \big(\id \ot (\id \ot \ka^f_{u})\big) 
        	&\tag{naturality of \(\ao\)}\\[.2em]
        &\eq (\ao \partimes \id) \circ \distl \circ (\id \ot \distl) 
        	\circ \big(\id \ot (\id \ot \ka^f_{u})\big) 
			&\tag*{\eqref{eq:P2}}\\[.2em]
        &\eq (\ao \partimes \id) \circ \distl \circ (\id \ot \ka^f_{\distl u}) 
        	&\tag{definition of \(\ka^f\)}\\[.2em]
        &\eq (\ao \partimes \id) \circ \ka^f_{\distl \distl u}. 
        	&\tag{definition of \(\ka^f\)}
    \end{align}
Equation~\eqref{eq: distr distl ao} is shown similarly:
    \begin{align}
        \ka^f_{\distr \distl u} \circ \ao 
        &\eq \distr \circ (\ka^f_{\distl u}\ot \id) \circ \ao 
        	&\tag{definition of \(\ka^f\)}\\[.2em]
        &\eq \distr \circ (\distl \ot \id) \circ \big((\id \ot \ka^f_u) \ot \id\big) 
        	\circ \ao 
			&\tag{definition of \(\ka^f\)}\\[.2em]
        &\eq \distr \circ (\distl \ot \id) \circ \ao 
        	\circ \big(\id \ot (\ka^f_u \ot \id)\big) 
			&\tag{naturality of \(\ao\)}\\[.2em]
        &\eq \distl \circ (\id \ot \distr) \circ \big(\id \ot (\ka^f_u \ot \id)\big) 
        	&\tag*{\eqref{eq:P3}}\\[.2em]
        &\eq \distl \circ (\id \ot \ka^f_{\distr u}) 
        	&\tag{definition of \(\ka^f\)}\\[.2em]
        &\eq \ka^f_{\distl \distr u}. 
        	&\tag{definition of \(\ka^f\)}
    \end{align}
Equation~\eqref{eq: distr distr ao} is proved as follows:
    \begin{align}
        \ka^f_{\distr \distr u} \circ \ao 
        &\eq \distr \circ (\ka^f_{\distr u} \ot \id) \circ \ao 
        	&\tag{definition of \(\ka^f\)}\\[.2em]
        &\eq \distr \circ (\distr \ot \id) 
        	\circ \big((\ka^f_{u}\ot \id) \ot \id\big) \circ \ao 
			&\tag{definition of \(\ka^f\)}\\[.2em]
        &\eq \distr \circ (\distr \ot \id) \circ \ao 
        	\circ \big(\ka^f_{u}\ot (\id \ot \id)\big) 
			&\tag{naturality of \(\ao\)}\\[.2em]
        &\eq (\id \partimes \ao) \circ \distr \circ \big(\ka^f_{u}\ot (\id \ot \id)\big) 
        	&\tag*{\eqref{eq:P5}}\\[.2em]
        &\eq (\id \partimes \ao) \circ \distr \circ (\ka^f_{u}\ot \id) 
        	&\tag{bifunctoriality of \(\ot\)}\\[.2em]
        &\eq (\id \partimes \ao) \circ (\ka^f_{\distr u}\ot \id). 
        	&\tag{definition of \(\ka^f\)}
    \end{align}
    
    The last three equations follow from~\eqref{eq: distl ao}\,--\,\eqref{eq: distr distr ao} 
    by precomposing with \(\aoi\) and possibly postcomposing with \(\id\partimes\aoi\) 
    or \(\aoi\partimes\id\).
\end{proof}

The following definition is the \(\partimes\)-analogue of Definition~\ref{def:ot-analysable}.
\begin{definition}\label{def:par-analysable}
	Let \(\ff\) be a set of morphisms whose targets are \(\partimes\)-products. A morphism \(g\) in \(\cC\) is \emph{\(\ff\)-analysable} if it can be written 
	as \(g=(g'\partimes g'')\circ\ka^f_u(\vect{A})\),
	for some morphisms \(g'\), \(g''\) in~\(\cC\), some morphism \(f\in\ff\), some word \(u\in\fY\), 
	and some vector \(\vect{A}\in\cC^{|u|}\).    
\end{definition}

\begin{remark}
	We distinguish two notions of \(\ff\)-analysability, depending on whether the targets of the morphisms in \(\ff\) are all \(\partimes\)-products (Definition \ref{def:par-analysable}) or all \(\otimes\)\kern0.1em-products
	 (Definition \ref{def:ot-analysable}).
\end{remark}

The following lemma plays a similar rôle to Lemma~\ref{lemma: ot f good}.
	
\begin{lemma}\label{lemma: induction step par}    
     \sloppy Fix a set \(\ff\) of morphisms in \(\cC\) whose targets are \(\partimes\)-products. Fix a non-empty word \(u\in\fZ\). Let  \(f\in \ff\), and let \(\vect{A}\in\cC^{|u|}\). Consider a morphism \(h\) in \(\cC\) with target \(H^f_u(\vect{A})\). Suppose that \(\ka^f_u(\vect{A})\circ h\) is not \(\ff\)-analysable.
     
    \begin{enumerate}[label=\rmlabel]
    	\item\label{it:312a} If $h$ is a component of 
    	\(\ao, \aoi, \distl, \distr, \ap,\) or \(\api\,,\)
    	then $u=\distr, \distl, \api, \ap, \api,$ or \(\ap\), respectively. 
    	Moreover, in the first four cases, \(\dom(f)\) is an \(\ot\)-product, 
    	while in the last two cases \(\dom(f)\) is a \(\prt\)-product.
		\item\label{it:312b} Let $h'$ be a morphism in \(\cC\). 
			\begin{center}
			\begin{tabular}{c|c|c}
			If \(h=\)& then \(\ter(u)=\) & and \\  
			\hline \noalign{\vskip 0.6ex}
			\(\id\ot h'\) & \(\distl\) & \(\ka^f_{\init(u)}(\vect{A}_{\ge 2})\circ h'\) \\[.7ex]
			\(h' \ot\id\) & \(\distr\) & \(\ka^f_{\init(u)}(\vect{A}_{\le |u|-1})\circ h'\) \\[.7ex]
			\(\id\partimes h'\) & \(\ap\) & \(\ka^f_{\init(u)}(\vect{A}_{\ge 2})\circ h'\) \\[.7ex]
			\(h'\partimes\id\) & \(\api\) & \(\ka^f_{\init(u)}(\vect{A}_{\le |u|-1})\circ h'\)
			\end{tabular}
			\end{center}
			is not \(\ff\)-analysable.
	\end{enumerate}
\end{lemma}

\begin{proof}
	For part~\ref{it:312a}, assume that \(h\) is a component of \(\ao\). Then, by definition of \(\ao\), there are objects~\(X\),~\(Y\), and~\(Z\) of \(\cC\) such that 
	\[
		(X\ot Y)\ot Z\eq\codom(h)\eq H^f_u(\vect{A})\,,
	\]
	where the last equality holds by assumption.
	By unraveling Definition \ref{def: Hparf} of \(H^f_u\), we see that either there exists \(w\in\fZ\) such that \(u\in\{\distl w, \distr \distl w, \distr \distr w\}\), 
	or \(u=\distr\) and \(\dom(f)\) is an \(\ot\)\kern0.1em-product. In the former case, 
	one of Equations~\eqref{eq: distl ao}\,--\,\eqref{eq: distr distr ao} would imply that \(\ka^f_u(\vect{A})\circ h\) is \(\ff\)-analysable, which contradicts our assumption.
	Thus, only the latter case can hold. The remaining five claims in~\ref{it:312a} are verified similarly, using the other fifteen equations in Lemmas~\ref{lemma: pushing kappa through} and~\ref{lemma: pushing kappa through2}. 

	\smallskip

	Proceeding with part~\ref{it:312b}, assume \(h=\id\ot h'\) for some morphism \(h'\) and set \(w:=\init(u)\). Then \(H^f_u(\vect{A})\) is again an \(\ot\)\kern0.1em-product, so that \(u\) ends either with \(\distr\) or with \(\distl\). Suppose, for a contradiction, that \(\ter(u)=\distr\). Omitting component indices, we have
	\begin{align}
        \ka^f_u \circ h
        &\eq \distr \circ \bigl(\ka^f_{w}\ot\id\bigr) \circ (\id\ot h') 
        	\tag{$u=\distr w$}\\[0.2em]
        &\eq \distr \circ \bigl((\id \partimes \id)\ot h'\bigr)
        	\circ (\ka^f_{w}\ot\id) 
			\tag{bifunctoriality of \(\ot\) and \(\prt\)}\\[0.2em]
        &\eq \bigl(\id \partimes {(\id\ot h')}\bigr) \circ \distr \circ 
        	(\ka^f_{w}\ot\id)
			\tag{naturality of $\distr$}\\[0.2em]
        &\eq \bigl(\id \partimes {(\id\ot h')}\bigr) \circ \ka^f_u\,, 
        	\tag{$u=\distr w$}
    \end{align}
	contradicting the assumption that $\ka^f_u \circ h$ is not \(\ff\)-analysable. Therefore \(\ter(u)=\distl\). 
	
	The three remaining rows in the 
	second column in part~\ref{it:312b} are treated analogously. In each case, our assumption on \(h\) determines whether \(H^f_u(\vect{A})\) is an~\(\ot\)\kern0.1em-product or a \(\partimes\)-product, leaving only two possibilities for \(\ter(u)\). One possibility can then be eliminated by a similar computation. 
	
	Returning to the case \(h=\id\ot h'\), suppose for a contradiction that \(\ka^f_w(\vect{A}_{\ge 2})\circ h'\) is \(\ff\)-analysable. Then it can be written as \((g'\prt g'')\circ \ka_v^{f'}(\vect{B})\) for some \(f'\in\ff\), \(v\in\fZ\), \(\vect{B}\in\cC^{|v|}\) and morphisms \(g'\), \(g''\). It follows that
	
	\begin{align}
      	\ka^f_u(\vect{A}) \circ h 
      	& \eq \distl \circ \bigl(A_1\ot \ka_{w}^f(\vect{A}_{\ge 2})\bigr)\circ (A_1\ot h') 
        \tag{\(u=\distl w\) and \(h=\id\ot h'\)}\\[0.2em]
        & \eq \distl \circ \bigl(A_1\ot (\ka_{w}^f(\vect{A}_{\ge 2})\circ h')\bigr)
        \tag{bifunctoriality of \(\ot\)}\\[0.2em]
        & \eq \distl \circ \bigl(A_1 \ot (g' \partimes g'')\bigr)
        	\circ \bigl(A_1 \ot \ka_{v}^{f'}(\vect{B})\bigr)
        \tag{\(\ff\)-analysability and bifunctoriality}\\[0.2em]
        & \eq \bigl((A_1 \ot g') \partimes g''\bigr)\circ \distl 
        	\circ \bigl(A_1 \ot \ka_{v}^{f'}(\vect{B})\bigr)
        	\tag{naturality of $\distl$}\\[0.2em]
        & \eq \big((A_1 \ot g')\partimes g''\big)\circ \ka^{f'}_{\distl v}(A_1, \vect{B}) \,,
        	\tag{definition of $\ka^{f'}_{\distl v}$}
    \end{align}
	again contradicting the assumption that \(\ka^{f}_u(\vect{A}) \circ h\) is not \(\ff\)-analysable. The three remaining rows in the 
	third column in part~\ref{it:312b} are treated analogously.
\end{proof}

\section{Normal form of morphisms}\label{sec:normform}
We now have the technical lemmas at our disposal to study the analysability of morphisms in tidy LD categories.
\subsection{LD categories}\label{sec: normal form LD cats}

 In this subsection, this will allow us establish normal form results (Proposition~\ref{prop:normal cat}) for morphisms in a tidy LD category \(\cC\), recursively built from identities and coherence morphisms via $\ot$, $\partimes$, and $\circ$. 
 
\begin{definition}\label{def atomar}
	The class of \emph{atomic} morphisms in a tidy LD category \(\cC\) is the smallest class of morphisms satisfying the following conditions:
	\begin{enumerate}[label=\alabel]
		\item It contains all components of the coherence morphisms \(\ao\), \(\aoi\), \(\ap\), 
		\(\api\), \(\distl\), and \(\distr\).
		\item It is closed under tensoring with identities. That is, for every \(A \in \cC\) and every atomic~\(f\), the morphisms \(f \ot \id_A\), \(\id_A \ot f\), \(f \partimes \id_A\), and \(\id_A \partimes f\) are also atomic.
	\end{enumerate}
\end{definition}

\begin{definition}\label{def elementary}
A morphism in a tidy LD category \(\cC\) is \emph{elementary} if it is an identity or a finite composite of atomic morphisms.
\end{definition}

\begin{remark}
	In our model of the free LD category on a set, as given in Construction \ref{con: free LD cat}, every morphism is elementary.
\end{remark}

\begin{remark}
	The class of elementary morphisms is the smallest class of morphisms such 
	that \(\Ob(\cC)\), together with these morphisms, forms a LD subcategory of \(\cC\).
\end{remark}

\begin{definition}
	We denote by \(\Idd^{\ot}_{\cC}\) and \(\Idd^{\partimes}_{\cC}\) the classes 
of identities of \(\ot\)\kern0.1em-products and \mbox{\(\partimes\)-products}, respectively. 
When the category \(\cC\) is clear from context, we omit the subscript. 
\end{definition}
 		
 		Fix a tidy LD category \(\cC\).
 		
\begin{lemma}\label{lemma idot-analysis}
	Let \(\fY\) be the free monoid with generators \(\ao, \aoi\), as defined in Subsection \ref{subs:morphot}. Let \(u\in\fY\),
	\(\vect{A}\in\cC^{|u|}\), and \(i\in \Idd^{\otimes}\). For any atomic morphism \(h\) with target \(H^{i}_u(\vect{A})\), the composite \(\tau^i_u(\vect{A})\circ h\) is \(\Idd^{\ot}\)-analysable. 
\end{lemma}

\begin{proof}
	We argue by induction on $|u|$. Let \(B,C\in \cC\) be the unique objects such that \(i=\id_{B\ot C}\). 

\medskip
	
  \noindent \textbfit{Base case.} For \(u = \e\), we have \(\tau^i_u(\vect{A})=\id_{B\ot C}\) by definition. It thus remains to show that every atomic 
  morphism \(h\) with target \(B\ot C\) is \(\Idd^{\ot}\)-analysable.
  
  This requires attention only if \(h\) is a component of \(\ao\) or \(\aoi\).  
  In the first case, \(B\) is an \(\ot\)\kern0.1em-product, say \(B=B'\ot B''\). 
  Then \(h=\ao_{B', B'', C}=\tau_\ao^{f}(B')\), where \(f=\id_{B''\ot C}\). Thus, \(h\) is indeed \(\Idd^{\ot}\)-analysable. In the second case, we similarly
  have \(h=\aoi_{B, C', C''}=\tau_{\aoi}^{f}(C'')\), where \(f=\id_{B\ot C'}\) and \(C',C''\in \cC\) are the unique objects such that \(C=C'\ot C''\). 
  
	\medskip

	\noindent \textbfit{Inductive step.}
	Let \(u\in \fY\sm\{\e\}\). For a contradiction, suppose that \(\tau^i_u(\vect{A})\circ h\) is not \(\Idd^{\ot}\)-analysable. Then \(h\) cannot be of the form \(h=\id\ot h'\) or \(h=h'\ot \id\) for any morphism \(h'\), as this would contradict the induction hypothesis by Lemma~\ref{lemma: ot f good}\ref{it:35c}\,--\,\ref{it:35d}. By Lemma~\ref{lemma: ot f good}\ref{it:35a}\,--\,\ref{it:35b}, \(h\) is thus either a component of \(\ao\) and \(u=\aoi\), or vice versa. In the first case, 
	\begin{align*}
		\tau^i_{\aoi}(A)\circ \ao
		&=
		\aoi\circ \id_{(B\ot C)\ot A}\circ \ao 
		=
		\id_{B\ot (C\ot A)}
		=
		\tau_{\e}^{\id_{B\ot (C\ot A)}}(\ast)
	\end{align*}
	contradicting that \(\tau^i_u(A) \circ h\) is not \(\Idd^{\ot}\)-analysable. The second case is treated similarly.
\end{proof}

The analogous statement for \(\partimes\)-products can be derived similarly from 
Lemma~\ref{lemma: induction step par}.

\begin{lemma}\label{lemma idpar-analysis}
		Let \(\fZ\) be the free monoid with generators \(\ap\), \(\api\), \(\distl\),~\(\distr\), as in Subsection \ref{subs:parmor}. Let \(u\in\fZ\), \(\vect{A}\in\cC^{|u|}\), and \(i\in \Idd^{\partimes}\). For any atomic morphism \(h\) with target \(H^{i}_u(\vect{A})\), the composite \(\ka^i_u(\vect{A})\circ h\) is \(\Idd^{\prt}\)-analysable. 
\end{lemma}

\begin{proof}
	Writing \(i=\id_{B\prt C}\) we proceed by induction on \(|u|\).
	
	\medskip
	
    \noindent \textbfit{Base case.} 
    We need to show that every atomic morphism \(h\) with target \(B\prt C\) 
    is \mbox{\(\Idd^{\prt}\)-analysable}. If \(h\) is a component of \(\distl\), \(\distr\), \(\ap\), 
    or \(\api\), we argue as in the base case of Lemma~\ref{lemma idot-analysis}. The 
    remaining case is clear. 
    
    \medskip

	\noindent \textbfit{Inductive step.}
	By Lemma~\ref{lemma: induction step par}\ref{it:312b}, \(h\) has to be a component of
	\(\ao\), \(\aoi\), \(\distl\), \(\distr\), \(\ap\), or \(\api\). Since the source~\(\dom(i)\)
	is a \(\prt\)-product, Lemma~\ref{lemma: induction step par}\ref{it:312a} 
	rules out the first four cases. So either \(h\) is a component of \(\ap\) and \(u=\api\),
	or vice versa. Now we can complete the argument as in the inductive step 
	of Lemma~\ref{lemma idot-analysis}.  
\end{proof}

With these preparations, the desired normal form result, phrased in terms of \(\ff\)-analysability, follows. 

\begin{proposition}\label{prop:normal cat}
	Let \(f\) be an elementary morphism in a tidy LD category \(\cC\). 
	\begin{enumerate}[label=\rmlabel]
		\item\label{it:normal cat 1} If \(\codom(f)\) is neither an \(\ot\)-product 
			nor a \(\partimes\)-product, then \(f\) is an identity. 
		\item\label{it:normal cat 2} If \(\codom(f)\) is an \(\ot\)-product, 
			then \(f\) is \(\Idd^{\ot}\)-analysable.
		\item\label{it:normal cat 3}   If \(\codom(f)\) is a \(\prt\)-product, 
			then \(f\) is \(\Idd^{\partimes}\)-analysable.
	\end{enumerate}
\end{proposition} 

\begin{proof}
	Part~\ref{it:normal cat 1} is a straightforward consequence of the fact that the target of every atomic morphism is either an \(\ot\)\kern0.1em-product or a \(\prt\)-product.
	
	\smallskip
	
	We define the \emph{height} of an elementary morphism~\(f\) 
	in \(\cC\) as the minimal number \(h(f)\in \NNez\) of atomic morphisms composing 
	to \(f\); identities have height zero.
	Parts~\ref{it:normal cat 2} and~\ref{it:normal cat 3} are proved by induction on the height \(n := h(f)\). 
	
	\smallskip
	
	We now present the argument for~\ref{it:normal cat 2}.
	Let \(\codom(f)=A\ot B\).
    
    \smallskip
    
    \noindent\textbfit{Base case.}
     If \(n=0\), then \(f=\id_{A\otimes B}=(\id_A\ot \id_B)\circ \tau^f_{\e}(\ast)\)
     is indeed \(\Idd^{\ot}\)-analysable. 
    \smallskip
    
    \noindent\textbfit{Inductive step.} 
    Let \(f\) be an elementary morphism of height \(n\ge 1\). Write
    \begin{equation}\label{eq: f has length n}
    	f \eq f_n\circ \dots \circ {f_1}
   \end{equation} 
   for atomic morphisms \(f_n, \dots, f_1\). 
   
 Since \(h(f_n \circ \cdots \circ f_2) \leq n-1\), the induction hypothesis yields 
  an identity \(j\in\Idd^{\ot}\), a word \(v \in \fY\), 
  a vector \(\vect{B}\in \cC^{\lvert v \rvert}\), and morphisms \(g', h'\), such that
  \begin{align}\label{eq: factorizing 2 to n}
  	f_n\circ \dots \circ f_2 
  	\eq (g'\ot h')\circ \ta^{j}_{v}(\vect{B})\,.
  \end{align}
  Applying Lemma~\ref{lemma idot-analysis}, we obtain an identity \(i\in\Idd^{\ot}\), 
  a word \(u\in \fY\), a vector \(\vect{A} \in \cC^{\lvert u \rvert}\), and morphisms \(g, h\), such that
  \begin{align}\label{eq: otimes factorisation induction step}
  	\ta^{j}_{v}(\vect{B}) \circ f_1 
  	\eq (g \ot h) \circ \ta_u^{i}(\vect{A})\,.
  \end{align}
  Substituting the appropriate factorisations,
  we obtain
  \begin{align}
  	f &\eq f_n\circ \dots \circ f_1
  	&\tag{factorisation \eqref{eq: f has length n}}\\[.2em]
  	&\eq (g'\ot h')\circ \ta_{v}^{j}(\vect{B}) \circ f_1
  	&\tag{factorisation \eqref{eq: factorizing 2 to n}}\\[.2em]
  	&\eq (g'\ot h')\circ (g\ot h)\circ \ta_u^i(\vect{A})
  	&\tag{factorisation \eqref{eq: otimes factorisation induction step}}\\[.2em] 
  	&\eq \big((g'\circ g)\ot (h'\circ h)\big)\circ \ta_u^i(\vect{A})\,.
  	&\tag{bifunctoriality of \(\ot\)}
  \end{align}
  This completes the induction and the proof of~\ref{it:normal cat 2}. The only 
  changes required for proving~\ref{it:normal cat 3} are that one has to replace 
  \(\ot\), \(\ta\), and \(\fY\) by \(\prt\), \(\ka\), and \(\fZ\), respectively. 
  Moreover, one has to appeal to Lemma~\ref{lemma idpar-analysis} rather than 
  Lemma~\ref{lemma idot-analysis}.
\end{proof}

\begin{remark}	
	Proposition~\ref{prop:normal cat} has an obvious analogue in tidy
	lax monoidal and tidy lax LD categories. 
\end{remark}

\subsection{Frobenius LD functors}\label{subsec:normFrob}
Throughout this subsection, we fix a tidy Frobenius LD functor \(F\colon \cC \to \cD\) between tidy LD categories, with coherence morphisms~\(\mu\) and~\(\Delta\) as in Definition~\ref{def: Vorsicht Funktor}. To establish a normal form result for \(F\) (Proposition~\ref{prop normal F}), we adapt the definitions of atomic and elementary morphisms from Definitions~\ref{def atomar} and~\ref{def elementary}.

\begin{definition}\label{def F atomic}
	The class of \emph{\(F\)-atomic} morphisms is the smallest class of morphisms 
	in \(\cD\) satisfying the following conditions:
	\begin{enumerate}[label=\alabel]
		\item It contains all components of the coherence morphisms \(\ao\), \(\aoi\), \(\ap\), 
		\(\api\), \(\distl\), \(\distr\), \(\mu\), \(\Delta\).
		\item It contains all morphisms of the form \(F(f)\), where \(f\) is an elementary morphism in \(\cC\). 
		\item It is closed under tensoring with identities.
	\end{enumerate}
\end{definition} 

\begin{definition}
A morphism in \(\cD\) is \emph{\(F\)-elementary} if it is an identity or a finite composite of \(F\)-atomic morphisms.
\end{definition}

\begin{remark}
	In our model of the free Frobenius LD functor \(F=\FreeF\) on a set \(S\), as given in Construction \ref{con: free Frob LD fun}, every morphism in the target category of \(F\) is \(F\)-elementary.
\end{remark}

Again we will find a normal form for \(F\)-elementary morphisms. 
The first case we consider has no parallel in the previous subsection: there are non-identity morphisms into objects of the form~\(F(X)\) even though these objects cannot be expressed as tensor products. Our analysis of these morphisms involves the following notion. 

\begin{definition}\label{def multiplicative} 
	Given a tidy Frobenius LD-functor \(F\colon \cC\to \cD\) between tidy LD categories, we define, by recursion on \(n\in\NNez\), the class of morphisms \(\fm_n\) in \(\cD\) as follows:
	\begin{align*}
		\fm_0&=\bigl\{\id_{F(X)}\:\colon\: X\in \cC\bigr\}\,, \\
		\fm_n&=\bigl\{\mu_{X, Y}\circ(f\ot g) \:\colon\: X,\: Y\in \cC,\: \codom(f)=F(X),\: \codom(g)=F(Y)
			\text{ and } f, g\in \bigcup_{m<n}\fm_m\bigr\}\,.
	\end{align*}
	The morphisms in the union \(\bigcup_{n\in\NNez}\fm_n\) are called \emph{multiplicative} morphisms.
\end{definition}

For instance, because of \(\mu_{X, Y}=\mu_{X, Y}\circ (\id_{F(X)}\ot \id_{F(Y)})\)
all components of \(\mu\) are in \(\fm_1\). The next lemma uses the hexagon axiom~\eqref{eq:H1}.  

\begin{lemma}\label{lem:F(X)}
Let \(f\) be a multiplicative morphism. For any \(F\)-atomic morphism \(h\) with \(\codom(h)=\dom(f)\), the composite \(f\circ h\) can be written in the form \(F(g)\circ k\), where \(g\) is an elementary morphism of~\(\cC\) and \(k\) is a multiplicative morphism.
\end{lemma}

\begin{proof}
	We proceed by induction on \(n\in\NNez\) to show that the statement holds for all \(f\in\fm_n\). 
	
	\medskip
	
    \noindent \textbfit{Base case.} If \(n=0\), then by Definition \ref{def multiplicative}, \(f=\id_{F(X)}\) for some \(X\in \cC\). In particular, \(\codom(h)=F(X)\). By the tidiness of \(F\), the target \(\codom(h)\) is thus neither an \(\ot\)\kern0.1em-product nor a \(\partimes\)-product. Since \(h\) is \(F\)-atomic, it follows that either \(h=\mu_{X', X''}\) for objects \(X',X''\in \cC\) with \(X=X'\ot X''\), or \(h=F(g')\) for some elementary morphism \(g'\) in \(\cC\). In the first case, set \(g:=\id_X\) and \(k:=\mu_{X', X''}\). In the second, set \(g:=g'\) and \(k:=\id_{F(X)}\). In both cases, \(f\circ h=F(g)\circ k\), as required.
    
    \medskip

	\noindent \textbfit{Inductive step.} Let \(n\geq 1\) and \(f\in\fm_n\). By definition, we can write
	\begin{equation}\label{eq: multiplicativity factorisation}
	f \eq \mu_{X', X''}\circ (f'\ot f'') \,,
	\end{equation}
	for objects \(X',X''\in \cC\) and multiplicative morphisms \(f', f''\in\bigcup_{m<n}\fm_m\). Since 
	\begin{equation}\label{eq codomh}
		\codom(h)\eq\dom(f)\eq\dom(f')\ot \dom(f'')
	\end{equation}
	is an \(\ot\)\kern0.1em-product and \(h\) is \(F\)-atomic, it follows that \(h\) is either of the form \(\id\ot h'\) or \(h'\ot \id\) 
	for some \(F\)-atomic morphism \(h'\), or \(h\) is a component of \(\ao\) or~\(\aoi\). 
	
	\smallskip
	
	{\it \hspace{2em} First case: \(h=h'\ot h''\), where one of \(h'\), \(h''\) is an identity and the other is \(F\)-atomic.}
	
	\smallskip
	
	If \(h'\) is \(F\)-atomic, the induction hypothesis yields an elementary morphisms \(g'\) in~\(\cC\) and a multiplicative morphism~\(k'\) such that
	\(f'\circ h'=F(g')\circ k'\). If \(h'\) is an identity morphism, set \(g':=\id_{X'}\) and \(k':=h'\); this gives the desired factorisation \(f'\circ h'=F(g')\circ k'\).
	
	Analogously, there exist an elementary morphism \(g''\) from~\(\cC\) and a 
	multiplicative morphism~\(k''\) such that \(f''\circ h''=F(g'')\circ k''\). 
	Altogether, we find 
	\begin{align*}
		f\circ h
		&\eq \mu_{X', X''}\circ (f'\ot f'')\circ (h'\ot h'') 
		\tag{factorisation \eqref{eq: multiplicativity factorisation} and \(h=h'\ot h''\)}\\[.2ex]
		&\eq \mu_{X', X''}\circ \bigl(F(g')\ot F(g'')\bigr)\circ (k'\ot k'') 
		\tag{factorising \(f'\circ h'\) and \(f''\circ h''\)}
		\\[.2ex]
		&\eq F(g'\ot g'')\circ \mu_{\dom(g'), \dom(g'')}\circ (k'\ot k'')\,.
		\tag{naturality of \(\mu\)}
	\end{align*}
	Since \(g'\ot g''\) is an elementary morphism in \(\cC\) and \(\mu_{\dom(g'), \dom(g'')}\circ (k'\ot k'')\)
	is multiplicative, we have thereby found the desired factorisation of \(f\circ h\). 
	
	\smallskip
	
	{\it \hspace{2em} Second case: \(h\) is a component of \(\ao\) or \(\aoi\).}
	
	\smallskip
	
	We treat only the case where \(h=\ao_{A, B, C}\) is a component of \(\ao\). Equation \eqref{eq codomh}, together with the tidiness of \(\cD\), then imply that \(\dom(f')=A\ot B\) is an \(\ot\)\kern0.1em-product. Since \(F\) is tidy, it follows that \(f'\not\in\fm_0\).
	Consequently, there exist objects \(Y', Y''\in \cC\) and multiplicative
	morphisms \(k', k''\in\bigcup_{m<n}\fm_m\) such that 
	\begin{equation}\label{eq: fact multiplicativity f'}
			f' \eq \mu_{Y', Y''}\circ (k'\ot k'').
	\end{equation}
	Now, the following calculation 
	\begin{align}
		f\circ h
		&\eq \mu_{X', X''}\circ (f'\ot f'')\circ \ao_{A, B, C} 
			\tag{factorisation \eqref{eq: multiplicativity factorisation} and \(h=\ao_{A,B,C}\)}\\[.2ex]
		&\eq \mu_{X', X''}\circ \bigl((\mu_{Y', Y''}\circ (k'\ot k''))\ot f''\bigr)\circ \ao_{A, B, C} 
			\tag{factorisation \eqref{eq: fact multiplicativity f'}} \\[.2ex]
		&\eq \mu_{X', X''}\circ (\mu_{Y', Y''}\ot \id_{F(X'')}) 
			\circ \bigl((k'\ot k'')\ot f''\bigr)\circ \ao_{A, B, C}
			\tag{bifunctoriality of \(\ot\)} \\[.2ex]
		&\eq \mu_{X', X''}\circ (\mu_{Y', Y''}\ot \id_{F(X'')})\circ 
			\ao_{F(Y'), F(Y''), F(X'')}\circ \bigl(k'\ot(k''\ot f'')\bigr)
			\tag{naturality of \(\ao\)} \\[.2ex]
		&\eq F(\ao_{Y', Y'', X''})\circ \mu_{Y', Y''\ot X''}\circ (\id_{F(Y')}\ot\mu_{Y'', X''})
			\circ \bigl(k'\ot(k''\ot f'')\bigr)
			\tag{hexagon~\eqref{eq:H1}} \\[.2ex]
		&\eq F(\ao_{Y', Y'', X''})\circ 
			\mu_{Y', Y''\ot X''}\circ \bigl(k'\ot (\mu_{Y'', X''}\circ (k''\ot f''))\bigr)
			\tag{bifunctoriality of \(\ot\)}
	\end{align}
	yields the desired factorisation of \(f\circ h\). The case where \(h\) is a component of \(\aoi\) 
	is similar. 
\end{proof}

Morphisms with target an \(\ot\)\kern0.1em-product are more tractable. They can be discussed without invoking the hexagon axioms:

\begin{lemma}\label{lem:F otimes}
	Let \(u\in\fZ\), \(\vect{A}\in\cD^{|u|}\), and \(i\in \Idd_{\cD}^{\otimes}\). For any \(F\)-atomic morphism \(h\) with target \(H^{i}_u(\vect{A})\), the composite \(\tau^i_u(\vect{A})\circ h\) is \(\Idd_{\cD}^{\ot}\)-analysable.
\end{lemma}

\begin{proof}
	We repeat the proof of Lemma~\ref{lemma idot-analysis}, which proceeded by induction on \(|u|\) and used Lemma~\ref{lemma: ot f good} in the induction step. There are no new cases, because the source of \(\tau^i_u(\vect{A})\) is 
	an \(\ot\)\kern0.1em-product, while the targets of the additional coherence morphisms \(\mu\), \(\Delta\) are never \(\ot\)\kern0.1em-products by the tidiness of \(F\) and \(\cD\).
\end{proof}

For \(\prt\)-products, a new phenomenon occurs compared to Subsection \ref{sec: normal form LD cats}.

\begin{lemma}\label{lem:Fpar1}
Let \(u\in\fZ\), \(\vect{A}\in\cD^{|u|}\), and \(i\in \Idd_{\cD}^{\partimes}\). Let \(h\) be an \(F\)-atomic morphism with target \(H^{i}_u(\vect{A})\). Suppose that the composite \(\ka^i_u(\vect{A})\circ h\) is not \(\Idd_{\cD}^{\partimes}\)-analysable. Then there are objects \(X,Y\in \cC\) such that \(i=\id_{F(X)\prt F(Y)}\) and \(\ka^{i}_u(\vect{A})\circ h=\ka^{\Delta(X, Y)}_u(\vect{A})\).
\end{lemma} 

Here, \(\Delta(X, Y)\) is a typographical variant of \(\Delta_{X, Y}\), which we prefer in contexts with many sub\kern0.1em- and superscripts.
  
\begin{proof}
	Since \(\ka^i_u(\vect{A})\circ h\) is not \(\Idd_{\cD}^{\partimes}\)-analysable, Lemma~\ref{lemma idpar-analysis} implies that \(h\) is not atomic. As an \(F\)-atomic morphism, \(h\) must therefore be of one of the following two forms.
	
	\smallskip
	
	{\it First case: \(h\) is a component of \(\mu\) or \(\Delta\), or \(h=F(f)\) for an elementary morphism \(f\).}
	
	\smallskip
	
	{\it Second case: \(h\) is obtained by tensoring an \(F\)-atomic morphism with an identity.}
	
	\medskip
	
	We proceed by induction on \(|u|\), making the above case distinction in both the base case and the inductive step. Let \(B,C\in \cD\) be the unique objects such that \(i=\id_{B\partimes C}\). 
	
	\medskip
	
	\noindent \textbfit{Base case.} 
	If \(u = \e\), then \(\codom(h)=H^{i}_{\e}(\ast)=B \prt C\) is a \(\prt\)-product. In the first case, \(h=\Delta_{X,Y}\) must be a component of \(\Delta\) for some \(X,Y\in \cC\), as otherwise its target would not be a \(\prt\)-product. Then \(i=\id_{F(X)\prt F(Y)}\) and \(\ka^{i}_\e(\ast)\circ h=\ka^{\Delta(X, Y)}_\e(\ast)\). 
	
	The second case cannot occur: since \(\codom(h)\) is a \(\prt\)-product, \(h\) would need to be of the form \(h=h'\prt \id\) or \(h=\id \prt h'\) for some morphism \(h'\). But then \(\ka^i_\e(\ast)\circ h\) would be \(\Idd_{\cD}^{\partimes}\)-analysable, contradicting our assumption.
	
	\medskip
	
	\noindent \textbfit{Inductive step.}
	Let \(u\in \fY\sm\{\e\}\). By Definition \ref{def: Hparf} of \(H^{i}_u\) and since \(\dom(i)=B \prt C\) is a \(\prt\)-product, the object \(\codom(h)=H^{i}_u(\vect{A})\) is built up from at least three objects of~\(\cD\) via \(\ot\) or \(\prt\). By the tidiness of \(F\), the first case therefore cannot occur. 
	
	For the second case, assume \(h\) is of one of the forms \(\id\ot h', h'\ot\id, \id\prt h'\) or \(h'\prt\id\) for some morphism \(h'\). We treat only the case \(h=\id \ot h'\); the other cases are analogous. 
	
	Set \(w:=\init(u)\). Lemma~\ref{lemma: induction step par}\ref{it:312b} implies \(\ter(u)=\distl\) and that
	the composite \(\ka^i_{w}(\vect{A}_{\ge 2})\circ h'\) is not \(\Idd^{\prt}_{\cD}\)-analysable. By the induction hypothesis, there thus exist objects \(X, Y\in\cC\) such that \(i=\id_{F(X)\prt F(Y)}\) and 
	\begin{equation}\label{eq: h' to Delta}
		\ka^i_{w}(\vect{A}_{\ge 2})\circ h'
		\eq
		\ka^{\Delta(X, Y)}_{w}(\vect{A}_{\ge 2})\,.
	\end{equation} 
	This implies the desired equality, as the following calculation shows:
	\begin{align*}
		\ka^{i}_u(\vect{A})\circ h
		&\eq
		\distl\circ \bigl(A_1\ot \ka^i_{w}(\vect{A}_{\ge 2})\bigr)\circ (A_1\ot h')
		\tag{\(u=\distl w\) and \(h=\id \ot h'\)}\\[.2ex]
		&\eq
		\distl\circ \bigl(A_1\ot \ka^{\Delta(X, Y)}_{w}(\vect{A}_{\ge 2})\bigr)
		\tag{bifunctoriality of \(\ot\) and Equation \eqref{eq: h' to Delta}}\\[.2ex]
		&\eq
		\ka^{\Delta(X, Y)}_u(\vect{A})\,.
		\tag{\(u=\distl w\)}
	\end{align*}
\end{proof}

For our normal form result (Proposition~\ref{prop normal F}) for Frobenius LD functors, we need the following set of morphisms whose targets are \(\prt\)-products:
\begin{equation}
	\fd
	\eq
	\bigl\{\Delta\circ F(\ka^i_u(\vect{X}))\:\colon\: 
	i\in\Idd^{\prt}_{\cC},\, u\in\fZ,\, \vect{X}\in \cC^{|u|}\bigr\}\,,
\end{equation}
where the recursive definition of \(\ka_u\) is evaluated in~\(\cC\). 
For transparency we point out that a morphism \(\Delta\circ F(\ka^i_u(\vect{X}))\) 
has source \(F\bigl(H^{i}_u(\vect{X})\bigr)\) 
and target \(F\bigl(L^{i}_u(\vect{X})\bigr)\prt F\bigl(R^{i}_u(\vect{X})\bigr)\).

The remaining hexagon axioms~\eqref{eq:H2}\,--\,\eqref{eq:H4} have the following consequence. 

	\begin{lemma}\label{lem:Fpar2}
	Let \(v\in\fZ\), \(\vect{A}\in\cD^{|v|}\), and \(c\in\fd\). Let \(h\) be an \(F\)-atomic morphism with target \(H^{c}_v(\vect{A})\). Then the composite \(\ka^c_v(\vect{A})\circ h\) is \(\fd\)-analysable. 
\end{lemma}

\begin{proof}
	Write \(c = \Delta\circ F(\ka^i_u(\vect{X}))\) for \(u\in\fZ\), \(\vect{X}\in \cC^{|u|}\), \(i\in\Idd^{\prt}_{\cC}\). We argue by induction on \(|v|\). 
	
	\medskip
	
	\noindent \textbfit{Base case.} By the tidiness of \(F\), the target  
	\begin{equation}\label{eq:1826}
		\codom(h)
		\eq H^{c}_\e(\vect{A})
		\eq \dom(c)
		\eq F\bigl(H^{i}_u(\vect{X})\bigr)
	\end{equation}
	is neither an \(\ot\)\kern0.1em-product nor a \(\prt\)-product. As an \(F\)-atomic morphism, \(h\) must therefore be of one of the following two forms.
	
	\smallskip
	
	{\it \hspace{2em} First case: \(h=F(f)\) for some elementary morphism \(f\) of \(\cC\).}
	
	\smallskip
	
	By functoriality,
	\begin{equation}\label{eq: xi f factorisation} 
		c\circ h
		\eq \Delta\circ F(\ka^i_u(\vect{X})\circ f)\,.
	\end{equation}
	Since the morphism \(\ka^i_u(\vect{X})\circ f\) of~\(\cC\) is elementary and its target is a \(\prt\)-product, 
	Proposition~\ref{prop:normal cat}\ref{it:normal cat 3} yields a factorisation 
	\begin{equation}\label{eq: kappa f factorisation}
		\ka^i_u(\vect{X})\circ f=(g'\prt g'')\circ \ka^{j}_w(\vect{Y}) \,,
	\end{equation}	
	for suitable \(j\in\Idd^{\prt}_{\cC}\), \(w\in\fZ\), \(\vect{Y}\in\cC^{|w|}\), 
	and morphisms \(g'\), \(g''\) in~\(\cC\). The calculation 
	\begin{align}
		c\circ h
		&\eq \Delta\circ F(\ka^i_u(\vect{X})\circ f) \tag{factorisation \eqref{eq: xi f factorisation}}\\[.2ex]
		&\eq \Delta\circ F(g'\prt g'')\circ F(\ka^{j}_w(\vect{Y}))
		\tag{factorisation \eqref{eq: kappa f factorisation} and functoriality of \(F\)} \\[.2ex]
		&\eq \bigl(F(g') \prt F(g'')\bigr) \circ \Delta 
		\circ F(\ka^{j}_w(\vect{Y}))
		\tag{naturality of \(\Delta\)}
	\end{align}
	shows that \(c\circ h=\ka^c_\e(\ast)\circ h\) is indeed \(\fd\)-analysable. 
	
	\smallskip
	
	{\it \hspace{2em} Second case: \(h\) is a component of \(\mu\).}
	
	\smallskip
	
	By Equation~\eqref{eq:1826}, this forces \(H^{i}_u(\vect{X})\) to be an \(\ot\)\kern0.1em-product. Hence \(u\) is non-empty and~\(\ter(u)\) is either \(\distl\) or \(\distr\). 
	Write \(w=\init(u)\).  
	
	If \(\ter(u)=\distl\), then
	\begin{align*}
		c\circ h
		&\eq\Delta\circ F\bigl(\distl\circ (X_1\ot\ka^{i}_w(\vect{X}_{\ge 2}))\bigr)
		\circ\mu
		\tag{definition of \(\ka^i_u\)} \\[.2ex]
		&\eq\Delta\circ F(\distl)\circ F\bigl(X_1\ot\ka^{i}_w(\vect{X}_{\ge 2})\bigr)
		\circ\mu
		\tag{functoriality of \(F\)} \\[.2ex]
		&\eq \Delta\circ F(\distl)\circ \mu 
		\circ \bigl(F(X_1)\ot F(\ka^{i}_w(\vect{X}_{\ge 2}))\bigr)
		\tag{naturality of \(\mu\)} \\[.2ex]
		&\eq(\mu\prt \id)\circ\distl\circ (F(X_1)\ot\Delta)
		\circ \bigl(F(X_1)\ot F(\ka^{i}_w(\vect{X}_{\ge 2}))\bigr)
		\tag{hexagon~\eqref{eq:H2}} \\[.2ex]
		&\eq (\mu\prt \id)\circ\distl\circ \bigl(F(X_1) \ot 
		(\Delta\circ F(\ka^{i}_w(\vect{X}_{\ge 2})))\bigr)
		\tag{bifunctoriality of~\(\ot\)} \\[.2ex]
		&\eq(\mu\prt \id)\circ\ka_{\distl}^{d}\big(F(X_1)\bigr)\,,
		\tag{definition of~\(\ka_{\distl}\)}
	\end{align*}
	where \(d=\Delta\circ F(\ka^{i}_w(\vect{X}_{\ge 2}))\) belongs to \(\fd\).
	
	If \(\ter(u)=\distr\), a similar calculation gives
	\begin{align*}
		c\circ h
		&\eq\Delta\circ F\bigl(\distr\circ (\ka^{i}_w(\vect{X}_{\le |u|-1})
		\ot X_{|u|})\bigr)\circ\mu
		\tag{definition of \(\ka^i_u\)} \\[.2ex]
		&\eq\Delta\circ F(\distr)\circ F\bigl(\ka^{i}_w(\vect{X}_{\le |u|-1})
		\ot X_{|u|}\bigr)\circ\mu
		\tag{functoriality of \(F\)} \\[.2ex]
		&\eq \Delta\circ F(\distr)\circ \mu 
		\circ \bigl(F(\ka^{i}_w(\vect{X}_{\le |u|-1}))\ot F(X_{|u|})\bigr)
		\tag{naturality of \(\mu\)} \\[.2ex]
		&\eq(\id\prt\mu)\circ\distr\circ (\Delta\ot F(X_{|u|}))
		\circ \bigl(F(\ka^{i}_w(\vect{X}_{\le |u|-1}))\ot F(X_{|u|})\bigr)
		\tag{hexagon~\eqref{eq:H3}} \\[.2ex]
		&\eq(\id\prt\mu)\circ\distr\circ \bigl(
		(\Delta\circ F(\ka^{i}_w(\vect{X}_{\le |u|-1}))\ot F(X_{|u|})\bigr)
		\tag{bifunctoriality of~\(\ot\)} \\[.2ex]
		&\eq(\id\prt\mu)\circ\ka_{\distr}^{e}\big(F(X_{|u|})\bigr)\,,
		\tag{definition of~\(\ka_{\distr}\)}
	\end{align*}
	where \(e=\Delta\circ F(\ka^{i}_w(\vect{X}_{\le |u|-1}))\) lies again in \(\fd\). In either case, \(c\circ h=\ka^c_\e(\ast)\circ h\) is \(\fd\)-analysable.
	
	\medskip
	
	\noindent \textbfit{Inductive step.} By tidiness, \(\dom(c)\) is neither an \(\ot\)\kern0.1em-product nor a \(\prt\)-product. The induction hypothesis, together with Lemma~\ref{lemma: induction step par} applied with \(\ff=\fd\) and \(f=c\), thus covers all cases except when \(h=F(f)\) for an elementary morphism \(f\) in \(\cC\) or \(h\) is a component of~\(\mu\)
	or~\(\Delta\). 
	
	Since \(v\neq \e\), the object \(H^c_v(\vect{A})\) is either an \(\ot\)\kern0.1em-product or a \(\prt\)-product. By tidiness, \(h\) is therefore neither of the form \(F(f)\) for an elementary morphism \(f\) in \(\cC\) nor a component of~\(\mu\). Hence we may assume that \(h\) is a component of \(\Delta\) and that \(v\) is either \(\ap\) or \(\api\). 
	
	Assume \(v=\ap\) and write \(A=\vect{A}\). By definition,
	\begin{equation}		
		\ka^c_v(A)\circ h
		\eq 
		\ap\circ(A\prt c)\circ \Delta
		\eq
		\ap\circ(A\prt \Delta)\circ \bigl(A\prt F(\ka_u^i(\vect{X}))\bigr)\circ\Delta\,.
	\end{equation}	
	Choose \(Y\in\cC\) with \(A=F(Y)\). Then
	\begin{align*}
		\ka^c_v(A)\circ h
		&\eq \ap\circ(F(Y)\prt \Delta)
		\circ \bigl(F(Y)\prt F(\ka_u^i(\vect{X}))\bigr)\circ\Delta \\[.2ex]
		&\eq \ap\circ(F(Y)\prt \Delta)\circ \Delta\circ 
		F\bigl(Y \prt \ka_u^i(\vect{X})\bigr)
		\tag{naturality of \(\Delta\)} \\[.2ex]
		&\eq (\Delta\prt\id)\circ \Delta\circ F(\ap)\circ 
		F\bigl(Y \prt \ka_u^i(\vect{X})\bigr)
		\tag{hexagon~\eqref{eq:H4}} \\[.2ex]
		&\eq (\Delta\prt\id)\circ \Delta\circ 
		F\bigl(\ap\circ (Y \prt \ka_u^i(\vect{X}))\bigr)
		\tag{functoriality of \(F\)} \\[.2ex]
		&\eq (\Delta\prt\id)\circ \Delta\circ F\bigl(\ka^i_{\ap u}(Y, \vect{X})\bigr)\,,
		\tag{definition of \(\ka^i_{\ap u}\)} 
	\end{align*}
	which is a \(\fd\)-analysable morphism since 
	\(\Delta\circ F(\ka^i_{\ap u}(Y, \vect{X}))\in \fd\). 
	The case \(v=\api\) is similar.
\end{proof}

\begin{proposition}\label{prop normal F}
	Let \(F\colon \cC \to \cD\) be a tidy Frobenius LD functor. Let \(f\) be an \(F\)-elementary morphism in \(\cD\). 
	\begin{enumerate}[label=\rmlabel]
		\item\label{it:normal F1} If \(\codom(f)\) is neither a tensor product nor 
			of the form \(F(A)\) for some \(A\in\cC\), then \(f\) is an identity. 
		\item\label{it:normal F2} If \(\codom(f)=F(A)\) for some \(A\in \cC\), then 
			there is a factorisation \(f=F(g)\circ h\), where~\(g\) is an elementary  
			morphism of~\(\cC\) and~\(h\) is multiplicative. 
		\item\label{it:normal F3} If \(\codom(f)\) is an \(\ot\)-product, then \(f\) 
		 	is \(\Idd^{\ot}_{\cD}\)-analysable. 
		\item\label{it:normal F4} If \(\codom(f)\) is an \(\prt\)-product, then \(f\) 
		 	is \(\Idd^{\prt}_{\cD}\)-analysable or \(\fd\)-analysable.
	\end{enumerate}
\end{proposition}

\begin{proof}
	Part~\ref{it:normal F1} holds since the target of an \(F\)-atomic morphism is either a tensor product or of the form \(F(A)\) with \(A\in\cC\). 
	
	Analogously to the notion of height in the proof of Proposition~\ref{prop:normal cat}, we define the \mbox{\emph{\(F\)-height}} of an \(F\)-elementary morphism~\(f\) in \(\cC\) as the minimal number of \(F\)-atomic morphisms composing to \(f\); identities have \(F\)-height zero. Parts~\ref{it:normal F2} and~\ref{it:normal F3} are now proved similarly to Proposition~\ref{prop:normal cat} by induction on the \(F\)-height of \(f\), using Lemmas~\ref{lem:F(X)} and~\ref{lem:F otimes}.

	We turn to part~\ref{it:normal F4}. The same argument as in~\ref{it:normal F2} and~\ref{it:normal F3}, now based instead on Lemma~\ref{lem:Fpar1}, reveals that
	\begin{quotation}
		\it 
		if \(\codom(f)\) is a \(\prt\)-product and there is no 
		factorisation \(f=\ka^{\Delta(X, Y)}_u(\vect{A})\circ f'\) 
		with \(X, Y\in \cC\), \(u\in\fZ\), \(\vect{A}\in\cD^{|u|}\), and an 
		\(F\)-elementary morphism \(f'\), then \(f\) is \(\Idd^{\prt}_{\cD}\)-analysable.
	\end{quotation}
	Taking into account that 
	\(\Delta(X, Y)=\Delta(X, Y)\circ F(\ka^{\id_{X\prt Y}}_{\e})\in\fd\),
	we can repeat the argument once more with Lemma~\ref{lem:Fpar2} in the induction step, 
	thereby learning that 
	\begin{quotation}
		\it 
		if there exists a factorisation \(f=\ka^{\Delta(X, Y)}_u(\vect{A})\circ f'\) 
		with some \(X, Y\in \cC\), \(u\in\fZ\), \(\vect{A}\in\cD^{|u|}\), and an 
		\(F\)-elementary morphism \(f'\), then \(f\) is \(\fd\)-analysable.
	\end{quotation}
	Both italicised statements together imply part~\ref{it:normal F4} of the proposition.   
\end{proof}

\section{Coherence theorems}\label{sec:cohthms}
In this section, we prove the main results of this paper. We need the following terminology.
\begin{definition}
Call an object \(Y\) of a category \(\cC\) \emph{thin} if for every object \(X\in \cC\) there is at most one morphism from \(X\) to \(Y\). A category is \emph{thin} if all of its objects are thin. 
\end{definition}

\subsection{LD categories}

In this subsection, we fix a set \(S\) and denote by \(\cC\) the free LD category on \(S\) from Construction \ref{con: free LD cat}. In particular, all morphisms of \(\cC\) are elementary. The discrete category of positive integers becomes an LD category under addition. Given any function \(r\colon S \to \NN_{>0}\), the universal property of \(\cC\) applied to this LD category yields a \emph{rank function}
\begin{equation}
\|\cdot\|\colon\; \Ob(\cC)\longrightarrow \NN_{>0}
\end{equation} 
satisfying 
\begin{align}\label{eq:C-rank}
	\|A\|\eq r(A) \qquad \text{and} \qquad  \|X\ot Y\| \eq \|X\prt Y\|\eq \|X\|+\|Y\| \,,
\end{align}
for all objects \(A\in S\subset \operatorname{Ob}(\cC)\) and \(X, Y\in\cC\). Every morphism \(f\) in \(\cC\) preserves the rank, i.e., \(\|\dom(f)\|=\|\codom(f)\|\). We fix such a rank function for the rest of this subsection.

We will prove by induction on \(\|Y\|\) that every object \(Y\) of \(\cC\) is thin.  
When \(Y\) is an \(\ot\)\kern0.1em-product, we require the following observation.
 
\begin{lemma}\label{lem: uniqueness lemma otimes}
	If for identities \(i, j\in\Idd^{\ot}\), words \(u,v\in \fY\),
	and vectors \(\vect{A}\in \cC^{\lvert u\rvert}\), \(\vect{B}\in \cC^{\lvert v\rvert}\)
	we have \(H_u^i(\vect{A})=H_v^{j}(\vect{B})\)  
	and \(\norm{L_u^i(\vect{A})}=\norm{L_v^{j}(\vect{B})}\), 
	then \(i=j\), \(u=v\), and \(\vect{A}=\vect{B}\). 
\end{lemma}

\begin{proof}
	For brevity, we set \(X=H_u^i(\vect{A})=H_v^{j}(\vect{B})\). 
	Since \(X\) is an \(\ot\)\kern0.1em-product, there exist unique 
	objects \(\Lambda, \Rho\in \cC\) such that \(X=\Lambda\ot \Rho\).
	
We will use the following table, which considers the rank \(\|L_u^{i}\|\):
	\begin{center}
	\begin{tabular}{c|c}
		if & then \\ \hline \noalign{\vskip 0.6ex}
		\(u=\e\) & \(\norm{L_u^i(\vect{A})}=\|\Lambda\|\)\,, \\[.2ex]
		\(\ter(u)=\ao\) & \(\norm{L_u^i(\vect{A})}>\|\Lambda\|\)\,, \\[.2ex] 
		\(\ter(u)=\aoi\) & \(\norm{L_u^i(\vect{A})}<\|\Lambda\|\)\,.
	\end{tabular}
	\end{center}
	
	Indeed, if \(u=\e\), then \(i\) is the identity on \(X\), and 
	Definition~\ref{dfn:2205} implies \(L_u^i(\vect{A})=\Lambda\).
	If \(\ter(u)=\ao\), then \(A_1=\Lambda\),
	\begin{equation}
		\norm{L_u^i(\vect{A})}
		\eq
		\|A_1\|+\norm{L_{\init(u)}^i(\vect{A}_{\ge 2})}
		\;>\;
		\|\Lambda\|\,.
	\end{equation}
	This shows the second row of the table. Similarly, if \(\ter(u)=\aoi\), then \(A_{|u|}=\Rho\), and 
	\begin{equation}
		\|R^{i}_u(\vect{A})\|
		\eq
		\|R^{i}_u(\vect{A})_{\le |u|-1}\|+\|A_{|u|}\|
		\;>\;
		\|\Rho\|\,.
	\end{equation}
	Since \(\|\Lambda\|+\|\Rho\|=\|X\|=\|L^{i}_u(\vect{A})\|+\|R^{i}_u(\vect{A})\|\), the third row follows. The same implications hold for \(\norm{L_v^{j}(\vect{B})}\). 
	
	We now prove by induction on \(|u|\) that \(u=v\). 
 	Lemma~\ref{lemma: H is injective on objects} then implies the remaining claims. 
	In the base case, \(u=\e\), the table immediately yields \(v=\e\), since we then have \(\|L^{j}_v(\vect{B})\|=\|L^{i}_u(\vect{A})\|=\|\Lambda\|\). In the inductive step, both~\(u\) and~\(v\) must start with the same letter, otherwise the table would be contradicted. 

	If \(\ter(u)=\ter(v)=\ao\), then \(A_1=\Lambda=B_1\),
	\begin{equation}
		H^{i}_{\init(u)}(\vect{A}_{\ge 2})
		\eq
		\Rho
		\eq
		H^{j}_{\init(v)}(\vect{B}_{\ge 2}) \,,
	\end{equation}
	and 
	\begin{equation}
		\|L^{i}_{\init(u)}(\vect{A}_{\ge 2})\|
		\eq
		\norm{L_u^i(\vect{A})}-\|\Lambda\|
		\eq
		\norm{L_v^{j}(\vect{B})}-\|\Lambda\|
		\eq
		\|L^{j}_{\init(v)}(\vect{B}_{\ge 2})\|\,,
	\end{equation}
	so the induction hypothesis implies \(u=v\). The case \(\ter(u)=\ter(v)=\aoi\)
	is similar.
\end{proof}

Lemma \ref{lem: uniqueness lemma otimes} yields the following result.

\begin{lemma}\label{prop:coherence into otimes}
	If two objects \(Y'\), \(Y''\) are thin, then so is \(Y'\ot Y''\).
\end{lemma}

\begin{proof}
	Let \(f_1,f_2 \colon X \to Y\) be parallel morphisms with target \(Y=Y'\ot Y''\). 
	By Proposition~\ref{prop:normal cat}\ref{it:normal cat 2}, both morphisms are \(\Idd^{\ot}\)-analysable. This means there exist identities \(i, j\in\Idd^{\ot}\),
	words \(u, v\in \fY\), vectors \(\vect{A}\in \cC^{\lvert u \rvert}\), 
	\(\vect{B}\in \cC^{\lvert v \rvert}\), and morphisms
    \begin{align}
    	&g_1 \colon\; L^{i}_{u}(\vect{A}) \slongrightarrow Y',
    	&h_1 \colon\; R^{i}_{u}(\vect{A}) \slongrightarrow Y'',\\
    	&g_2 \colon\; L^{j}_{v}(\vect{B}) \slongrightarrow Y',
    	&h_2 \colon\; R^{j}_{v}(\vect{B}) \slongrightarrow Y''\,,
    \end{align}
    such that  
    \begin{align}\label{eq:1230}
    	f_1 \eq (g_1 \ot h_1) \circ \ta^{i}_{u}(\vect{A})\,,
    	\qquad
    	f_2 \eq (g_2 \ot h_2) \circ \ta^{j}_{v}(\vect{B})\,.
    \end{align}
    
    Since morphisms are rank-preserving, we have 
    \begin{equation}
		\norm{L^{i}_u(\vect{A})}
		\eq \norm{Y'}
		\eq \norm{L^{j}_{v}(\vect{B})}\,.
    \end{equation}
   	By Lemma~\ref{lem: uniqueness lemma otimes}, it follows that \(i=j\), \(u=v\), and 
	\(\vect{A}=\vect{B}\).
	Thus, \(g_1\), \(g_2\) are parallel morphisms from \(L_u^{i}(\vect{A})\) to \(Y'\), and \(h_1\), \(h_2\) are parallel morphisms from \(R_u^{i}(\vect{A})\) to \(Y''\).
	Since \(Y'\) and \(Y''\) are thin, we conclude~\(g_1=g_2\) and~\(h_1=h_2\). By~\eqref{eq:1230}, this implies \(f_1=f_2\).
\end{proof}
 
 By replacing the free monoid~\(\fY\) from Subsection \ref{subs:morphot} with its submonoid generated by \(\ao\) only, Lemma~\ref{prop:coherence into otimes} allows us to recover the following theorem of Laplaza.

\begin{corollary}[\cite{Lap72}]\label{thm: unitless lax monoidal coherence}
	The free lax monoidal category on a set is thin. 
	Equivalently, all formal diagrams in a unitless lax monoidal category commute.
\end{corollary}

The analogue of Lemma~\ref{lem: uniqueness lemma otimes} for \(\prt\)-products 
reads as follows. 

\begin{lemma}\label{lem: uniqueness lemma partimes}
	If, for identities \(i, j\in\Idd^{\prt}\), words \(u,v\in \fZ\),
	and vectors \(\vect{A}\in \cC^{\lvert u\rvert}\), \(\vect{B}\in \cC^{\lvert v\rvert}\),
	we have \(H_u^i(\vect{A})=H_v^{j}(\vect{B})\)  
	and \(\norm{L_u^i(\vect{A})}=\norm{L_v^{j}(\vect{B})}\), 
	then \(u=v\), \(i=j\), and \(\vect{A}=\vect{B}\). 
\end{lemma}

\begin{proof}
	The argument is similar to that of Lemma~\ref{lem: uniqueness lemma otimes},
	so we focus on the necessary changes. 
	
	The object 
	\(X=H_u^i(\vect{A})=H_v^{j}(\vect{B})\) may be either an \(\ot\)\kern0.1em-product or 
	a~\(\prt\)-product. If \(X=\Lambda\ot \Rho\) is an \(\ot\)\kern0.1em-product, then \(u\) and \(v\) 
	must end with \(\distl\) or \(\distr\). Moreover, 
	\begin{center}
	\begin{tabular}{c|c}
		if & then \\ \hline \noalign{\vskip 0.6ex}
		\(\ter(u)=\distl\) & \(\norm{L_u^i(\vect{A})}>\|\Lambda\|\)\,, \\[.2ex] 
		\(\ter(u)=\distr\) & \(\norm{L_u^i(\vect{A})}<\|\Lambda\|\)\,.
	\end{tabular}
	\end{center} 
	
	Similarly, if \(X=\Lambda\prt \Rho\) is a \(\prt\)-product, then \(u\) and \(v\) are either empty or end with \(\ap\) or \(\api\). In this case, 
	\begin{center}
	\begin{tabular}{c|c}
		if & then \\ \hline \noalign{\vskip 0.6ex}
		\(u=\e\) & \(\norm{L_u^i(\vect{A})}=\|\Lambda\|\)\,, \\[.2ex]
		\(\ter(u)=\ap\) & \(\norm{L_u^i(\vect{A})}>\|\Lambda\|\)\,, \\[.2ex] 
		\(\ter(u)=\api\) & \(\norm{L_u^i(\vect{A})}<\|\Lambda\|\)\,.
	\end{tabular}
	\end{center}
	
	Based on these observations, a proof by induction on \(|u|\)  analogous to the one in Lemma~\ref{lem: uniqueness lemma otimes} yields \(u=v\).
	The remaining claims follow from Lemma~\ref{lemma: Hpar is injective on objects}.
\end{proof}

\begin{lemma}\label{prop:coherence into par}
 	If \(Y',Y''\in \cC\) are thin, then so is \(Y'\prt Y''\).
\end{lemma}

\begin{proof}
	We repeat the argument from the proof of Lemma~\ref{prop:coherence into otimes},
	now using Proposition~\ref{prop:normal cat}\ref{it:normal cat 3} to analyse two given parallel morphisms \(f_1, f_2\colon X\longrightarrow Y'\prt Y''\),
	and Lemma~\ref{lem: uniqueness lemma partimes} to compare the resulting factorisations.
\end{proof}

We arrive at our first main result:

\begin{theorem}\label{thm: coherence}
	The free unitless LD category on a set is thin. Equivalently, every formal 
	diagram in a unitless LD category commutes. This also holds for unitless lax LD categories.
\end{theorem}

\begin{remark}
	The same result was obtained with different techniques in \cite[\S 7]{DP04}.
\end{remark}

\begin{proof}[Proof of Theorem~\ref{thm: coherence}]
	We prove by induction on \(\norm{Y}\) that every object~\(Y\) of \(\cC\) is thin. 
	The base case is an immediate consequence of 
	Proposition~\ref{prop:normal cat}\ref{it:normal cat 1}. In the inductive step, we apply Lemmas~\ref{prop:coherence into otimes}
	and~\ref{prop:coherence into par}. 
	A simplified version of this argument applies to unitless lax LD categories by replacing the free monoid \(\fY\) from Subsection~\ref{subs:morphot} with its submonoid generated by \(\ao\) alone, and the free monoid \(\fZ\) from Subsection~\ref{subs:parmor} with its submonoid generated by \(\ap,\distl,\distr\) alone.
\end{proof}

As mentioned in the introduction (Subsection \ref{subsec:unitsbreak}), it is \emph{not} true that any formal diagram in a linearly distributive category \emph{with units} commutes.

\begin{example}\label{ex:counterex coherence 2}
Let \(\cC\) be an LD category equipped with unit objects $1$ and $K$ for the monoidal products $\ot$ and $\partimes$, respectively. Tensoring with the wrong unit yields endofunctors \(L= K \ot -\) and \(R= 1 \partimes -\). These form a pair of adjoint functors \(L\dashv R\) with counit and unit
\begin{align*}
	LR(A) \sxlongrightarrow{\distl_{K,1,A}} {(K \ot 1)} \partimes A \Simeq A \,,
	\qquad \text{and} \qquad
	A \Simeq (1 \partimes K)\ot A \sxlongrightarrow{\distr_{1,K,A}} RL(A) \,,
\end{align*}
where the unlabeled isomorphisms are unit constraints. The snake relations follow from the triangle axioms relating distributors and unitors in LD categories with units; see \cite{CS97}.

The counit yields two generally distinct natural transformations, depicted in Figure~\ref{fig:capcap}.

	\begin{figure}[H]
	\centering
	\begin{tikzpicture}
		\draw[line width=1] (5.1,0) to[out=90,in=180] (5.5,0.8)
		to[out=0,in=90] (5.9,0);
		\node at (4.85,0) {$L$};
		\node at (5.65,0) {$R$};
		\draw[line width=1] (4.5,0) to (4.5,3);
		\draw[line width=1] (3.5,0) to (3.5,3);
		\node at (3.25,0) {$L$};
		\node at (4.25,0) {$R$};
		\node at (7.4,1.5) {$\Neq$};
		\draw[line width=1] (9,0) to[out=90,in=180] (9.4,0.8)
		to[out=0,in=90] (9.9,0);
		\node at (8.85,0) {$L$};
		\node at (9.65,0) {$R$};
		\draw[line width=1] (11.5,0) to (11.5,3);
		\draw[line width=1] (10.5,0) to (10.5,3);
		\node at (10.25,0) {$L$};
		\node at (11.25,0) {$R$};
	\end{tikzpicture}
	\caption{Two natural transformations $(L\circ R)^2\to L \circ R$.}
	\label{fig:capcap}
\end{figure}
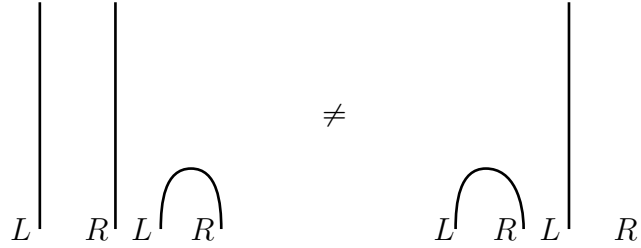

To make this explicit, consider the LD category of bimodules over the three-dimensional \(\CC\)-algebra
\(S:=\CC[x,y]/\langle x^2,y^2,xy \rangle\), with the two tensor products introduced in Example
\ref{ex: Abimod2}.

This algebra does not admit the structure of a Frobenius algebra. In particular, the unit $S$ of $\ot$ and $S^*$ of $\prt$ are not isomorphic bimodules, and consequently the functors $L$ and $R$
are not isomorphic to the identity. Fix a basis \(\{1,x,y\}\) of \(S\) with dual basis \(\{1^{\ast},x^{\ast},y^{\ast}\}\). Evaluating the two natural transformations in Figure~\ref{fig:capcap} 
on the unit object \(1= {_SS_S}\) yields two distinct morphisms $LRLR(1) \to LR(1)$. Indeed, their values on the element
\begin{equation}
	[x^{\ast}\ot_{\CC}(x\ot_{\CC}[x^{\ast}\ot_{\CC}(y\ot_{\CC}y)])]
	\;\in\; 
	LRLR(1) \eq S^{\ast} \ot (S \prt {(S^{\ast} \ot (S \prt S))})
\end{equation}
are
\begin{equation}
x^{\ast}(y)\cdot[x^{\ast}\ot_{\CC}(x\ot_{\CC}y)]
\eq
0
\;\neq\;
[x^{\ast}\ot_{\CC}(y\ot_{\CC}y)]
 \eq
 x^{\ast}(x)\cdot[x^{\ast}\ot_{\CC}(y\ot_{\CC}y)]
 \,.
 \end{equation}

Many similar counterexamples can be constructed. Notably, the same string diagrams as in Figure~\ref{fig:capcap}, when interpreted in a different bicategory, also yield counterexamples to coherence for duoidal categories \cite[Prop. 3.1.6]{Ro24}.
\end{example}

\subsection{Frobenius LD functors}\label{sec: coh unitless Frob LD}
In this subsection, we fix a set \(S\) and let \(F\colon \cC \to \cD\) denote the free Frobenius LD functor on \(S\) from Construction \ref{con: free Frob LD fun}. In particular, all morphisms of \(\cD\) are \(F\)-elementary. 
Choose a rank function \(X\mapsto \|X\|_{\cC}\) on \(\cC\) satisfying~\eqref{eq:C-rank}.
The universal property of the free LD category~\(\cD^-\) provides a 
rank function  \(\|\cdot\|\colon \Ob(\cD)\longrightarrow \NN_{>0}\) 
satisfying \(\|F(X)\|=\|X\|_{\cC}\) for all \(X\in \cC\) and 
\begin{equation}
	\|A\ot B\|\eq\|A\prt B\|\eq\|A\|+\|B\| \,,
\end{equation} 
for all \(A, B\in\cD\). We have 
\begin{equation}
	\|F(X)\ot F(Y)\|\eq\|F(X\ot Y)\| 
	\quad \text{ and } \quad 
	\|F(X)\prt F(Y)\|\eq\|F(X\prt Y)\| \,,
\end{equation} 
for all \(X, Y\in\cC\), because all four terms simplify to \(\|X\|_{\cC}+\|Y\|_{\cC}\).
Consequently, all morphisms of \(\cD\) are rank-preserving. 
	
\begin{lemma}\label{lem:multmax1}
	For every \(A\in\cD\), there is at most one multiplicative morphism with source~\(A\).
\end{lemma}

\begin{proof}
	By induction on \(\|A\|\).
\end{proof} 

\begin{lemma}\label{lem:FXthin}
	For every \(X\in \cC\), the object \(F(X)\in \cD\) is thin. 
\end{lemma}

\begin{proof}
	 Let \(A\in \cD\), and let \(f_1, f_2\colon A\longrightarrow F(X)\) be parallel morphisms.
	
	Proposition~\ref{prop normal F}\ref{it:normal F2} yields 
	factorisations \(f_1=F(g_1)\circ h_1\) and \(f_1=F(g_2)\circ h_2\) such that~\(h_1\) and~\(h_2\) are multiplicative and~\(g_1\) and~\(g_2\) are morphisms in \(\cC\). Since \(\dom(h_1)=A=\dom(h_2)\), Lemma \ref{lem:multmax1} implies~\(h_1=h_2\). Therefore, \(g_1\) and \(g_2\) are parallel morphisms and Theorem~\ref{thm: coherence} shows \(g_1=g_2\). Altogether, we find \(f_1=f_2\).
\end{proof}

\begin{lemma}\label{lem:otthin}
	If \(A', A''\in\cD\) are thin, then so is \(A'\ot A''\in \cD\).
\end{lemma}

\begin{proof}
	The proof is analogous to the one of Lemma~\ref{prop:coherence into otimes}.  
	Proposition~\ref{prop normal F}\ref{it:normal F3} provides the required factorisations, and the uniqueness argument based on Lemma~\ref{lem: uniqueness lemma otimes} carries over to the present rank function.
\end{proof}

As usual, one has to work harder for \(\prt\)-products. Here we need the following generalisation of Lemma~\ref{lem: uniqueness lemma partimes}.

\begin{lemma}\label{lem: F uniqueness prt}
	If, for morphisms \(f, g\in\Idd^{\prt}_{\cD}\cup\fd\), 
	words \(u,v\in \fZ\), and vectors \(\vect{A}\in \cC^{\lvert u\rvert}\), 
	\(\vect{B}\in \cC^{\lvert v\rvert}\),
	we have \(H_u^f(\vect{A})=H_v^{g}(\vect{B})\)  
	and \(\norm{L_u^f(\vect{A})}=\norm{L_v^{g}(\vect{B})}\), 
	then \(u=v\), \(f=g\), and \(\vect{A}=\vect{B}\). 
\end{lemma}

\begin{proof}
	We proceed similarly to the proof of Lemma~\ref{lem: uniqueness lemma partimes}. For brevity, set \(C=H_u^f(\vect{A})=H_v^{g}(\vect{B})\). 
	
	\smallskip

	If \(u\ne\e\), then, as in Lemma~\ref{lem: uniqueness lemma partimes}, the object \(C\) is either an \(\ot\)\kern0.1em-product or a \(\partimes\)-product. 
	
	All five rows of the two tables in Lemma~\ref{lem: uniqueness lemma partimes}, except the one involving \(u=\e\), remain valid (with \(i\) replaced by \(f\)).
	
	If \(u=\e\), then either \(f\in \Idd^{\prt}_{\cD}\) and \(C\) is a \(\prt\)-product, or \(f\in\fd\) and \(C=\dom(f)\) is of the form \(F(X)\) for some \(X\in\cC\). The inductive argument establishing \(u=v\) carries through, and 
	Lemma~\ref{lemma: Hpar is injective on objects} implies \(\vect{A}=\vect{B}\) and \(\dom(f)=\dom(g)\). Moreover, at each step of the induction, when we cancel the common last letter of \(u\) and \(v\), the assumption 
	\(\norm{L_u^f(\vect{A})}=\norm{L_v^{g}(\vect{B})}\) is preserved. Thus, in the end, we obtain
	\begin{equation}\label{eq:rnkLempty}
	\|L_{\e}^f(\ast)\| \eq \|L_{\e}^g(\ast)\|\,.
	\end{equation}
	
	\smallskip
	
	We still need to show that \(f=g\). If at least one of \(f\) or \(g\) lies in \(\Idd^{\prt}_{\cD}\), then \(\dom(f)=\dom(g)\) is a \(\prt\)-product, so both \(f\) and \(g\) belong to \(\Idd^{\prt}_{\cD}\), and thus \(f=g\) follows. 
	
	\smallskip
	
	It remains to consider the case \(f, g\in\fd\).  In this case, we find 
	identities \(i, j\in\Idd^{\prt}_{\cC}\), words \(x, y\in\fZ\), 
	and vectors \(\vect{A}\in\cC^{|x|}\), \(\vect{B}\in\cC^{|y|}\) such that 
	\begin{equation}
		f\eq\Delta\circ F\bigl(\ka_x^{i}(\vect{X})\bigr)
		\quad \text{ and } \quad 
		g\eq\Delta\circ F\bigl(\ka_y^{j}(\vect{Y})\bigr)\,.
	\end{equation}
	Since \(F\) is injective on objects by Construction~\ref{con: free Frob LD fun}~\ref{it: con T}, 
	\begin{equation}
		F\bigl(H^{i}_x(\vect{X})\bigr)
		\eq
		\dom(f)
		\eq
		\dom(g)
		\eq
		F\bigl(H^{j}_y(\vect{Y})\bigr)
	\end{equation}
	implies \(H^{i}_x(\vect{X})=H^{j}_y(\vect{Y})\). 
	Due to \(\codom(f)=F(L^{i}_x(\vect{X}))\prt F(R^{i}_x(\vect{X}))\),
	we have \(L_{\e}^f(\ast)=F(L^{i}_x(\vect{X}))\)
	and, similarly, \(L_{\e}^g(\ast)=F(L^{j}_y(\vect{Y}))\). 
	Thus, the equality of ranks \eqref{eq:rnkLempty} yields
	\begin{equation}
		\|F\bigl(L^{i}_x(\vect{X})\bigr)\|\eq\|F\bigl(L^{j}_y(\vect{Y})\bigr)\|\,.
	\end{equation}
	
	By the definition of the rank function \(\|\cdot\|\), this implies \begin{equation}
		\|L^{i}_x(\vect{X})\|_{\cC}\eq\|L^{j}_y(\vect{Y})\|_{\cC}\,.
	\end{equation}
	Applying Lemma~\ref{lem: uniqueness lemma partimes} in \(\cC\), 
	we conclude \(i=j\), \(x=y\), and \(\vect{X}=\vect{Y}\).
	Consequently, \(f=g\).  
\end{proof}

\begin{lemma}\label{lem:prtthin}
	If \(A', A''\in\cD\) are thin, then so is \(A'\prt A''\).
\end{lemma}

\begin{proof}
	This follows from Proposition~\ref{prop normal F}\ref{it:normal F4}
	and Lemma~\ref{lem: F uniqueness prt}. The argument is analogous to that of Lemma~\ref{prop:coherence into otimes}.
\end{proof}

\begin{theorem}\label{thm: funcoherence}
The target category of the free unitless Frobenius LD functor on a set is thin. Equivalently, given a unitless Frobenius LD functor \(F\), every \(F\)-formal diagram commutes.
\end{theorem}

\begin{remark}
Since the target category of the free Frobenius LD functor on a set is unique up to categorical isomorphism, thinness is independent of the model of such a functor.
	\end{remark}

\begin{proof}
	For a contradiction, let \(A\) be an object in \(\cD\) which is not thin, chosen such that \(\|A\|\)
	is minimal. Owing to Lemmas~\ref{lem:FXthin},~\ref{lem:otthin}, and~\ref{lem:prtthin}, the object 
	\(A\) is neither a tensor product nor of the 
	form \(F(X)\) for any \(X\in \cC\). By Proposition~\ref{prop normal F}\ref{it:normal F1}, the only morphism into \(A\) is therefore an identity, contradicting that \(A\) is not thin.
\end{proof}

Recall that an \emph{LD Frobenius algebra} \cite[\S 3]{DeS25}, see also \cite{FSSW25b}, is a Frobenius LD functor with source category the terminal category \(\ast\). In the following corollary, we consider a \emph{unitless LD Frobenius algebra} in a LD category \(\cC\), i.e., a unitless Frobenius LD functor \(F\colon \ast \to \cC\). We can identify such an \(F\) with an object \(A=F(\ast)\) equipped with multiplication \(\mu\colon A\ot A \to A\) and comultiplication \(\Delta\colon A \to A \prt A\) satisfying specializations of the equations from Definition \ref{def: Vorsicht Funktor}. The following \emph{unitless spider theorem} is a modification of the well-known planar spider theorem \cite[\S 5.2]{HeuVic19}, \cite{MaRie21} to the unitless linearly distributive setting.

\begin{corollary}\label{cor:spider}
	Let \(A\) be a unitless LD Frobenius algebra in \(\cC\), and let \(m,n\in \NN_{>0}\). Every morphism \(A^{\ot m}\to A^{\prt n}\) built from finitely many multiplications \(\mu\), comultiplications \(\Delta\) and identities using the operations \(\circ\), \(\ot\), and \(\prt\), is equal to the following normal form:
	\begin{equation}
		\Delta^{(n)} \circ \cdots\circ \Delta^{(1)} \circ \mu^{(1)} \circ \cdots\circ \mu^{(m)}.
	\end{equation}
	Here, \(A^{\ot 1}=A^{\prt 1}=A\), \(A^{\ot k}=A \ot A^{\ot k-1}\), and \(A^{\prt k}=A \prt A^{\prt k-1}\) for \(k>1\). Moreover, \(\mu^{(1)}=\Delta^{(1)}=\id\), \(\mu^{(2)}=\mu\), \(\Delta^{(2)}=\Delta\), and for \(k>2\), \(\mu^{(k)}=A\ot \mu^{(k-1)}\), \(\Delta^{(k)}=A\prt \Delta^{(k-1)}\).
\end{corollary}


\vskip 15em

\begin{center}
    \sc Acknowledgments\\[.5em]
\end{center}

We thank Simon Lentner and Yorck Sommerhäuser for discussions about free categories, Paul-André Melliès for discussions about Grothendieck-Verdier categories, and Tony Zorman for discussions about duoidal coherence.

M.D. and C.S. acknowledge support from the DFG through the CRC 1624 \emph{Higher Structures, Moduli Spaces and Integrability}, project number 506632645. C.S. is partially funded by the Excellence Cluster EXC 2121 \emph{Quantum Universe}, project number 390833306.

\bibliographystyle{alphaurl}
\bibliography{references}

@misc{All23,
	title={Hopf algebroids and {G}rothendieck-{V}erdier duality}, 
	author={Allen, R.},
	year={2023},
	note={Preprint},
	eprint={2308.01029},
}

@article{ALSW25,
	author = {Allen, R. and Lentner, S. and Schweigert, C. and Wood, S.},
	title = {Duality structures for representation categories of vertex operator algebras and the {Feigin}-{Fuchs} boson},
	fjournal = {Selecta Mathematica. New Series},
	journal = {Sel. Math., New Ser.},
	issn = {1022-1824},
	volume = {31},
	number = {2},
	pages = {57},
	year = {2025},
}

@article {BoD13,
	AUTHOR = {Boyarchenko, M. and Drinfeld, V.},
	TITLE = {A duality formalism in the spirit of {G}rothendieck and
	{V}erdier},
	JOURNAL = {Quantum Topol.},
	FJOURNAL = {Quantum Topology},
	VOLUME = {4},
	YEAR = {2013},
	NUMBER = {4},
	PAGES = {447--489},
}

@phdthesis{Bla23,
	author = {Blanco, N.},
	title = {Bifibrations of polycategories and classical linear logic},
	year ={2023},
	school ={University of Birmingham},
	eprint = {2305.15139},
}

@incollection{BlZ20,
	author = {Blanco, N. and Zeilberger, N.},
	title = {Bifibrations of polycategories and classical linear logic},
	booktitle = {Proceedings of the 36th conference on mathematical foundations of programming semantics, MFPS XXXVI, virtual conference, June 2--6, 2020},
	pages = {29--52},
	year = {2020},
	publisher = {Amsterdam: Elsevier}
}

@article {CuLaA24,
	AUTHOR = {Curien, P.-L. and Laplante-Anfossi, G.},
	TITLE = {Topological proofs of categorical coherence},
	JOURNAL = {Cah. Topol. G\'eom. Diff\'er. Cat\'eg.},
	FJOURNAL = {Cahiers de Topologie et G\'eom\'etrie Diff\'erentielle
	Cat\'egoriques},
	VOLUME = {65},
	YEAR = {2024},
}

@article {CS97,
	AUTHOR = {Cockett, J. R. B. and Seely, R. A. G.},
	TITLE = {Weakly distributive categories (corrected version)},
	JOURNAL = {J. Pure Appl. Algebra},
	FJOURNAL = {Journal of Pure and Applied Algebra},
	VOLUME = {114},
	YEAR = {1997},
	NUMBER = {2},
	PAGES = {133--173},
}

@article{CS99,
	author = {Cockett, J. R. B. and Seely, R. A. G.},
	title = {Linearly distributive functors},
	fjournal = {Journal of Pure and Applied Algebra},
	journal = {J. Pure Appl. Algebra},
	issn = {0022-4049},
	volume = {143},
	number = {1-3},
	pages = {155--203},
	year = {1999},
}

@article{CzKQW25,
	author = {Czenky, A. and Kesten, J. and Quinonez, A. and Walton, C.},
	title = {On extended {Frobenius} structures},
	fjournal = {Theory and Applications of Categories},
	journal = {Theory Appl. Categ.},
	issn = {1201-561X},
	volume = {44},
	pages = {1218--1255},
	year = {2025},
}

@article {Deh05,
	AUTHOR = {Dehornoy, P.},
	TITLE = {Geometric presentations for {T}hompson's groups},
	JOURNAL = {J. Pure Appl. Algebra},
	FJOURNAL = {Journal of Pure and Applied Algebra},
	VOLUME = {203},
	YEAR = {2005},
	NUMBER = {1-3},
	PAGES = {1--44},
	ISSN = {0022-4049,1873-1376},
}

@misc{De26,
	title={A lifting theorem for {G}rothendieck-{V}erdier categories}, 
	author={Demirdilek, M.},
	year={2026},
	eprint={2601.14812},
	note={Preprint},
}

@misc{DeS25,
	title={{S}urface {D}iagrams for {F}robenius {A}lgebras and {F}robenius-{S}chur {I}ndicators in {G}rothendieck-{V}erdier {C}ategories}, 
	author={M. Demirdilek and C. Schweigert},
	year={2025},
	note={Accepted for publication in \emph{Higher Structures}},
}

@book{DP04,
	author = {Do{\v{s}}en, K. and Petri{\'c}, Z.},
	title = {Proof-theoretical coherence},
	fseries = {Studies in Logic (London)},
	series = {Stud. Log. (Lond.)},
	volume = {1},
	isbn = {1-904987-06-0},
	year = {2004},
	publisher = {London: King's College Publications},
	language = {English},
}

@article {Ep66,
	AUTHOR = {Epstein, D. B. A.},
	TITLE = {Functors between tensored categories},
	JOURNAL = {{I}nvent. {M}ath.},
	FJOURNAL = {{I}nventiones {M}athematicae},
	VOLUME = {1},
	YEAR = {1966},
	PAGES = {221--228},
}

@article{FlLaPo25,
	author = {Flake, J. and Laugwitz, R. and Posur, S.},
	title = {Frobenius monoidal functors from ambiadjunctions and their lifts to {Drinfeld} centers},
	fjournal = {Advances in Mathematics},
	journal = {Adv. Math.},
	issn = {0001-8708},
	volume = {475},
	pages = {59},
	note = {Id/No 110344},
	year = {2025},
}

@article{Fo08,
	author = {Forcey, S.},
	title = {Convex hull realizations of the multiplihedra},
	fjournal = {Topology and its Applications},
	journal = {Topology Appl.},
	issn = {0166-8641},
	volume = {156},
	number = {2},
	pages = {326--347},
	year = {2008},
}

@article{FSSW25b,
	author = {Fuchs, J. and Schaumann, G. and Schweigert, C. and Wood, S.},
	title = {Grothendieck-{V}erdier module categories, {Frobenius} algebras and relative {Serre} functors},
	fjournal = {Advances in Mathematics},
	journal = {Adv. Math.},
	issn = {0001-8708},
	volume = {475},
	pages = {69},
	note = {Id/No 110325},
	year = {2025},
	language = {English},
}

@article{FuSY25,
	author = {Fuchs, J. and Schweigert, C. and Yang, Y.},
	title = {String-net models for pivotal bicategories},
	fjournal = {Theory and Applications of Categories},
	journal = {Theory Appl. Categ.},
	issn = {1201-561X},
	volume = {44},
	pages = {474--543},
	year = {2025},
}

@article {Gi87,
	AUTHOR = {Girard, J.-Y.},
	TITLE = {Linear logic},
	JOURNAL = {Theoret. Comput. Sci.},
	FJOURNAL = {Theoretical Computer Science},
	VOLUME = {50},
	YEAR = {1987},
	NUMBER = {1},
	PAGES = {101},
}

@article {GuJo25,
	AUTHOR = {Gurski, N. and Johnson, N.},
	TITLE = {Universal pseudomorphisms, with applications to diagrammatic
	coherence for braided and symmetric monoidal functors},
	JOURNAL = {Compositionality},
	FJOURNAL = {Compositionality},
	VOLUME = {7},
	YEAR = {2025},
	NUMBER = {3},
	PAGES = {68},
}

@misc{Ha84,
	author = {Haiman, M.},
	title = {Constructing the associahedron},
	year = {1984},
	note = {Unpublished manuscript. Available at:\\ \url{https://math.berkeley.edu/~mhaiman/ftp/assoc/manuscript.pdf}},
}

@book{HeuVic19,
	author = {Heunen, C. and Vicary, J.},
	title = {Categories for quantum theory. {An} introduction},
	fseries = {Oxford Graduate Texts in Mathematics},
	series = {Oxf. Grad. Texts Math.},
	volume = {28},
	isbn = {978-0-19-873962-3; 978-0-19-873961-6},
	year = {2019},
	publisher = {Oxford: Oxford University Press},
}

@phdthesis{Hou07, 
	author = {Houston, R.}, 
	title = {Linear Logic without Units}, 
	year = {2007}, 
	school = {University of Manchester},
	eprint ={1305.2231},
}

@article{JaYa26,
	author = {Jaklitsch, D. and Yadav, H.},
	title = {{{\(\otimes\)}}-Frobenius functors and exact module categories},
	fjournal = {IMRN. International Mathematics Research Notices},
	journal = {Int. Math. Res. Not.},
	issn = {1073-7928},
	volume = {2026},
	number = {8},
	pages = {35},
	note = {Id/No rnag065},
	year = {2026},
}

@article {Ka93,
	AUTHOR = {Kapranov, M.},
	TITLE = {The permutoassociahedron, {M}ac {L}ane's coherence theorem and
	asymptotic zones for the {KZ} equation},
	JOURNAL = {J. Pure Appl. Algebra},
	FJOURNAL = {Journal of Pure and Applied Algebra},
	VOLUME = {85},
	YEAR = {1993},
	NUMBER = {2},
	PAGES = {119--142},
}

@article{Lap72,
	author = {Laplaza, M. L.},
	title = {Coherence for associativity not an isomorphism},
	fjournal = {Journal of Pure and Applied Algebra},
	journal = {J. Pure Appl. Algebra},
	issn = {0022-4049},
	volume = {2},
	pages = {107--120},
	year = {1972},
	language = {English},
}

@misc{Le72,
	author = {Lewis, G.},
	title = {Coherence for a closed functor},
	year = {1972},
	language = {English},
	howpublished = {Coherence in {Categories}, {Lect}. {Notes} {Math}. 281, 148-195 (1972).},
}

@article{LiWi94,
	author = {Lincoln, P. and Winkler, T.},
	title = {Constant-only multiplicative linear logic is {NP}-complete},
	fjournal = {Theoretical Computer Science},
	journal = {Theor. Comput. Sci.},
	issn = {0304-3975},
	volume = {135},
	number = {1},
	pages = {155--169},
	year = {1994},
}

@article {MaPo22,
	AUTHOR = {Malkiewich, C. and Ponto, K.},
	TITLE = {Coherence for bicategories, lax functors, and shadows},
	JOURNAL = {Theory Appl. Categ.},
	FJOURNAL = {Theory and Applications of Categories},
	VOLUME = {38},
	YEAR = {2022},
	PAGES = {Paper No. 12, 328--373},
}

@misc{MaRie21,
	author = {Majid, S. and Rietsch, K.},
	title = {Planar spider theorem and asymmetric {Frobenius} algebras},
	year = {2021},
	eprint = {2109.12106},
	note ={Preprint},
}

@article{MW10,
	author = {Ma'u, S. and Woodward, C.},
	title = {Geometric realizations of the multiplihedra},
	fjournal = {Compositio Mathematica},
	journal = {Compos. Math.},
	issn = {0010-437X},
	volume = {146},
	number = {4},
	pages = {1002--1028},
	year = {2010},
}

@misc{MePou26,
	author = {Melani, V. and Pourcelot, H.},
	title = {Dioperads, {Frobenius} monoidal functors and duality},
	year = {2026},
	note = {Preprint},
	eprint = {2604.01080}
}

@article{PP25,
	author = {Pilaud, V. and Poliakova, D.},
	title = {Hochschild polytopes},
	fjournal = {Mathematische Annalen},
	journal = {Math. Ann.},
	issn = {0025-5831},
	volume = {392},
	number = {2},
	pages = {2395--2441},
	year = {2025},
}

@article {Po89,
	AUTHOR = {Power, A. J.},
	TITLE = {A general coherence result},
	JOURNAL = {J. Pure Appl. Algebra},
	FJOURNAL = {Journal of Pure and Applied Algebra},
	VOLUME = {57},
	YEAR = {1989},
	NUMBER = {2},
	PAGES = {165--173},
}

@phdthesis{Ro24,
	author = {Rom{\'a}n, M.},
	title = {Monoidal {Context} {Theory}},
	year = {2024},
	school ={Tallinn University of Technology},
	eprint = {2404.06192},
}

@article{SaUm04,
	author = {Saneblidze, S. and Umble, R.},
	title = {Diagonals on the permutahedra, multiplihedra and associahedra},
	fjournal = {Homology, Homotopy and Applications},
	journal = {Homology Homotopy Appl.},
	issn = {1532-0073},
	volume = {6},
	number = {1},
	pages = {363--411},
	year = {2004},
}

@misc{Sim98,
	author = {Simpson, C.},
	title = {Homotopy types of strict 3-groupoids},
	year = {1998},
	eprint = {math/9810059},
	note = {Preprint},
}

@article{St63,
	author = {Stasheff, J.},
	title = {Homotopy associativity of {{\(H\)}}-spaces. {I}, {II}},
	fjournal = {Transactions of the American Mathematical Society},
	journal = {Trans. Am. Math. Soc.},
	issn = {0002-9947},
	volume = {108},
	pages = {275--292, 293--312},
	year = {1963},
	language = {English},
}

@book{St70,
	author = {Stasheff, J.},
	title = {{{\(H\)}}-spaces from a homotopy point of view},
	fseries = {Lecture Notes in Mathematics},
	series = {Lect. Notes Math.},
	issn = {0075-8434},
	volume = {161},
	year = {1970},
	publisher = {Springer, Cham},
}

@article{JoSt91,
	author = {Joyal, A. and Street, R.},
	title = {The geometry of tensor calculus. {I}},
	fjournal = {Advances in Mathematics},
	journal = {Adv. Math.},
	issn = {0001-8708},
	volume = {88},
	number = {1},
	pages = {55--112},
	year = {1991},
}

@phdthesis{Ta51,
	author = {Tamari, D.},
	title = {Monoides pr{\'e}ordonn{\'e}s et cha{\^{\i}}nes de {Malcev}},
	year = {1951},
	school ={Université de Paris},
}

@article{Ya24,
	author = {Yadav, H.},
	title = {Frobenius monoidal functors from (co){Hopf} adjunctions},
	fjournal = {Proceedings of the American Mathematical Society},
	journal = {Proc. Am. Math. Soc.},
	issn = {0002-9939},
	volume = {152},
	number = {2},
	pages = {471--487},
	year = {2024},
}

\end{document}